\documentclass{amsart}
\usepackage{amsmath, amsfonts, latexsym, array, amssymb, amscd }
\usepackage[all]{xypic}
\usepackage[all]{xy}
\usepackage{graphicx}
\usepackage{color}
\usepackage{enumitem}
\usepackage{epstopdf}
\usepackage{hyperref}
\usepackage{mathtools}
\usepackage{mathrsfs}
\usepackage{parskip}
\usepackage{accents}
\usepackage[svgnames]{xcolor}
\usepackage{bbold}
\usepackage{caption}
\usepackage{subcaption}
\usepackage{tikz}
\usepackage{ieeetrantools}
\usepackage{centernot}
\usetikzlibrary{calc}

\numberwithin{equation}{section}
\numberwithin{figure}{section}

\newtheorem{theorem}{Theorem}[section]
\newtheorem*{theorem*}{Theorem}
\newtheorem{lemma}[theorem]{Lemma}

\newtheorem{proposition}[theorem]{Proposition}
\newtheorem*{proposition*}{Proposition}
\newtheorem{corollary}[theorem]{Corollary}
\newtheorem{thmx}{Theorem}

\theoremstyle{definition}
\newtheorem{definition}[theorem]{Definition}

\newtheorem{remark}[theorem]{Remark}

\usepackage{indentfirst}
\usepackage{ulem}

\DeclareMathOperator{\supp}{supp}

\DeclareMathOperator{\diam}{diam}

\newcommand{\PPP}{\mathcal{P}}

\newcommand{\CCC}{\mathcal{C}}

\newcommand{\eps}{\epsilon}

\newcommand{\NN}{\mathbb{N}}

 \newcommand{\bphi}{\bar \varphi }

\newcommand{\RR}{\mathbb{R}}
\newcommand{\ZZ}{\mathbb{Z}}
\newcommand{\abs}[1]{\left\lvert#1\right\rvert}
\newcommand{\dd}[2]{\frac{d#1}{d#2}}

\newcommand{\FFF}{\mathcal{F}}

\def\CAT{\operatorname{CAT}}
\newcommand{\bX}{\partial_\infty X}

\newcommand{\mo}{m^{0}}

\DeclareFontFamily{U}{bbold}{}
\DeclareFontShape{U}{bbold}{m}{n}
{  <5> <6> <7> <8> <9> gen * bbold
	<10> <10.95> bbold10
	<12> <14.4> bbold12
	<17.28> <20.74> <24.88> bbold17
}{}
\DeclareSymbolFont{bbold}{U}{bbold}{m}{n}
\DeclareMathSymbol{\bbid}{\mathord}{bbold}{'061}

\newcommand{\su}{W^{\mathrm{uu}}}
\newcommand{\wuu}{W^{\mathrm{uu}}}
\newcommand{\wu}{W^{\mathrm{u}}}
\newcommand{\wss}{W^{\mathrm{ss}}}
\newcommand{\ws}{W^{\mathrm{s}}}

\title[Gibbs measures for CAT(-1) spaces]
{Gibbs measures for geodesic flow on CAT(-1) spaces}
\author{Caleb~Dilsavor and Daniel~J.~ Thompson} 
\date{\today}
\thanks{This work is partially supported by NSF grants DMS-$1954463$ and DMS-$2349915$.}
\address{D.~J.~Thompson, Department of Mathematics, The Ohio State University, Columbus, OH 43210, \emph{E-mail address:} \tt{thompson@math.osu.edu}}
\address{C. Dilsavor, Department of Mathematics, Northwestern University, Evanston, IL, 60208,
\emph{E-mail address:} \tt{caleb.dilsavor@northwestern.edu}}
\begin{document}
\maketitle
\begin{abstract}
For a proper geodesically complete CAT(-1) space equipped with a discrete non-elementary action, and a
bounded continuous potential with the Bowen property, we construct weighted quasi-conformal Patterson
densities  and use them to build a Gibbs measure on the space of geodesic lines. Our construction yields a Gibbs measure with local product structure for any potential in this class, which includes bounded H\"older continuous potentials.
Furthermore, if the Gibbs measure is finite, then we prove that it is the unique equilibrium state. In
contrast to previous results in this direction, we do not require any condition that the potential must
take the same value on two geodesic lines which share a common segment.
\end{abstract}

\section{Introduction}
We develop the geometric approach to the theory of Gibbs measures and equilibrium states for $\CAT(-1)$
spaces for bounded continuous potentials. For manifolds with pinched negative curvature, and uniformly
locally H\"older potentials $\varphi$, a beautiful theory of weighted Patterson-Sullivan measures and
equilibrium states was developed by Paulin, Pollicott and Schapira \cite{PPS}. This built on seminal work by Otal and Peign\'e \cite{OP04}, and some parts of the theory were extended to
$\CAT(-1)$ spaces by Broise-Alamichel,  Parkkonen and Paulin \cite{BPP} with a focus on trees. 
There has been a fundamental obstacle to fully extending this theory to $\CAT(-1)$ spaces.
Namely, the presence of branching in $\CAT(-1)$ spaces 
means that the weights used in \cite{PPS} will usually not be defined.
We develop a coarse approach to the theory and find appropriate weights 
that we use to construct quasi-conformal Patterson-Sullivan measures. We then establish the equilibrium state theory. This is a challenge due to the need to develop techniques which apply in our coarse framework.

Our setting is a proper geodesically complete $\CAT(-1)$ space $(X, d)$ equipped with
a non-elementary discrete group of isometries $\Gamma < \mathrm{Isom}(X)$.
Let $GX$ be the space of isometric embeddings of $\RR$ into $X$, which we call \emph{geodesic lines}.
We equip $GX$ with a standard metric $d_{GX}$, see \S \ref{s.prelim}. 

Denote $X_0=\Gamma\backslash X$ and $GX_0 = \Gamma\backslash GX$.
The space $GX_0$ is the natural phase space for a geodesic flow $\left(g_{t}\right)_{t \in \mathbb{R}}$
given by precomposing
geodesic lines with translations.  If $X_0$ is a Riemannian manifold, then this flow $(g_t)_{t\in\RR}$
is naturally identified with the usual geodesic flow on the unit tangent bundle
$T^1X_0$. We study the thermodynamic formalism of this flow
	for potential functions $\varphi$ that are 
	bounded, continuous, and satisfy the Bowen regularity property. This class includes bounded
	H\"older potentials, see \S \ref{subsec.Bowen}. Let $\Omega X_0$ denote the non-wandering set of the geodesic flow on $GX_0$, and let $\Omega X$
be its lift to $GX$. Our first main result is the following. 
\begin{thmx} \label{maintheorem}
Suppose that $\Gamma$ has finite critical exponent. Let $\varphi: GX_0 \to \RR$ be a bounded continuous potential function 
whose lift to $GX$ satisfies the Bowen property.
Then there exists a Patterson-Sullivan construction of a flow-invariant
Radon measure $m_\varphi$ on $GX_0$ which is fully supported on $\Omega X_0$,
and whose lift to $GX$ is a quasi-product measure satisfying
the Gibbs property for $\varphi$  on $(GX, \Gamma)$.
If the measure $m_{\varphi}$ is conservative, then it is ergodic.
\end{thmx}
 Our next goal is to show that when $m_{\varphi}$  is finite, it can be characterized as the unique equilibrium state for $\varphi$.
A crucial technical ingredient for the geometric approach to this theory is to construct a measurable partition with desirable dynamical properties, which we call the Ledrappier-Ma\~n\'e-Otal-Peign\'e partition. Our next main theorem shows that this partition exists in our $\CAT(-1)$ setting. 

\begin{thmx}\label{thmx:op} 
	Suppose that $\Omega X$ has finite upper
	box dimension with respect to $d_{GX}$. Let $\nu$ be an ergodic probability
	measure on $GX_0$, and let $\tau > 0$ be chosen so that $\nu$ is ergodic with respect to 
	$g \coloneqq g_\tau$. Then there exists a measurable partition $\zeta$ such that
	$(g^{-1}\zeta)(c) \subseteq \zeta(c)$ for every $c \in GX_0$, $\lim_{n \to \infty}
	\diam((g^{-n})(c)) = 0$ for $\nu$-a.e.\ $c \in GX_0$, and
	$\zeta$ is $\nu$-subordinated to the strong unstable partition of $GX_0$. Furthermore, we have
	$h_\nu(g,\zeta) = h_\nu(g)$.
\end{thmx}
 We comment on the additional dimension hypothesis on $\Omega X$ that appears in Theorem \ref{thmx:op}.	The set  $\Omega X$ consists of geodesic lines whose endpoints are in the limit set 
	$\Lambda \subseteq \partial_\infty X$ of $\Gamma$,
	and it follows from \cite[Corollary 5.1.5]{thesis} and \cite[Lemma
	5.1.6]{thesis} that our dimension hypothesis is satisfied if and only if $\Lambda$
	has finite upper box dimension with respect to a visual metric. 
By the Bishop-Jones theorem \cite{FSU, nC21}, the critical exponent of $\Gamma$
is the Hausdorff dimension of the conical limit set 
$\Lambda_c \subseteq \Lambda$. Our hypothesis that  $\Omega X$  has finite upper box dimension is thus a mild strengthening of our previous assumption in Theorem \ref{maintheorem} that the critical exponent of $\Gamma$ is finite. This can be verified from mild conditions on the metric space $(X, d)$. Bonk and Schramm proved in \cite[Theorem 9.2]{BS} that if $(X,d)$ has `bounded growth at some scale' (that is, there exist $0 <r < R$ and $N \in \NN$ such that any ball of radius $R$ in $X$ can be covered using $N$ balls of radius $r$), then the visual
boundary $\partial_\infty X$ has finite Assouad dimension, and hence finite upper box dimension.
In particular, our hypothesis is satisfied if $(X,d)$ has finite Assouad dimension. Using Theorem \ref{thmx:op} as a key ingredient, our next main result is as follows.
\begin{thmx} \label{thmx:existunique}
Suppose that $\Omega X$ has finite upper box dimension with respect to $d_{GX}$.  Let $\varphi:GX_0 \to \RR$ be a bounded continuous potential function 
whose lift to $GX$ satisfies the Bowen property and let $m_\varphi$ be a measure provided by Theorem \ref{maintheorem}.
If $m_\varphi$ is finite, then after normalizing it to be a probability measure,
it is the unique equilibrium state for $\varphi$.
\end{thmx}

\subsection{Previous results} For Riemannian manifolds with pinched negative curvature, the
characterization of the Bowen-Margulis measure (i.e.\ the measure $m_\varphi$ constructed in Theorem \ref{maintheorem} in the case $\varphi = 0$) as the unique measure of  maximal entropy when it is finite
is given in seminal work of Otal and Peign\'e \cite{OP04}, and with a refined proof in Ledrappier
\cite{L13}. The extension to equilibrium states 
for pinched negatively curved manifolds with locally H\"older potentials
was developed in the groundbreaking monograph of Paulin-Pollicott-Schapira \cite{PPS}.  The results in \cite{OP04, PPS} additionally require that the first derivatives of sectional curvatures are bounded. We discuss this regularity issue in \S \ref{Holdercomment}. 

In the $\CAT(-1)$ setting, the construction of the Bowen-Margulis measure
was developed for torsion-free $\Gamma$ in the seminal monograph by Roblin
\cite{Roblin}.  Beyond this, progress has only been made for a restricted class of potentials.
As defined in \cite{CM07}, we say a continuous potential $\varphi : GX_0 \to \RR$ is \emph{tempered}
if, lifted to $GX$, it depends only on the germ of the geodesic line at $0$.  Gibbs measures were constructed for tempered potentials with a technical regularity condition called the
H\"older-Control property by Broise-Alamichel, Parkkonen and Paulin \cite{BPP}. The H\"older-Control property follows from H\"older continuity in the Riemannian setting,
but it is difficult to verify in the $\CAT(-1)$ setting. In the cocompact torsion-free setting, the
construction of the conformal densities was sketched by Connell and Muchnik in \cite{CM07} for tempered
H\"older potentials.  There are no prior results on the  Patterson-Sullivan construction for non-tempered potentials, even when $\Gamma$ is cocompact.

For convex-cocompact and torsion-free actions, the existence and uniqueness of equilibrium states for
potentials with the Bowen property is known by work of Constantine, Lafont, and the second-named author \cite{CLT19, CLT2}.  The theory of equilibrium measures in this setting is similar to the well-established theory
for Anosov and Axiom A flows \cite{BR75, FH19}. However, the proofs must be
accomplished without using any smooth structure, and these results have been established only recently.
For a $\CAT(-1)$ space equipped with an action that is not convex-cocompact,
even the characterization of Roblin's Bowen-Margulis measure as the unique measure of maximal entropy has not previously been established in the literature. To the best of our knowledge, the only other previously known result on equilibrium states for $\CAT(-1)$ spaces beyond the Riemannian setting is \cite[Theorem 5.12]{BPP}, which applies only for tempered potentials on metric trees.
\subsection{Proof ideas and challenges} \label{s.introproofideas}
We now discuss the fundamental issues which have previously limited progress in this setting. We develop the geometric approach of Otal and Peign\'e \cite{OP04} and Paulin, Pollicott and Schapira \cite{PPS}. The dynamical approaches in
\cite{CLT19, CLT2} using Bowen's specification approach or symbolic dynamics currently
only apply when the non-wandering set is compact.  
	The Paulin-Pollicott-Schapira argument is based on weighted Patterson densities,
	which are constructed using weight terms of the form $\int_{x}^{y}\varphi$. In the manifold
	case, these terms are defined by integrating the lift of $\varphi$ along the unique
	orbit segment of the geodesic flow corresponding to the geodesic segment from 
	$x \in X$ to $y \in X$.
	In the $\CAT(-1)$ setting, however, 
	geodesic segments do not necessarily extend uniquely to geodesic lines, and thus there may be many orbit segments of $GX$ corresponding to the same geodesic segment in $X$.

In \cite[\S 3.1]{PPS}, the authors remark that this makes it unclear how to generalize the construction
to the $\CAT(-1)$ case, because the terms $\int_{x}^{y}\varphi$ are no longer well-defined.
One route around this issue is to simply \emph{require} that they are well-defined,
which motivates the restriction to the class
of tempered potentials in \cite{CM07, BPP, GT20}. The class of tempered potentials contains  natural examples,
including functions that depend only on the footpoint. However, it is clear that being tempered is a
significant restriction on the potential when the space is not a manifold.
In the context of metric trees, the
tempered assumption has a particularly strong characterization: in the symbolic model of the tree
considered in \cite{BPP}, these potentials only depend on the first coordinate.
A key challenge of our
approach is to remove the restriction that the potentials are tempered,
	and this requires substantial novelties in our analysis.

\subsubsection{On the proof of Theorem \ref{maintheorem}} 
To develop the geometric approach of \cite{PPS}, we must find a substitute for the integrals $\int_x^y \varphi$, which are not well-defined 
without the assumption that $\varphi$ is tempered. We replace $\int_x^y \varphi$ with a function $\bphi (x,y)$ that takes a supremum of the ergodic integrals of the lift of $\varphi$ over all possible orbit segments of $GX$ that correspond to the geodesic segment from $x$ to $y$, see \eqref{eq:phibar}. The cost is that this requires us to work in a coarse
setting -- the expression $\bphi (x,y)$ is not continuous in $x,y$,
and all of the objects under consideration are now defined only up to a bounded error.
Particularly, in place of the Gibbs cocycle defined in \cite{PPS},
we must work with a Gibbs quasi-cocycle on the boundary.
A Patterson-Sullivan construction similar to
Coornaert's unweighted construction  for Gromov hyperbolic spaces \cite{Co93} then
produces a quasi-conformal measure for this quasi-cocycle.
Our construction takes inspiration from work by Cantrell and Tanaka \cite{CT1, CT2} in the setting of weighted hyperbolic groups.

\subsubsection{On the proof of Theorem \ref{thmx:op}} \label{Holdercomment} 
Our construction of the Ledrappier-Ma\~n\'e-Otal-Peign\'e partition follows Ledrappier \cite{L13} in the pinched negative curvature manifold case, with part of the
argument going back further to Ma\~n\'e \cite{Mane}. While the essential ideas  are already in the literature, particularly in Ledrappier's proof, we give a complete account to clear up some subtle points that arise in \cite{OP04, PPS}.  Furthermore, we allow $\Gamma$ to have torsion.

Otal and Peign\'e's original proof of Theorem \ref{thmx:op} in \cite{OP04} in the pinched negative curvature manifold case
 requires H\"older continuity of the strong stable and
strong unstable distributions.
This may fail in the class of pinched negatively curved
manifolds unless one additionally asks for bounded 
first derivatives of curvature. In \cite{PPS},
Paulin, Pollicott and Schapira require this additional assumption so that they can use the partition
constructed by Otal and Peign\'e.
Lack of regularity of the dynamical distributions has thus been widely understood to be an
obstacle in developing the thermodynamic formalism for pinched curvature manifolds without bounded
first derivatives of curvature, let alone for general $\CAT(-1)$ spaces. Ledrappier's version
of the argument \cite{L13} does not require this regularity (see Remark \ref{rem:balls}).

	A main step in the proof of Theorem \ref{thmx:op} 
	is to construct an auxiliary partition of a set
	$V$ that has full $\nu$-measure inside some open subset of $GX_0$.
	In \cite[Proposition 6.4]{L13},
	elements of this partition are chosen based on
	the return times of recurrent elements to the open subset. The fact that these return times are usually unbounded means that the partition must be
infinite; the Ma\~n\'e argument ensures
that it can be chosen to have finite entropy. The unboundedness of the return times appears to be overlooked in the proof of
in \cite[Fait 5]{OP04}, where a simplified version of the auxiliary partition is constructed with only two elements. This issue is resolved by Ledrappier's version of the construction in \cite[Proposition 6.4]{L13}.

\subsubsection{On the proof of Theorem \ref{thmx:existunique}}
Given Theorem \ref{thmx:op}, the proof of Theorem \ref{thmx:existunique} starts with the same strategy as \cite{OP04,PPS} to characterize the Gibbs measure as an equilibrium state.  We review the strategy and differences with the Riemannian case at the start of \S \ref{s.thermodynamic}. A key difference is that the uniqueness argument in \cite{OP04,PPS} makes use of the equality case of Jensen's inequality, which we can never arrive at in the presence of 
the error terms arising from the coarse construction. We give a novel argument based on Kullbeck-Leibler divergences which is robust enough to apply in this setting.

\subsection{Organization of paper}
In \S \ref{s.prelim}, we set up our preliminaries. In \S
\ref{s.weightedPS} we construct quasi-conformal measures and the
measure $m_\varphi$. In \S \ref{s.gibbs}, we prove that this measure satisfies the Gibbs property,
which completes Theorem \ref{maintheorem}. In  \S \ref{sec:op}, we prove Theorem \ref{thmx:op}. In \S \ref{s.thermodynamic}, we construct
measures on strong unstable leaves with good scaling properties (up to bounded error) and use them to show that $m_\varphi$ is
the unique equilibrium state when it is finite, proving Theorem \ref{thmx:existunique}.

\section{Preliminaries} \label{s.prelim}
\subsection{$\CAT(-1)$ spaces and geodesic flow}\label{subsec:CAT (-1)}
We set up our notation and collect some facts we will use. References
include \cite{Roblin, wB95, BH99}. Let $(X, d)$ be a proper geodesically complete $\CAT(-1)$
space, and let $\Gamma < \mathrm{Isom}(X)$ be a non-elementary discrete group of isometries of $X$. Let
$X_0=\Gamma\backslash X$. When $\Gamma$ has no torsion, the space $X_0$ is a proper connected geodesically
complete locally $\CAT(-1)$ space. Every such space arises this way, and $\Gamma$ is the fundamental group
acting as a group of isometries on the universal cover $X$. More generally, when $\Gamma$ has torsion,
$X_0$ is a locally $\CAT(-1)$ good orbispace
(see \cite[\S 11.3]{GH} for a definition). We define 
\[
	GX \coloneqq \{c : \RR \to X \colon c \text{ is an isometry onto its image}\}.
\] 
That is, $GX$ is the space of geodesic lines in $X$. We equip $GX$ with the metric
\[ 
	d_{G  X}(  c,  c') =  \int_{-\infty}^\infty d(  c(s),  c'(s)) e^{-2|s|}ds
.\]
The group $\Gamma$ acts naturally on $GX$ by isometries. We let $GX_0 = \Gamma \backslash GX$,
and for $c_1, c_2 \in GX_0$, we define
\[ 
	d_{GX_0}(c_1,c_2) = \inf_{\tilde{c}_1, \tilde{c}_2} d_{G X}(\tilde{c}_1,\tilde{c}_2),
\]
where the infimum is taken over all lifts $\tilde{c}_1, \tilde{c}_2$ of $c_1, c_2$.

For $t \in \RR$, we define the geodesic flow $g_t : GX \to GX$ at time $t$ by 
\[
	(g_tc)(s) = g(t+s)
.\]
It is easy to check from the definition of $d_{GX}$ that the flow is unit speed, and that
\begin{equation}\label{eq:bddgrowth}
	e^{-2\abs{t}}d_{GX}(c,c') \leq d_{GX}(g_tc,g_tc') \leq e^{2\abs{t}}d_{GX}(c,c'), \quad t\in\RR.
\end{equation}
Since $g_t$ is $\Gamma$-equivariant, it descends to a well-defined flow on $GX_0$.

We let $\pi_{GX} : GX \to GX_0$ and $\pi_X : X \to X_0$ denote the quotient maps. We define the footprint map $\pi_{\mathrm{fp}}: GX \to X$ 
and the flip map $\iota : GX \to GX$ by
\[
\pi_{\mathrm{fp}}(c) = c(0), \qquad \iota(c)(t) = c(-t)
.\] 
By $\Gamma$-equivariance, we obtain maps $\pi_{\mathrm{fp}} : GX_0 \to X_0$ and
$\iota : GX_0 \to GX_0$.

Let $\partial_\infty X$ denote the boundary at infinity of $X$, defined to be
the collection of equivalence classes of geodesic rays, where two rays are equivalent if they stay within
bounded distance of each other. We equip $\partial_\infty X$ with the cone topology. In this topology, $\partial_\infty X$
is compact. Given $c \in GX$, we write $c(+\infty)$ (resp. $c(-\infty)$) for the point in $\partial_\infty
X$ determined by the positive (resp. negative) geodesic ray defined by $c$. The union $\bar X = X\cup
\partial_\infty X$ is then a compactification of $X$ such that
$\lim_{t\to\pm\infty}c(t) = c(\pm\infty)$ for any $c \in GX$.

If $x, y \in X$, we write $[x,y] \subseteq X$
for the geodesic segment from $x$ to $y$. For $x \in X$ and $\xi \in \partial_{\infty}X$, we write $[x,\xi) \subseteq X$
	for the geodesic ray from $x$ to $\xi$.
	These sets always exist and are uniquely defined because $X$ is a $\CAT(-1)$ space.

Let $\partial_\infty^2 X \coloneqq (\partial_\infty X \times \partial_\infty X) \setminus \Delta$,
where $\Delta$ is the diagonal. For a $\CAT(-1)$ space $X$,
the space of geodesic lines $GX$ can be identified with $\partial_\infty^2 X \times \mathbb{R}$ using a
Hopf parametrization. The Hopf parametrization is not canonical since the $\RR$ parameter is chosen based
on an arbitrary reference point in $X$ to determine the parametrizations of each element of $GX$.
Under a Hopf parametrization, the topologies induced on $GX$ by $d_{GX}$ and on $\partial_\infty^2 X \times \mathbb{R}$ 
using the cone topology of $\partial_\infty X$ agree. In Hopf coordinates, the projection to the $\partial_\infty^2 X$ coordinate is the map
$c \mapsto (c(-\infty),c(+\infty))$, and translation on the $\RR$ coordinate corresponds to the geodesic flow on $GX$.

The action of $\Gamma$ preserves equivalence classes of geodesic rays, and thus we can consider $\Gamma$ acting on $\partial_\infty X$.
Let $\Lambda \subseteq \partial_\infty X$ denote the limit set of $\Gamma$,
defined to be the set of limit points of $\{a x\}_{a \in \Gamma}$ in $\partial_\infty  X$.
This definition is independent of $x$. The set $\Lambda$ is compact and $\Gamma$-invariant. Since $\Gamma$ is discrete and non-elementary, then $\Lambda$ is uncountable,
and $\Gamma$ acts minimally on $\Lambda$. We let $\Omega X$ denote the collection of geodesic lines $c \in GX$ such that
$c(-\infty), c(+\infty) \in \Lambda$,
and we let $\Omega X_0$  be 
the quotient of this set by $\Gamma$.
The set  $\Omega X_0$ is the non-wandering set of the geodesic flow on $GX_0$. Since $\Gamma$ acts minimally on $\Lambda$, the 
geodesic flow is transitive on $\Omega X_0$.

\subsection{Busemann functions and partitions into dynamical sets} For
$\xi \in \bX$, $x, y \in X$, the Busemann functions are given by
\[
\beta_\xi(x,y) = \lim_{z \to \xi} (d(x, z)-d(y,z)).
\]
For a basepoint $p \in X$, and for $x,y \in X$,
the Gromov product is defined by
\[
	(x|y)_p = \frac{1}{2}(d(x,p)+d(p,y)-d(x,y))
.\] 
For $(\xi, \eta) \in \partial_\infty^2 X$, the definition extends by setting
\[
	(\xi| \eta)_p = \lim_{x\to\xi,y\to\eta}\frac{1}{2}(d(x,p)+d(p,y)-d(x,y)).
\]

On $GX$, we define the dynamical sets $W^\epsilon(c)$ for $\epsilon \in \{\mathrm{uu},\mathrm{u},
\mathrm{ss},\mathrm{s}\}$ by setting
\[
	W^\mathrm{uu}(c) = \{c' \in GX \colon \lim_{t\to-\infty}d_{GX}(g_tc,g_tc') = 0\},
	\quad W^\mathrm{u}(c) = \bigcup_{t \in \RR}g_t W^{\mathrm{uu}}(c)
,\] 
\[
	W^\mathrm{ss}(c) = \{c' \in GX \colon \lim_{t\to +\infty}d_{GX}(g_tc,g_tc') = 0\},
	\quad W^\mathrm{s}(c) = \bigcup_{t \in \RR}g_t W^{\mathrm{ss}}(c)
.\] 
We call $\wuu(c)$, $\wu(c)$, $\wss(c)$, and $\ws(c)$ the  \emph{strong unstable}, \emph{weak unstable},
\emph{strong stable}, and \emph{weak stable set},
respectively, of $c \in GX$. Each of these families gives a partition of $GX$ which we denote by $\mathcal{W}^\epsilon$ for
$\epsilon \in \{\mathrm{uu},\mathrm{u},\mathrm{ss},\mathrm{s}\}$.

The dynamical partitions can be defined equivalently using the boundary at infinity as follows.
We have $c' \in W^\mathrm{u}(c)$ if and only if $d_{GX}(g_tc,g_tc')$
is bounded for $t \leq 0$, if and only if $d(c(t),c'(t))$ is bounded for $t \leq 0$,
and similarly for $W^\mathrm{s}(c)$. In other words, 
\[
	W^\mathrm{u}(c) = \{c' \in GX \colon c'(-\infty) = c(-\infty)\}, \quad
	W^\mathrm{s}(c) = \{c' \in GX \colon c'(+\infty) = c(+\infty)\}
.\] 
Within each $W^\mathrm{u}(c)$ and each $W^\mathrm{s}(c)$, we can use the Busemann function to specify the
strong stable and unstable sets:
\begin{align*}
	W^\mathrm{uu}(c) &= \{c' \in W^\mathrm{u}(c) \colon \beta_{c(-\infty)}(c'(0),c(0)) = 0\},\\
	W^\mathrm{ss}(c) &= \{c' \in W^\mathrm{s}(c) \colon \beta_{c(+\infty)}(c'(0),c(0)) = 0\}.
\end{align*}
We define the Hamenst\"adt metric $d^+$ on $\su(c)$ by
\[
	d^+(c_1,c_2) = e^{\lim_{t\to +\infty}\tfrac{1}{2}d(c_1(t),c_2(t))-t}
\] 
for $c_1,c_2 \in W^\mathrm{uu}(c)$. It follows from the definition that $d^+(g_tc_1,g_tc_2) = e^td^+(c_1,c_2)$.
Similarly, we define the Hamenst\"adt metric $d^-$ on $W^\mathrm{ss}(c)$ by 
\[
	d^-(c_1,c_2) = e^{\lim_{t\to+\infty}\tfrac{1}{2}d(c_1(-t),c_2(-t))-t}
\]
for $c_1,c_2 \in W^\mathrm{ss}(c)$ and we have $d^-(g_tc_1,g_tc_2) = e^{-t}d^-(c_1,c_2)$.

It is immediate from the definition that $W^\epsilon(ac) = aW^\epsilon(c)$ when $a \in \Gamma$ and $c
\in GX$. Using this we define dynamical sets
$W^\epsilon(c) \subset GX_0$ for $c \in GX_0$ by pushing them down from $GX$.
We  have 
\[
	W^\mathrm{uu}(c) \subseteq \{c' \in GX_0 \colon \lim_{t\to-\infty}d_{GX_0}(g_tc,g_tc') = 0\},
\] 
and we emphasize that this inclusion may be strict in the absence of a lower bound on the injectivity radius of $\pi_{GX}$.

Since $d^+$ is $\Gamma$-equivariant,  it descends to a well-defined
distance on $W^\mathrm{uu}(c)$ for $c \in GX_0$ given by
\[
	d^+(c_1,c_2) \coloneqq \inf_{\tilde{c}_1,\tilde{c}_2}d^+(\tilde{c}_1,\tilde{c}_2)
,\] 
where the infimum is taken over all lifts $\tilde{c}_1, \tilde{c}_2$ such that $\tilde{c}_1 \in
W^\mathrm{uu}(\tilde{c}_2)$. We define $d^-$ on $W^\mathrm{ss}(c)$ similarly for $c \in GX_0$.   
We write  $B^+$ for balls inside of elements of $\mathcal{W}^\mathrm{uu}$ with respect to $d^+$,
and we write $B^-$ for balls inside elements of $\mathcal{W}^\mathrm{ss}$ with respect to $d^-$.
It will be clear whether these balls are subsets of $GX$ or $GX_0$ from their centers.
The following  is proved in \cite[Lemma 2.4]{BPP}.

\begin{lemma}\label{lem:hamenstadt}
There exists $C \geq 0$ such that for any $c\in GX$ and $c_1,c_2 \in W^\mathrm{uu}(c) \subset GX$, we have
\[
	d_{GX}(c_1,c_2) \leq C d^+(c_1,c_2).
\] 
\end{lemma}
For $c \in GX_0$ and $c_1,c_2 \in \wuu(c) \subset GX_0$, it follows from Lemma \ref{lem:hamenstadt} that 
\begin{equation}\label{eq.downstairshamenstadt}
d_{GX_0}(c_1,c_2) \leq C d^+(c_1,c_2).
\end{equation}

\subsection{Thermodynamic Formalism and the Gibbs property} \label{subsec:compact}
Consider a continuous flow $\mathcal F = (f_t)_{t\in\RR}$ on a complete separable metric space $(Z,d_Z)$. 
For an $\mathcal F$-invariant Borel probability measure $\nu$,
its measure-theoretic entropy $h_\nu(\FFF)$ with respect to $\FFF$ is defined to be its entropy
with respect to the time-one map $h_\nu(f_1)$.
By Abramov's formula, we have $h_\nu (f_\tau) = |\tau| h_\nu(\FFF)$. For a continuous function $\varphi: Z \to \mathbb R$ (called the \emph{potential}),
we define the \emph{(variational) pressure} to be
\begin{equation} \label{variationalpressure}
	P(\varphi) = \sup_\nu \left\{h_\nu(\FFF) + \int\varphi\,d\nu  \right\},
\end{equation}
where the supremum is taken over all $\mathcal F$-invariant Borel probability measures $\nu$ such that
$\int\varphi^-\,d\nu < \infty$, where $\varphi^- = \max\{0,-\varphi\}$.
Since we consider only
bounded potentials, the extra restriction on $\int \varphi^- \, d\nu$ is redundant in this paper.
For each such $\nu$, the quantity
$h_\nu(\mathcal{F}) + \int\varphi\,d\nu$ is called its \emph{free energy} with respect to $\varphi$.

An \emph{equilibrium state} for $\varphi$ is a probability measure whose free energy with respect to $\varphi$
is equal to $P(\varphi)$. An equilibrium state for the constant function $\varphi = 0$ is called a \emph{measure of maximal entropy}. 

 Topological and combinatorial characterizations of the variational pressure
are known in the compact case \cite{Wa}, for countable-state symbolic dynamics \cite{oS99}, and for
geodesic flows on pinched negatively curved Riemannian manifolds \cite{PPS}.
These characterizations are not currently available for geodesic flow on $GX_0$, and therefore
we will only work with the variational pressure.

For each $t > 0$, we define a metric $d_t(z,z') = \max_{s \in [0,t]}d_{Z}(f_sz,f_sz')$ on $Z$. A  \emph{Bowen ball} is a set of the form $B_t(z,r; Z):=\{ z' \in Z : d_t(z,z') < r\}$. When the phase space $Z$ is clear from context, we just write a Bowen ball as $B_t(z,r)$. For a potential $\varphi:Z \to \RR$, we often denote the ergodic integrals by
\begin{equation} \label{ergodicintegral}
\Phi(z,t) := \int_0^t \varphi(f_\tau z)\,d\tau,
\end{equation}
or more generally for $s<t$, we write
\begin{equation} \label{ergodicintegral2} 
\Phi(z,[s,t]) := \int_s^t \varphi(f_\tau z)\,d\tau.
\end{equation}

The Gibbs property plays a crucial role in thermodynamic formalism. See \cite{BR75, FH19} for the
definitions in the case of a compact phase space. In the non-compact case, the constants in the Gibbs
property must be  allowed to depend on a reference set $F$,  or the property is too restrictive.
The basic general definition is as follows.

\begin{definition} \label{def:Gibbspropertygeneral}
	Let $m$ be a flow-invariant probability measure on $Z$, and let ${F \subset Z}$, $r > 0$, and
	$\sigma \in \RR$. We say that $m$ satisfies the \emph{Gibbs property on $Z$ for $\varphi$ at scale
$r$ with respect to the reference set $F$} if there is a constant $k_{F, r} \geq 1$ such
that, whenever $z \in Z$ and $t \geq 0$ are such that $z \in F$ and $f_t z \in F$, then we have 
\[
	k_{F,r}^{-1} e^{-t\sigma + \Phi(z,t)} \leq m(B_t(z,r; Z)) \leq k_{F,r} e^{-t\sigma + \Phi(z,t)}
.\]
The constant $\sigma$ is called the \emph{exponent} of the Gibbs property.
\end{definition}
When $Z$ is compact, we recover 
the standard definition of the Gibbs property from Definition \ref{def:Gibbspropertygeneral}
by setting $F = Z$.
For a countable-state shift of finite type,
the reference sets are taken to be the collection of cylinder sets \cite[Appendix A]{BPP}. In the general setting that $Z$ is a complete metric space, the collection of reference sets is chosen depending on the context, and candidates include the closed metric balls, or the closed bounded sets with non-empty interior.

Let $(Y,d_Y)$ be a proper metric space. Suppose that $\Gamma < \mathrm{Isom}(Y)$ is discrete, $Y_0= \Gamma\backslash Y$, $(f_t)_{t\in\RR}$ is a
$\Gamma$-equivariant flow on $Y$, and $\varphi : Y \to \RR$ is a $\Gamma$-invariant potential.
In this setting, instead of considering the Gibbs property on  $F \subset Y_0$ as in Definition \ref{def:Gibbspropertygeneral},  it is natural to consider a version of the Gibbs property
`upstairs' on the pair $(Y,\Gamma)$, as follows.
\begin{definition} \label{def:Gibbspropertyupstairs}
Let $m$ be a flow-invariant, $\Gamma$-invariant probability measure on $Y$, and
let $F \subset Y$, $r > 0$, and $\sigma \in \RR$.
We say that $m$ satisfies the \emph{Gibbs property on $(Y, \Gamma)$ for $\varphi$ 
at scale $r$ with respect to the reference set $F$} if
there is a constant $k_{F,r}\geq 1$ such
that, whenever $a \in \Gamma$, $y \in Y$, and $t \geq 0$
are such that $y \in F$ and $f_ty \in a F$, then we have
\[
	k_{F,r}^{-1} e^{-t\sigma + \Phi(y,t)}
	\leq m(B_t(y,r; Y))
	\leq k_{F,r} e^{-t\sigma + \Phi(y,t)}
.\] 
The constant $\sigma$ is called the \emph{exponent} of the Gibbs property.
\end{definition} 
 If $\Gamma$ acts freely on $Y$ with a lower bound on the injectivity radius, then the measure $m$, the reference sets, and the Bowen balls in $Y$ at small scales can be pushed down to $Y_0$ and it is clear that the Gibbs property on $Y_0$ and $(Y,\Gamma)$ at small scales coincide.
However, without these assumptions on $\Gamma$, Bowen balls in $Y$ no longer project down to 
Bowen balls in $Y_0$ and the two properties are a priori different. We consider the Gibbs property on $(GX, \Gamma)$, rather than on $GX_0$.  We follow \cite[\S3.8]{PPS} and take our reference sets to be the compact subsets.

\begin{definition} \label{def:Gibbspropertyupstairs2}
We say that $m$ has the \emph{Gibbs property on $(Y, \Gamma)$ for $\varphi$ with exponent $\sigma \in
\RR$} if, for every compact set $F \subseteq Y$, there exists a scale $R > 0$
such that, for any $r \geq R$, the measure $m$ has the Gibbs property on $(Y,\Gamma)$ for $\varphi$ 
with exponent $\sigma$ at scale $r$ with respect to $F$.
\end{definition}

Definition \ref{def:Gibbspropertyupstairs2}  is the version of the Gibbs property on $(GX, \Gamma)$ that we use in our theorems. The scale $R>0$ depends on the reference set and might be large\footnote{If $m$ is a flow-invariant measure with the Gibbs property on $(GX, \Gamma)$ at \emph{all} small scales $R>0$, or at a single scale independent of the compact reference set, it follows that $m$ is fully supported on $GX$. This would be too restrictive, since we construct measures supported on $\Omega X$.}.

\subsubsection{Thermodynamic Formalism for geodesic flows on non-compact manifolds} 
We recall some results in the special case of a complete connected pinched negative curvature Riemannian manifold $M$.
See \cite{PPS} for more background. Let $\varphi$ be a potential function on $T^1M$, and also denote by $\varphi$ its lift to
$T^1\widetilde{M}$. In \cite{PPS}, the potential $\varphi$ is assumed
to be uniformly locally H\"older as a function on
$T^1\widetilde{M}$. This allows $\varphi$ to be unbounded,
but implies that it has at most linear growth. For $x, y \in \widetilde{M}$, we recall that the expression $\int_{x}^{y}\varphi$ is well defined.
For $x,y \in
\widetilde{M}$ and $c> 0$,
the critical exponent $\delta_\varphi$ is given by
\begin{equation} \label{PPSweightedPoincare}
\delta_\varphi=\limsup_{t\to\infty} \frac{1}{t} \log \sum_{\gamma} e^{\int_x^{\gamma y} \varphi},
\end{equation}
where the sum is over $\gamma \in \Gamma$ with $t-c\leq d(x, \gamma y) \leq t$.
The Gurevic pressure $P^G(\varphi)$ can be defined as the exponential growth rate of closed
geodesics weighted by $\varphi$ which intersect a fixed relatively compact reference set.

It is proved in \cite{PPS} that these two quantities coincide with the variational pressure.
Therefore, in this setting an equilibrium state is an 
invariant probability measure $m$ such that $\int \varphi^- \,dm < \infty$ and 
$h(m) + \int \varphi\, d m = \delta_\varphi = P^G(\varphi) = P(\varphi)$.
The following is proved in \cite{PPS}.
It is the  benchmark result in this area, and our main results extend most of the content of this statement to our setting. 
\begin{theorem}[Paulin, Pollicott and Schapira] \label{PPSVP}
For a  complete connected
pinched negatively curved Riemannian manifold with bounded first derivatives of curvature, and a
uniformly locally H\"older potential function $\varphi$, if $\delta_\varphi< \infty$, then there exists a
Gibbs measure $m_\varphi$ of exponent $\delta_\varphi$.
This measure is a Radon measure and is invariant. If $m_\varphi$ is infinite, there is no equilibrium state.  If $m_\varphi$ is finite, then normalizing it to make it a probability measure, it is the unique equilibrium state. 
\end{theorem}
The definition of a Gibbs measure in Paulin, Pollicott and Schapira is the same as ours. That is, in the above statement, a Gibbs measure is a measure satisfying the Gibbs property on $(T^1 \widetilde M, \Gamma)$ in the sense of Definition \ref{def:Gibbspropertyupstairs2}.  
\subsection{Facts from coarse geometry}
We introduce some terminology and results from coarse geometry with applications to $\CAT(-1)$ spaces.
\begin{definition}
Let $(Y_1,d_{Y_1})$, $(Y_2,d_{Y_2})$ be two metric spaces.
We say that a map $f : Y_1 \to Y_2$ is a \emph{quasi-isometric embedding} if
there exist constants $\kappa \geq 1$, $\epsilon \geq 0$ such that
\[
	\kappa^{-1}d_{Y_1}(y,y') - \epsilon \leq d_{Y_2}(f(y),f(y')) \leq \kappa d_{Y_1}(y,y') + \epsilon
\] 
for all $y,y' \in Y_1$. If additionally
there is some $C \geq 0$ such that every element $y \in Y_2$ is distance at
most $C$ from the image of $f$, then we say that $f$ is a \emph{quasi-isometry}.
\end{definition}

\begin{definition}
Let $(Y,d_Y)$ be a metric space.
We say that a map $\gamma : I \to Y$, where $I$ is an interval in $\RR$,
is a \emph{quasi-geodesic} if it is a quasi-isometric embedding. If we want to specify the constants
$\kappa$, $\epsilon$, then we say that $\gamma$ is a $(\kappa,\epsilon)$-quasi-geodesic.
\end{definition}

\begin{definition}
Let $(Y,d_Y)$ be a metric space. Then we say that a map $\gamma : I \to Y$, 
where $I \subseteq \RR$ is an interval, is a $(\kappa,\epsilon,L)$-local quasi-geodesic if 
$\gamma$ is a $(\kappa,\epsilon)$-quasi-geodesic when
restricted to any subinterval $J \subseteq I$ of length at most $L$.
\end{definition}

We have the following basic property of the footprint map of a $\CAT(-1)$ space.
\begin{lemma}\label{lem:footprint}
Let $(X,d)$ be a geodesically complete $\CAT(-1)$ space. Then $\pi_{\mathrm{fp}} : (GX, d_{GX}) \to (X,d)$ is a quasi-isometry.
\end{lemma}
\begin{proof}
Since $(X,d)$ is geodesically complete, the map $\pi_{\mathrm{fp}}$ is surjective. Given $c,c' \in GX$, we have $d(c(0),c'(0)) \leq 2d_{GX}(c,c')$
by \cite[Lemma 2.8]{CLT19}. On the other hand, we have
\begin{align*}
	d_{GX}(c,c') &= \int_{-\infty}^{\infty}d(c(t),c'(t))e^{-2\abs{t}}\,dt\\
	&\leq 2\int_{0}^{\infty}(d(c(0),c'(0))+2t)e^{-2t}\,dt \\
	&= \left(2\int_{0}^{\infty}e^{-2t}\,dt\right)d(c(0),c'(0))
	+\left(2\int_{0}^{\infty}2te^{-2t}\,dt\right). \qedhere
\end{align*}
\end{proof}

By \cite[Proposition III.H.1.2]{BH99},
any $\CAT(-1)$ space is Gromov hyperbolic. Therefore, we have the following standard results.

\begin{lemma}[Stability of quasi-geodesic segments, {\cite[Th{\'e}or{\`e}me 3.1.2]{CDP}}]
\label{lem:finitestability}
Suppose that $(X,d)$ is a $\CAT(-1)$ space. Then for any $\kappa \geq 1$, $\epsilon \geq 0$,
and $R \geq0$, there exists $C \geq 0$ such that, if $\gamma,\gamma'$ are two 
$(\kappa,\epsilon)$-quasi-geodesics in $X$ such that 
$d(\gamma(t),\gamma'(t')) \leq R$ and $d(\gamma(s),\gamma'(s')) \leq R$, where $t \leq s$ and $t'\leq s'$,
then $\gamma([t,s])$ is contained in the $C$-neighborhood of $\gamma'([t',s'])$.
\end{lemma}

\begin{lemma}[Stability of quasi-geodesic lines, {\cite[Th{\'e}or{\`e}me 3.3.1]{CDP}}]
\label{lem:infinitestability}
Suppose that $(X,d)$ is a $\CAT(-1)$ space. Then for any $\kappa \geq 1$, $\epsilon \geq 0$,
there exists $C \geq 0$ such that, if $\gamma,\gamma' : \RR \to X$ are two 
$(\kappa,\epsilon)$-quasi-geodesic lines in $X$ joining the same two points on $\partial_\infty X$,
then $\gamma(\RR)$ is contained in the $C$-neighborhood of $\gamma'(\RR)$.
\end{lemma}

\begin{lemma}[Local quasi-geodesics are quasi-geodesics, {\cite[Th{\'e}or{\`e}me 3.1.4]{CDP}}]
	\label{lem:localquasi}
	Suppose that $(X,d)$ is a $\CAT(-1)$ space.
	Then for any $\kappa \geq 1$, $\epsilon \geq 0$, there exist $L \geq 0$, $\kappa' \geq 1$, and
	$\epsilon'\geq 0$ such that any
	$(\kappa,\epsilon,L)$-local quasi-geodesic in $X$ is a $(\kappa',\epsilon')$-quasi-geodesic.
\end{lemma}

\begin{lemma}[{\cite[Th{\'e}or{\`e}me 8.1]{CDP}}]\label{lem:treelike}
	Let $(X,d)$ be a $\CAT(-1)$ space, and let $\delta \geq 0$ be a Gromov hyperbolicity constant for $X$.
	Consider $n+1$ points $x_0,x_1,\ldots,x_n \in X$, and let $k \in \NN$
	be such that $2n \leq 2^k+1$. Let $Y$ be the subset of $X$ obtained by taking the union of the
	geodesic segments $[x_0,x_i]$, $1\leq i\leq n$. Then there exists a metric tree
	$(\mathcal{T},d_{\mathcal{T}})$ and
	a surjective map $f: Y \to \mathcal{T}$ such that the following is true.
	\begin{enumerate}
		\item On each $[x_0,x_i]$, the map $f$ restricts to an isometry.
		\item For every $x,y \in Y$, we have
				$d(x,y)-2k\delta \leq d_{\mathcal{T}}(f(x),f(y)) \leq d(x,y)$. 
	\end{enumerate}
\end{lemma}
\begin{lemma}\label{lem:treelike2}
Let $X$, $x_0,\, \ldots,\, x_n$, $Y$, $\mathcal{T}$, and $f$ be as in Lemma
\ref{lem:treelike}.  Take any $1 \leq i,j \leq n$. Then there is some $C \geq 0$ depending only on $n$ and
$X$ such that the geodesic segment $[x_i,x_j]$ is contained in the $C$-neighborhood of any union of
subsegments of $Y$ which injectively corresponds under $f$ to the geodesic segment $[f(x_i),f(x_j)]$ in
$\mathcal{T}$.
\end{lemma}
\begin{proof}
Choose a hyperbolicity constant $\delta>0$ for $X$ and choose $k \in \NN$ so that $2n \leq 2^k + 1$.
By (2) of Lemma \ref{lem:treelike}, such a union of segments is the image of 
a $(1,2k\delta)$-quasi-geodesic in $X$ starting in $B(x_i,2k\delta)$ and ending in $B(x_j,2k\delta)$.
The stability of quasi-geodesic segments (Lemma \ref{lem:finitestability}) gives the desired conclusion.
\end{proof}

\subsection{Invariance from Quasi-invariance} 
We formulate a general criterion for finding an invariant measure in the same measure class as a quasi-invariant measure. This
argument is standard, see e.g.\ \cite{CT1, Fur, Pic},
but it is worth emphasizing it in
a convenient form. It will be applied when $G = \Gamma$ is acting diagonally on $\partial_\infty^2 X$.

\begin{lemma}\label{lem:cocycles}
Let $G$ be a group acting on a set $Y$, and suppose that $C : G\times Y \to \RR$ is a bounded function
satisfying the cocycle relation
\[
	C(gh,y) = C(h,y)+C(g,hy)
.\]
Then the function $\psi(y) \coloneqq \sup_{g\in G}C(g,y)$ solves the cohomological equation
\[
	C(g,y) = \psi(y) - \psi(gy) 
.\]
\end{lemma}
\begin{proof}
We calculate that
\begin{align*}
	\psi(y)-\psi(gy) &= \sup_{h\in G}C(h,y) - \sup_{h\in G} C(h,gy) \\
	&= \sup_{h\in G}C(h,y)-\sup_{h \in G} C(hg,y) + C(g,y) \\
	&= \sup_{h\in G}C(h,y) - \sup_{h\in G}C(h,y) + C(g,y) = C(g,y). \qedhere
\end{align*}
\end{proof}

\begin{proposition}\label{prop.invariance}
	Let $G$ be a countable group acting measurably on a measure space $(Y,\lambda')$
	so that $\lambda'$ is \textit{quasi-invariant} under the action. 
	That is, $G$ preserves the measure class of
	$\lambda'$, and there exists a constant $k \geq 1$ independent of $g \in G$ such that
	\begin{equation}\label{eq:quasi-invariance}
	k^{-1} \leq \frac{d g_* \lambda'}{d\lambda'}(y) \leq k
	\end{equation}
	for almost every $y \in Y$. Then there exists a positive measurable function $\psi : Y \to \RR$,
	bounded away from 0 and $\infty$, such that the measure $d\lambda \coloneqq \psi \,d\lambda'$
	is $G$-invariant.
\end{proposition}
\begin{proof}
For almost every $y \in Y$, we define
	\[
		\psi(y) \coloneqq \sup_{g\in G} \frac{dg_* \lambda'}{d\lambda'}(y)
	.\] 
	Applying Lemma \ref{lem:cocycles} to the logarithm of the Radon-Nikodym cocycle, given by
	$C(g,y) = \log \frac{dg^{-1}_*\lambda'}{d\lambda'}(y)$, we deduce that
\[
	\frac{dg^{-1}_*\lambda'}{d\lambda'}(y) = \frac{\psi(y)}{\psi(gy)}
\] 
on a full measure set. Therefore
\[
	\frac{dg^{-1}_*\lambda}{d\lambda}(y) = \frac{\psi(gy)}{\psi(y)}
	\frac{dg_*^{-1}\lambda'}{d\lambda'}(y) =
	\frac{\psi(gy)}{\psi(y)}\frac{\psi(y)}{\psi(gy)} = 1.
	\qedhere
\] 
\end{proof}

\subsection{More on isometries} \label{sec:isometriesetc} 
We collect some basic lemmas about isometries, stabilizer groups, and inducing measures on a quotient space.
Suppose $(Y,d_Y)$ is a proper metric space and $\Gamma < \mathrm{Isom}(Y)$ is discrete.
Let $Y_0 = \Gamma \backslash Y$, and let $\pi_Y : Y \to Y_0$ denote the quotient map.
 For $y \in Y$, let $\mathrm{Stab}_\Gamma(y)$ denote the stabilizer group of $y$ in $\Gamma$. For $n \in \NN$, we let
\[\mathrm{Fix}_n(Y) = \{y \in Y \colon \abs{\mathrm{Stab}_\Gamma(y)} = n\}, \]
and we let $\mathrm{Fix}_n(Y_0) = \pi_{Y}(\mathrm{Fix}_n(Y))$.
\begin{remark} \label{pushdownameasure}
Let $\tilde m$ denote a $\Gamma$-invariant Radon measure on $Y$. Following \cite[ \S2.6]{PPS}, we describe how  $\tilde m$ induces a measure on $Y_0$. When $\Gamma$ does not act freely, one cannot simply consider the push-down of $\tilde m$ by $\pi_Y$. However,  the restriction of $\pi_Y$ to $\mathrm{Fix}_n(Y)$ is a local
homeomorphism onto its image and we use this to obtain a measure $m_n$ on $\mathrm{Fix}_n(Y_0)$
for each $n \in \NN$. The measure induced by $\tilde m$ on $Y_0$ is defined to be
\[
	m \coloneqq \frac{1}{n}\sum_{n\in\NN} m_n
.\] 
The normalization $\frac{1}{n}$ is necessary for the map $\tilde m \to m$ to be continuous, see \cite[ \S2.6]{PPS}.
\end{remark}

For $y \in Y$, we define 
\begin{equation} \label{epsilonc}
	\epsilon(y) \coloneqq \min_{a\in \Gamma,\, ay \neq y}\frac{1}{2}d_{Y}(ay,y).
\end{equation}
Since $\Gamma$ acts discretely on $Y$, we have $\epsilon(y) > 0$ for every $y \in Y$.
Since $\epsilon$ is $\Gamma$-invariant, we can also consider it as a function on $Y_0$.

\begin{remark}\label{rem:quotientball}
If the action of $\Gamma$ on $Y$ is free, then $\epsilon(y)$
is just the injectivity radius of the quotient map $\pi_{Y}$ at $y$.
More generally, we have the following useful fact. 
Let $y\in Y$ be a lift of $y_0 \in Y_0$, and let $r\leq \eps(y_0)$. Then $a \cdot B(y,r)$ intersects
$B(y,r)$ if and only if $a \in \mathrm{Stab}_\Gamma(y)$, and if $a \in \mathrm{Stab}_\Gamma(y)$,
then $a \cdot B(y,r) = B(y,r)$. Hence $B(y_0,r)$ is naturally identified with
$\mathrm{Stab}_\Gamma(y)\backslash B(y,r)$ as long as $r \leq \epsilon(y_0)$.
\end{remark}

\begin{remark}\label{rem:epsbounded}
The function $\epsilon$ may not be bounded below. 
If $\Gamma$ does not act cocompactly,
then we can have $\epsilon(y_n) \to 0$ if $(y_n)_{n\in\NN}$
is a sequence in $Y_0$ that leaves every compact subset of $Y_0$.
Alternatively, if the action of $\Gamma$ on $Y$ is not free,
there could be a sequence $(y_n)_{n\in\NN}$ in $Y$ such that $y_n \to y$,
where $y \in Y$ is fixed by $a \in \Gamma$, 
but the $y_n$ are not fixed by $a$. We would have $\epsilon(y_n) \to 0$, even though $\epsilon(y) > 0$.
\end{remark}

We have the following basic lemmas concerning the function $\epsilon$.

\begin{lemma}\label{lem.epsiloncontinuous}
	The function $\epsilon$ is continuous on each $\mathrm{Fix}_n(Y)$.
\end{lemma}
\begin{proof}
Suppose that $(y_k)_{k\in\NN}\subseteq \mathrm{Fix}_n(Y)$
is a sequence converging to $y \in \mathrm{Fix}_n(Y)$.
Observe that if $B(y,\epsilon(y))$ intersects
$a\cdot B(y,\epsilon(y))$, then $ay = y$.
Thus, if $k$ is large enough that $y_k \in B(y,\epsilon(y))$,
we must have $\mathrm{Stab}_\Gamma(y_k) \subseteq \mathrm{Stab}_\Gamma(y)$.
These sets have the same cardinality, so they are equal for large enough $k$.
It follows from the definition of $\epsilon$ that $\lim_{k\to\infty}\epsilon(y_k) =
\epsilon(y)$.
\end{proof}

\begin{lemma}\label{lem.fixncompact}
Let $n \in \NN$, and suppose $y \in \mathrm{Fix}_n(Y)$. Then for any $r < \epsilon(y)$,
the set $\mathrm{Fix}_n(Y) \cap \overline{B}(y,r)$ is compact.
\end{lemma}
\begin{proof}
Since $d_Y$ is proper, it suffices to prove that $\mathrm{Fix}_n(Y) \cap \overline{B}(y,r)$ is closed.
Suppose that $y'_k \to y'$ for some sequence
${y'_k\in\mathrm{Fix}_n(Y) \cap
\overline{B}(y,r)}$.  Observe that $\mathrm{Stab}_\Gamma(y) = \mathrm{Stab}_\Gamma(y'_k)$ since  $y'_k
\in\mathrm{Fix}_n(Y) \cap B(y,\epsilon(y))$ for all $k \in \NN$. It follows that $\mathrm{Stab}_\Gamma(y)
\subseteq \mathrm{Stab}_\Gamma(y')$.
Since $y' \in \overline{B}(y,r) \subseteq B(y,\epsilon(y))$,
we also have $\mathrm{Stab}_\Gamma(y') \subseteq
\mathrm{Stab}_\Gamma(y)$. Hence $\mathrm{Stab}_\Gamma(y') =
\mathrm{Stab}_\Gamma(y)$, and so $y' \in \mathrm{Fix}_n(Y)$.
This shows that $y' \in \mathrm{Fix}_n(Y)
\cap \overline{B}(y,r)$, and thus
$\mathrm{Fix}_n(Y)\cap \overline{B}(y,r)$ is closed.
\end{proof}
	
	By Lemma \ref{lem.epsiloncontinuous} and the fact that $\eps$ is $\Gamma$-invariant,
	we know that $\eps$ is continuous on $\mathrm{Fix}_n(Y_0)$.
If $y_0 \in \mathrm{Fix}_n(Y_0)$ and $r < \epsilon(y_0)$,
then the set $\mathrm{Fix}_n(Y_0) \cap \overline{B}(y_0,r)$ is compact by Lemma \ref{lem.fixncompact}.
As an immediate corollary, we have the following result.

\begin{lemma}\label{cor:epsbound1}
	Let $n\in \NN$, and suppose $y_0 \in \mathrm{Fix}_n(Y_0)$.
	Then for any $r < \epsilon(y_0)$, we have 
	$\inf \{ \eps(y) : y \in \mathrm{Fix}_n(Y_0) \cap B(y_0,r)\}>0$.
\end{lemma}

Suppose that $X$ and $\Gamma$ are defined as in \S\ref{subsec:CAT (-1)}.
Applying the definitions of this section to
$(Y,d_Y) = (GX,d_{GX})$, then we have the following basic property of the function $\epsilon$.

\begin{lemma}\label{lem:epsbound2}
	For any $c_0 \in GX_0$ and $t \in \RR$, we have $\epsilon(g_tc_0) \geq \epsilon(c_0)e^{-2\abs{t}}$.
\end{lemma}
\begin{proof}
Take any lift $c \in GX$ of $c_0$.
For any $a \in \Gamma$ such that $ac \neq c$, by \eqref{eq:bddgrowth} we have
$d_{GX}(g_{t}c,ag_{t}c) \geq 2\epsilon(c_0)e^{-2\abs{t}}$. Furthermore, we have
$\mathrm{Stab}_\Gamma(g_tc) = \mathrm{Stab}_\Gamma(c)$ since the geodesic flow is $\Gamma$-equivariant. 
Hence we have $\epsilon(g_{t}c_0) \geq \epsilon(c_0) e^{-2\abs{t}}$.
\end{proof}
\subsection{The Bowen property}\label{subsec.Bowen} Let $\mathcal F = (f_t)_{t\in\RR}$ be a continuous flow on a complete separable metric space $(Z,d)$. Let $\varphi:Z \to \RR$ be a continuous potential function. Recall from \eqref{ergodicintegral}, that we use $\Phi$ to denote  the ergodic integrals of $\varphi$. 
\begin{definition}
A potential $\varphi: Z \to \RR$ has the \emph{Bowen property on $Z$ at scale $\delta>0$} (for the flow $\mathcal F$) if
there exist a constant $C \geq 0$ such that, for any $x \in Z$ and any $t \geq 0$, we have
\begin{equation}\label{eq.Bowenprop}
\sup_{y \in B_t(x,\delta)}\abs{\Phi(x, t) -\Phi(y, t)} \leq C.
\end{equation}
\end{definition}
If $\varphi:GX_0 \to \RR$ satisfies the Bowen property at scale $\eps$ on $GX_0$, then its lift  satisfies the Bowen property on $GX$. This is because the projection map $\pi_{GX} : GX \to GX_0$ does not increase distance, and thus $\pi_{GX} (B_t(c,\delta)) \subseteq B_t(\pi_{GX}(c), \delta)$ for any $c \in GX$. For the converse, 
let $\eps:GX \to (0, \infty)$ be the function defined at \eqref{epsilonc} with $Y = GX$.
It is clear that, if $d_{GX}(c,c') < \epsilon(c)$, then 
\[d_{GX_0}(\pi_{GX}(c),\pi_{GX}(c')) = d_{GX}(c,c').\]
It can be shown that, for $c \in GX$ and $t\geq0$, if $\epsilon_0 < \inf_{s\in [0,t]}\epsilon(g_sc)$,
then $\pi_{GX}(B_t(c,\epsilon_0)) = B_t(\pi_{GX}(c),\epsilon_0)$.
Hence, if $\epsilon(c)$ has a lower bound
$\epsilon_0 > 0$ over all $c \in GX$, then the Bowen property on $GX$ at scale $\epsilon_0$
is equivalent to the Bowen property on $GX_0$ at scale $\epsilon_0$.
By Remark \ref{rem:epsbounded}, this may not be the case in our setting, and thus 
the Bowen property on $GX$ does not necessarily imply the Bowen property on $GX_0$. 
\begin{lemma} \label{lem:scaleindependence}
Let $(X, d)$ be a $\CAT(-1)$ space. Suppose that there exists $\delta>0$ so that $\varphi$ has the Bowen
property on $GX$ at scale $\delta>0$. Then for all $L>0$, the function
$\varphi$ has the Bowen property on $GX$ at scale $L$.
\end{lemma}
\begin{proof} Let $\varphi$ satisfy the Bowen property at scale $\delta>0$ with constant
$C > 0$, and suppose $L > 0$. 
The $\CAT(-1)$ property of $X$ implies there exists a constant $s \geq 0$ depending only on $\delta$ and
$L$ such that, if $c \in GX$ and $c' \in B_t(c,L)$,
then there exists $\abs{r} \leq s$ such that $g_{s+r}c' \in
B_{t-2s}(g_sc,\delta)$; this can be seen from the proof of \cite[Proposition 4.3]{CLT19}.
Then we have
\begin{align*}
	\textstyle \abs{\int_0^t\varphi(g_rc)\,dr - \int_{0}^{t}\varphi(g_rc')\,dr}
	&\textstyle \leq 6s\| \varphi \|_\infty + \abs{\int_{s}^{t-s}\varphi(g_rc)\,dr -
\int_{s+s'}^{t-s+s'}\varphi(g_rc')\,dr} \\
	&\leq 6s\| \varphi \|_\infty + C. \qedhere
	\end{align*}
\end{proof}
We say that $\varphi$ has the \emph{Bowen
	property} if it has the Bowen property on $GX$ at some scale $L>0$ (equivalently at all
	scales $L>0$ by Lemma \ref{lem:scaleindependence}).
	This is a priori weaker than the Bowen property on $GX_0$ at some scale. The Bowen property (on
	$GX$) is our main regularity assumption on $\varphi$.

Similarly, we say that $\varphi$ \emph{is H\"older} if it is
H\"older continuous as a function on $GX$. This is generally weaker than asking for $\varphi$ to be
H\"older on $GX_0$. The following statement on $GX$
is contained in the proof of \cite[Proposition 4.3]{CLT19}. 
\begin{lemma}\label{lem:HimpliesB}
	If $\varphi$ is bounded and H\"older, then it satisfies the Bowen property.
\end{lemma}

\section{The weighted Patterson-Sullivan construction} \label{s.weightedPS}
As in \S\ref{subsec:CAT (-1)}, we suppose that $(X,d)$ is a proper geodesically complete $\CAT(-1)$
space, $\Gamma < \mathrm{Isom}(X)$ is a non-elementary discrete group, and
$X_0 \coloneqq \Gamma \backslash X$ and $GX_0 \coloneqq \Gamma \backslash GX$ are
the quotients. Let $\varphi : GX_0 \to \RR$ be a continuous potential function.
We also write $\varphi$ for the lift of the potential to $GX$. 
We develop a Patterson-Sullivan construction in this setting. 
References include Paulin, Pollicott and Schapira \cite{PPS} for the construction with potentials for
pinched negatively curved manifolds, Coornaert \cite{Co93} for the unweighted construction in the Gromov
hyperbolic setting, and \cite{Pic, CT1, CT2} for some weighted constructions for hyperbolic groups.
\subsection{Weight function}
Given $x,y \in X$, let $\CCC([x, y])$ be the set of geodesic lines $c \in GX$ with $c(0)=x$ which extend the geodesic segment $[x,y]$. That is, we define
\[
\CCC([x, y]) \coloneqq\{c \in GX \colon {c}(0)=x, \:\: {c}(d(x,y))=y \}.
\]
The set $\CCC([x,y])$ is nonempty because $X$ is assumed to be geodesically complete. We define a weight function $\bphi : X \times X \to \RR$ by
setting
\begin{equation} \label{eq:phibar}
	\bphi (x, y) \coloneqq 
	\sup
	\left\{
		\int_{0}^{d(x,y)}\varphi(g_tc)\,dt \colon c \in \CCC([x, y]) 
	\right\}
.\end{equation}
Since $\varphi$ is continuous and $\mathcal{C}([x,y])$ is a nonempty compact set, the supremum in the
definition of $\bphi(x,y)$ is always attained. We have the following continuity property.
\begin{lemma} \label{lem.usc}
The function $\bphi$ is upper semi-continuous.
\end{lemma}
\begin{proof}
Suppose that $x_n \to x$ and $y_n \to y$
are convergent sequences in $X$, and suppose 
that $\lim_{n\to\infty}\bphi(x_n,y_n)$ exists.
Then, since the sequence of geodesic lines $c_n$ realizing the values
$\bphi(x_n,y_n)$ remain within a compact set, by passing to a subsequence we may assume that
$c_n \to c$. By the definition of $\bphi$ as a supremum, along with the continuity of $\varphi$,
we have
\[
	\bphi(x,y) \geq \int_{0}^{d(x,y)}\varphi(g_tc)\,dt =
	\lim_{n\to\infty}\int_{0}^{d(x_n,y_n)}\varphi(g_tc_n)\,dt =
	\lim_{n\to\infty} \bphi(x_n,y_n)
. \qedhere\]
\end{proof}
\begin{remark}
In the presence of branching, $\bphi$ is not necessarily continuous.
To see that lower semi-continuity can fail,
let $X$ be a tree, fix $x_n = x$, and allow $y_n$ to approach a vertex $y$ of degree $\geq 3$
in such a way that $d(x,y_n)$ is decreasing.
It is clear in this case that there exists $c \in \CCC([x,y])$ that does
not arise as a limit of any sequence $c_n \in \CCC([x,y_n])$. Furthermore, if $x$ is not a fixed point
of any element of $\Gamma$, then the same is true on $GX_0$. One can thus use Tietze's
theorem to specify a continuous $\Gamma$-invariant potential $\varphi$ which assigns a larger weight
along $c$ than along any of the geodesic lines in $\mathcal{C}([x,y_n])$, and we obtain $\bphi(x,y) > \limsup_{n\to\infty}\bphi(x,y_n)
.$
\end{remark}
\begin{remark}
Lemma \ref{lem.usc} implies that $\bphi$ is measurable, which is all that will be required in our
construction. The same proof as Lemma \ref{lem.usc} shows that, if $\underline{\varphi}$
is defined analogously to $\bphi$ using an infimum,
then it would be lower semi-continuous. This is equally suitable to use as a weight function in
our proofs. If $\varphi$ is tempered, then $\bphi = \underline{\varphi}$, and hence the weights are both
	upper and lower semi-continuous. For our class of potential functions, we do not expect there to be any way to choose a continuous weight function.
\end{remark} 

\subsection{Some properties of $\overline \varphi$}
First we prove a comparison lemma for the ergodic integrals of $\varphi$.
Recall that we write $\Phi(c,[s,t])$ for the ergodic integral $\int_s^t \varphi(g_rc)\,dr$.
\begin{lemma}\label{lem:integrals}
Suppose that $\varphi$ is bounded and satisfies the Bowen property.
Then for any $L \geq 0$ there exists $K=K(L) \geq 0$ such that, for any $c,c' \in GX$ 
and any $s,t,s',t' \in \RR$ with $s\leq t$ and $s'\leq t'$,
if $d(c(s),c'(s')) \leq L$ and $d(c(t),c'(t')) \leq L$, then
\[
\abs{\Phi(c, [s,t])- \Phi(c', [s', t'])} \leq K.
\]
\end{lemma}
\begin{proof}
We can assume without loss of generality that $s=s'=0$ by reparameterizing $c$ and $c'$. Note that $\abs{t'-t} \leq 2L$ by the triangle
inequality, and hence
\begin{equation} \label{almostbounded}
\abs{\Phi(c, t)- \Phi(c', t')} \leq \abs{\Phi(c, t)- \Phi(c', t)} + 2L \| \varphi \|_\infty.
\end{equation}
Note that $d(c(0),c'(0)) \leq L$
and $d(c(t), c'(t)) \leq d(c(t),c'(t'))+\abs{t'-t} \leq 3L$.
Using the convexity of the distance on $X$, we have ${d(g_sc(0),g_sc'(0)) \leq 3L}$ for all $s \in [0,t]$.
Using Lemma \ref{lem:footprint}, we can find
$L' \geq 0$ depending only on $L$ such that $d_{GX}(g_sc,g_sc') \leq L'$ for all $s \in
[0,t]$, and thus $c' \in B_t(c,L')$.

Since $\varphi$ satisfies the Bowen property at all scales by Lemma \ref{lem:scaleindependence}, it follows
that $\abs{\Phi(c, t)- \Phi(c', t)}$ is bounded by a constant $C(L') \geq 0$. Let $K = C(L')  + 2L\|
\varphi \|_\infty$. Then the required bound follows by \eqref{almostbounded}.
\end{proof}

We now prove the key properties that we need from $\bphi$, including the roughly geodesic property,
which was introduced by Cantrell and Tanaka in \cite{CT1, CT2}.
\begin{lemma}\label{lem:bphi}
Suppose that $\varphi$ is bounded and satisfies the Bowen property. Then the function $\bphi$ satisfies
\begin{enumerate}
\item $\Gamma$-invariance: for any $a \in \Gamma$, we have $\bphi(ax,ay) = \bphi(x,y)$;
\item local boundedness near the diagonal: there exists a constant $C \geq 0$ such that,
if $d(x,y) \leq 1$, then $\bphi(x,y)\leq C$;
\item the \emph{roughly geodesic property}: for any $L \geq 0$, there exists
$K_L \geq 0$ such that, whenever $y$ is in the $L$-neighborhood of $[x,z]$, then
\[
	\abs{\bphi(x,y)+\bphi(y,z)-\bphi(x,z)} \leq K_L.
\]

	\end{enumerate}
\end{lemma}

\begin{proof}
The $\Gamma$-invariance is immediate because the lifted potential $\varphi$ is $\Gamma$-invariant. Local boundedness near the diagonal is immediate by setting $C = \|\varphi \|_\infty$.

We establish the roughly geodesic property. Let $y$ belong to the $L$-neighborhood of $[x,z]$, and let $q \in [x,z]$ minimize $d(y,q)$,
so that $d(y,q) \leq L$. Let $c_{xy} \in GX$ be a geodesic which attains the supremum in the definition of $\bphi(x,y)$: that is, $c_{xy} \in \mathcal{C}([x,y])$ with  $\bphi(x,y) = \Phi(c_{xy},d(x,y))$.
Similarly consider $c_{xz},c_{yz}$ attaining the supremums in $\bphi(x,z)$ and $\bphi(y,z)$, respectively.
Let $t = d(x,y)$, $t' = d(x,q)$, $s = d(y,z)$, and $s' = d(q,z)$.
Then
\begin{align*}
	\abs{\bphi(x,y)+\bphi(y,z)-\bphi(x,z)} &= 
	\abs{\Phi(c_{xy}, t)+ \Phi(c_{yz}, s)- \Phi(c_{xz}, t'+s')}\\
	& \leq \abs{\Phi(c_{xy}, t)- \Phi(c_{xz}, t')} 
	+ \abs{\Phi(c_{yz}, s) - \Phi(g_{t'}c_{xz}, s')}\\ &\leq 2K,
\end{align*}
where $K = K(L)$ is the constant from Lemma \ref{lem:integrals}, completing the proof.
\end{proof}
	
\begin{lemma}\label{lem:independence}
If $\varphi$ is bounded and satisfies the Bowen property, then for any $p,x,p',x' \in X$, the set $\{\abs{\bphi(p,ax) - \bphi(p',ax')}\colon a\in\Gamma\}$
is bounded. 
\end{lemma}
\begin{proof}
Let $L = \max\{d(p,p'),d(x,x')\}$.  Using the roughly geodesic property from Lemma \ref{lem:bphi},
we have
\begin{align*}
	\abs{\bphi(p,ax) - \bphi(p',ax')} &\leq \abs{\bphi(p,p')+\bphi(p',ax)-\bphi(p',ax')}+K_L \\
	&\leq \abs{\bphi(p,p')+\bphi(p',ax')+\bphi(ax',ax)-\bphi(p',ax')}+2K_L \\
	&\leq \abs{\bphi(p,p')}+\abs{\bphi(ax',ax)}+2K_L \leq L\| \varphi \|_{\infty} + 2K_L. \qedhere
\end{align*}
\end{proof}

\subsection{Poincar{\'e} series} 
\label{sec:weightedpoincare}
We make a standing assumption that $\varphi$ is
bounded and satisfies the Bowen property.
Given basepoints $p,x \in X$ and $s \in \RR$, we define a Poincar\'e series
\[
	\PPP (s; \varphi,p,x) \coloneqq 
	\sum_{a\in\Gamma} e^{\overline{ (\varphi- s)}(p, ax) } =
	\sum_{a\in\Gamma} e^{\bphi(p, ax)-sd(p,ax)} 
.\] 
This generalizes the classical Poincar\'e series (i.e.\ our Poincar\'e series when $\varphi = 0$)
by using $\bphi$ to weight the terms. We define the \textit{critical exponent} for $\varphi$ to be
\[
	\delta_\varphi \coloneqq \inf\{s \colon \PPP(s; \varphi, p, x) < \infty \},
\]
noting that by Lemma \ref{lem:independence}, this critical exponent is independent of the choice of
basepoints $p,x \in X$.
In the Riemannian setting, this definition of $\delta_\varphi$ agrees with
\eqref{PPSweightedPoincare}. See \cite[\S3.2]{PPS}.

We say that $\varphi$ is \textit{of divergence type} if 
$\PPP(\delta_\varphi; \varphi,p,x) = \infty$, and we say $\varphi$ is \emph{of convergence type} if
$\PPP(\delta_\varphi; \varphi,p,x)$ is finite.
Again by Lemma \ref{lem:independence}, the definition of divergence or convergence type
is independent of the choice of $p,x \in X$.
For simplicity, we fix a basepoint $p \in X$ and set
\[
\PPP(s;\varphi) \coloneqq \PPP(s;\varphi,p,p).
\]
Note that $\PPP(s;\varphi)=\PPP(s;\varphi\circ\iota)$, 
so we also have $\delta_{\varphi} = \delta_{\varphi\circ\iota}$,
and $\varphi$ and $\varphi \circ \iota$ are either
both of divergence type or both of convergence type.

The classical unweighted critical exponent $\delta_\Gamma$ is defined as $\delta_\varphi$ for $\varphi
= 0$.  See \cite{BJ97, nC21} for dimension-theoretic interpretations of $\delta_\Gamma$. Recall that our main theorems have the hypothesis that
$\delta_\Gamma < \infty$. Since we have assumed that $\varphi$ is a bounded
function, it is easy to verify that $\delta_\Gamma < \infty$ if and only if $\delta_\varphi <\infty$, and
in particular, we have $\delta_\Gamma -\|\varphi\|_{\infty}\leq
\delta_\varphi\leq\delta_\Gamma+\|\varphi\|_{\infty}$ when they are both finite.

\subsection{Gibbs quasi-cocycle and weighted Gromov product}\label{subsec:qc and gp}

We fix a basepoint $p \in X$, and we write $(x|y)$ for $(x|y)_p$.

\begin{definition}
Given $a\in \Gamma$, $x,y \in X$, we let
\[
	Q(a,x; \varphi) \coloneqq \bphi(ap,x)-\bphi(p,x),
\] 
\[
(x| y; \varphi) \coloneqq \frac{1}{2}(\bphi(x,p) + \bphi(p,y) - \bphi(x,y))
.\] 
\end{definition}
We call the functions $Q(a,x;\varphi)$ and $(x|y;\varphi)$ the \emph{Gibbs quasi-cocycle}
and \emph{weighted Gromov product} for $\varphi$
respectively. These are natural weighted analogs of the Busemann cocycle and the Gromov product.
The following statement is easily verified directly from the definitions.
\begin{lemma}\label{lem:calc}
For any $a \in \Gamma$, $x,y\in X$, we have
\[
	2(a^{-1}x|a^{-1}y;\varphi)-2(x|y;\varphi) = Q(a,x;\varphi\circ\iota)+Q(a,y;\varphi)
.\] 
\end{lemma}
We now work
to extend the domain of definition of these functions to the boundary. The main technical lemma that allows us to do this is the following.
\begin{lemma}\label{lem:properties} The following are true for the Gibbs quasi-cocycle
	and weighted Gromov product.
	\begin{enumerate}[label=(\alph*)]
	\item There exists $A \geq 0$ such that, for any $a \in \Gamma$, there
			exists $R \geq 0$ such that, whenever $(x|x') \geq R$, then
			\[
				\left|Q(a,x;\varphi)-Q(a,x';\varphi) \right| \leq A
			.\] 
	\item There exists $B \geq 0$ such that, for any $L \geq 0$, there exists $R \geq 0$ such
			that, whenever $(x|y)\leq L$ and $\min\{(x|x'),(y|y')\} \geq R$, then
			\[
				\left|(x|y;\varphi) - (x'|y';\varphi) \right| \leq B
			.\] 
	\end{enumerate}
\end{lemma}
The  proof relies on the trees provided by Lemma \ref{lem:treelike}. 
\begin{figure}
	\centering
	\begin{minipage}{.5\textwidth}
	\centering
	\begin{tikzpicture}[scale=.4]
		\draw (3,0) -- (1.5,2) -- (1.5,5) -- (3,7);
		\draw (0,0) -- (1.5,2) -- (1.5,5) -- (0,7);
		\filldraw (0,0) circle (0.1cm);
		\filldraw (3,0) circle (0.1cm);
		\filldraw (1.5,3.5) circle (0.1cm);
		\filldraw (0,7) circle (0.1cm);
		\filldraw (3,7) circle (0.1cm);
		\node at (-1,-.1) {\small $f(p)$};
		\node at (4.6,0) {\small $f(ap)$};
		\node at (1.05,3.4) {\small $\overline{q}$};
		\node at (-1,7) {\small $f(x)$};
		\node at (4.1,7) {\small $f(x')$};
	\end{tikzpicture}
	\captionof{figure}{The tree $\mathcal{T}$ used in the proof of (a)}
	\label{fig:a}
	\end{minipage}%
	\begin{minipage}{.5\textwidth}
	\centering
	\begin{tikzpicture}[scale=.815]
		\draw (0,0) -- (0,1.5) -- (-.75,2.4375)
			-- (-1.534,2.5365);
		\draw (-.75,2.4375) -- (-.6745,3.224);
		\draw (0,1.5) -- (.75,2.4375)
			-- (1.534,2.5365);
		\draw (.75,2.4375) -- (.6745,3.224);
		\filldraw (0,0) circle (0.05cm);
		\filldraw (-.6745,3.224) circle (0.05cm);
		\filldraw (-1.534,2.5365) circle (0.05cm);
		\filldraw (.6745,3.224) circle (0.05cm);
		\filldraw (-.75/2,2.4375/2+1.5/2) circle (0.05cm);
		\filldraw (.75/2,2.4375/2+1.5/2) circle (0.05cm);
		\filldraw (1.534,2.5365) circle (0.05cm);
		\node at (-.5,0) {\small $f(p)$};
		\node at (-.6745+.03,3.224+.3) {\small $f(x')$};
		\node at (-1.534-.5,2.5365+.04) {\small $f(x)$};
		\node at (.6745-.03,3.224+.3) {\small $f(y')$};
		\node at (-.75/2-.22,2.4375/2+1.5/2-.18) {\small $\overline{q}_1$};
		\node at (.75/2+.33,2.4375/2+1.5/2-.18) {\small $\overline{q}_2$};
		\node at (1.534+.5,2.5365+.04) {\small $f(y)$};
	\end{tikzpicture}
	\captionof{figure}{The tree $\mathcal{T}$ used in the proof of (b)}
	\label{fig:b}
	\end{minipage}
\end{figure}
\begin{proof}
To prove property (a), let $a \in \Gamma$, and $x,x' \in X$.
Let $R = d(p,ap)$, and suppose that $(x|x') \geq R$. We apply Lemma \ref{lem:treelike} with $n=3$ to the
points $p,ap,x,x'$, choosing $x_0 = p$ to be the root,  giving a function $f$ with values in a metric
tree $(\mathcal{T},d_{\mathcal{T}})$.

We first show that there must be a point $\bar q \in \mathcal{T}$ that lies on
any geodesic segment joining a point in $\{f(p),f(ap)\}$ to a point in $\{f(x),f(x')\}$.
See Figure \ref{fig:a}. It suffices to show that
\[
	d_{\mathcal{T}}(f(p),[f(x),f(ap)]) \leq d_\mathcal{T}(f(p),[f(x),f(x')])
.\] 
Because $\mathcal{T}$ is a metric tree, the distance from a point to a geodesic is exactly given by
the Gromov product, and hence it is equivalent to show that
\[
	(f(x)|f(ap))_{f(p)} \leq (f(x)|f(x'))_{f(p)}
.\]
By the triangle inequality and part (1) of Lemma \ref{lem:treelike}, we have
$(f(x)|f(ap))_{f(p)} \leq d_{\mathcal{T}}(f(p),f(ap)) = d(p,ap)$. 
Thus, by parts (1) and (2) of Lemma \ref{lem:treelike}, we have
\[ 
	(f(x)|f(x'))_{f(p)} \geq (x|x') \geq R = d(p,ap) \geq (f(x)|f(ap))_{f(p)},
\]
and so the existence of such a point $\overline{q}$ is guaranteed.

Let $q \in f^{-1}(\overline{q})$. By Lemma \ref{lem:treelike2},
there is a constant $C \geq 0$ such that $q$ is in the $C$-neighborhood of
any geodesic segment joining a point in $\{p,ap\}$ to a point in $\{x,x'\}$.
Thus we have
\begin{align*}
	\abs{Q(a,x;\varphi)-Q(a,x';\varphi)}
	&= \abs{\bphi(ap,x)-\bphi(p,x)-\bphi(ap,x')+\bphi(p,x')} \\
	&\leq 
	\left|
		(\bphi(ap,q)+\bphi(q,x))-(\bphi(p,q)+\bphi(q,x))\right. \\
		&\qquad\left.-(\bphi(ap,q)+\bphi(q,x'))+(\bphi(p,q)+\bphi(q,x'))
	\right|+4K_{C} \\
	&= 4K_{C},
\end{align*}
where $K_{C}$ is the constant from the roughly geodesic property.

For property (b), suppose that $L$ is given, and let $R=L+3\delta$, where $\delta>0$ is a hyperbolicity
constant for $X$. Let $x,x',y,y'$ be any points such that
$(x|y) \leq L$ and $\min\{(x|x'),(y|y')\} \geq R$. We apply Lemma \ref{lem:treelike} with $n=4$
to the points $p,x,x',y,y'$, again using the root $x_0=p$, giving a function $f$ with values in a tree $\mathcal{T}$. Since $n = 4$, we may choose $k=3$
in the statement of the lemma.  We show that there are $\overline{q}_1, \overline{q}_2 \in \mathcal{T}$ such that any
geodesic segment joining a point in $\{f(p),f(y),f(y')\}$ to a point in $\{f(x),f(x')\}$ passes through
$\overline{q}_1$, and any geodesic segment joining a point in $\{f(p),f(x),f(x')\}$ to a point in
$\{f(y),f(y')\}$ passes through $\overline{q}_2$.
See Figure \ref{fig:b}. To do this, it suffices to show that
\[
	\min\{d_{\mathcal{T}}(f(p),[f(x),f(x')]),d_{\mathcal{T}}(f(p),[f(y),f(y')])\} \geq
	d_{\mathcal{T}}(f(p),[f(x),f(y)])
,\] 
or in other words, it suffices to show that
\[
\min\{(f(x)|f(x'))_{f(p)},(f(y)|f(y'))_{f(p)}\} \geq (f(x)|f(y))_{f(p)}
.\] 
Now, the lemma allows us to estimate
\[
	(f(x)|f(x'))_{f(p)} \geq (x|x') \geq R = L+3\delta \geq (x|y)+3\delta \geq (f(x)|f(y))_{f(p)}
,\] 
and the proof that $(f(y)|f(y'))_{f(p)} \geq (f(x)|f(y))_{f(p)}$ is analogous.

Choose any $q_1 \in f^{-1}(\overline{q}_1)$, $q_2 \in f^{-1}(\overline{q}_2)$.
By Lemma \ref{lem:treelike2},
$q_1$ must lie in the $C$-neighborhood of any geodesic segment
joining a point in $\{p,y,y'\}$ to a point in
$\{x,x'\}$, and $q_2$ lies in the $C$-neighborhood of any geodesic segment joining a point in $\{p,x,x'\}$
to a point in $\{y,y'\}$. Using the roughly geodesic property, we see that
\begin{align*}
	\abs{(x|y;\varphi)-(x'|y';\varphi)}
	&= \tfrac{1}{2}\abs{\bphi(x,p)+\bphi(p,y)-\bphi(x,y)-\bphi(x',p)-\bphi(p,y')+\bphi(x',y')}\\
	&\leq\tfrac{1}{2} |(\bphi(x,q_1)+\bphi(q_1,p))+(\bphi(p,q_2)+\bphi(q_2,y))\\
	&\qquad\quad -(\bphi(x,q_1)+\bphi(q_1,q_2)+\bphi(q_2,y)) \\
	&\qquad\quad -(\bphi(x',q_1)+\bphi(q_1,p))-(\bphi(p,q_2)+\bphi(q_2,y'))\\ 
	&\qquad\quad+(\bphi(x',q_1)+\bphi(q_1,q_2)+\bphi(q_2,y'))|+\tfrac{1}{2}8K_{C}= 4K_{C}. \qedhere
\end{align*}
\end{proof}

\begin{remark} 
The proof of Lemma \ref{lem:properties} only uses the properties of $\bphi$ stated in
	Lemma \ref{lem:bphi} and the fact that 
	$X$ is Gromov hyperbolic and roughly geodesic. A similar argument can be used to show
	that the potentials considered in \cite{CT1, CT2} admit two-sided
	Gromov products only assuming that they satisfy the roughly geodesic property
	and $\Gamma$-invariance, which simplifies many of their statements. See also
	{\cite[Lemma 3.2.2]{thesis}} for a generalized construction.
\end{remark}
\begin{corollary}\label{cor:bdddisc}
Let $A \geq 0$ and $B \geq 0$ be the constants from Lemma \ref{lem:properties}.
Let $\xi\neq\eta\in\partial_\infty X$, and suppose that $x_n,x'_n,y_n,y'_n \in X$ are such that
$x_n,x'_n\to\xi$ and $y_n,y'_n\to\eta$.
Then
\begin{gather*}
	\limsup_{n\to\infty}\abs{Q(a,x_n;\varphi)-Q(a,x'_n;\varphi)} \leq A; \\
	\limsup_{n\to\infty}\abs{(x_n|y_n;\varphi)-(x'_n|y'_n;\varphi)} \leq B.
\end{gather*} 
\end{corollary}
\begin{proof}
This is immediate from Lemma \ref{lem:properties} and the fact that $(x_n|x'_n) \to \infty$ 
if and only if $x_n$ and $x'_n$ converge to the same point on the boundary.
The same fact implies that if
$x_n$ and $y_n$ converge to different points on the boundary,
then there must be some $L \geq 0$ such that $(x_n|y_n) \leq L$ for all $n$.
\end{proof}

We are therefore justified by Corollary \ref{cor:bdddisc} to make the following definition.
\begin{definition}\label{def:extension}
	Given $a \in \Gamma$, $\xi, \eta \in \partial_\infty X$, $\xi \neq \eta$, we define
	\[
		Q(a,\xi;\varphi)\coloneqq
		\sup
		\left\{\limsup_{n\to\infty}\:Q(a,x_n;\varphi) \colon
		x_n \in X, n \in \NN, x_n \to \xi \right\}
	\] 
	and
	\[
		(\xi|\eta;\varphi) \coloneqq 
		\sup
		\left\{\limsup_{n \to \infty} \: (x_n|y_n;\varphi)
		\colon x_n, y_n \in X, n \in \NN,x_n \to \xi, \:\: y_n\to\eta\right\}
	.\] 
\end{definition}
\begin{remark}
	Some references, such as \cite{PPS, Co93},
	take the limits in Definition \ref{def:extension} only along
	geodesic rays approaching $\xi,\eta \in \partial_\infty X$. By Corollary \ref{cor:bdddisc},
	defining it that way would  change Definition \ref{def:extension}
	by at most a constant, and thus does not affect our construction.
\end{remark}

Corollary \ref{cor:bdddisc} tells us that formulae on $X$ which involve the Gibbs quasi-cocycle and weighted Gromov product extend to the boundary, but only in a coarse sense. For example, letting $A \geq 0$ and
$B \geq 0$ be the constants from Lemma \ref{lem:properties}, one can check that
\begin{equation} \label{eq:quasicocycle}
	\abs{Q(a_1a_2,\xi;\varphi) - Q(a_2,a_1^{-1}\xi;\varphi) - Q(a_1,\xi;\varphi)} \leq 2A
\end{equation}
for all $a_1,a_2 \in \Gamma$, $\xi \in \partial_\infty X$,
and so $Q(\,\cdot\,,\,\cdot\,\,;\varphi)$ is indeed a quasi-cocycle on $\bar X$. 
Similarly, there is a coarse version of Lemma \ref{lem:calc} on the boundary: we have
\begin{equation}\label{eq:extcalc}
	\abs{2(a^{-1}\xi|a^{-1}\eta;\varphi)-2(\xi|\eta;\varphi)-Q(a,\xi;\varphi\circ\iota)
	-Q(a,\eta;\varphi)} \leq 2A+2B
\end{equation}
for all $\xi \neq \eta \in \partial_\infty X$, $a \in \Gamma$.
The constants appear because each expression may require a different
sequence to realize its value.


\begin{lemma}\label{cor:nbhds}
	With $A \geq 0$ as in Lemma \ref{lem:properties}, for every $a\in\Gamma$,
	there exist a finite number of neighborhoods
	$U_1,\ldots, U_n \subseteq \bar X$ which cover $\partial_\infty X$, and which satisfy
	\[
		\abs{Q(a,x;\varphi)-Q(a,x';\varphi)} \leq A
	\] 
	whenever $x,x'$ lie in a single $U_i$.
	Moreover, after possibly increasing $A$ and changing the neighborhoods $U_i$,
	we may additionally assume that, for all $s \in [0,1)$ and all $x, x' \in U_i$, we have
	\[
		\abs{Q(a,x;\varphi-s)-Q(a,x';\varphi-s)} \leq A
	.\] 
\end{lemma}
\begin{proof}
Let $a \in \Gamma$, and fix $R$ according to property (a) of Lemma \ref{lem:properties}.
Every $\xi\in\partial_\infty X$ has a neighborhood $\xi\in
U_\xi\subseteq\bar X$ such that $(x|x') \geq R$ for any $x,x' \in X\cap U_{\xi}$. Since $\partial_\infty X$
is compact, choose a finite subcover $U_1,\ldots,U_n$. Thus, if  $x,x' \in X\cap U_i$ for some $i$, then by (a) of Lemma \ref{lem:properties}, we have
$ \abs{Q(a,x;\varphi)-Q(a,x';\varphi)} \leq A$. The same is true for $x,x' \in U_i$ by the
definition of the extension.

For the second claim, we use the fact that, for any $x \in \bar X$, we have
\[
	Q(a,x;\varphi-s) = Q(a,x;\varphi)-sQ(a,x;1).
\] 
This is clear if $x \in X$. If $x \in \partial_\infty X$, then take any sequence $(x_n)_{n\in \NN} \subseteq X$ with $x_n \to x$.
Since $Q(a,x;1)$ is the standard Busemann cocycle, which extends continuously to the boundary, we have
\[
	\limsup_{n\to\infty}Q(a,x_n;\varphi-s) =
	\left(\limsup_{n\to \infty}Q(a,x_n;\varphi)\right)-sQ(a,x;1)
.\] 
By Definition \ref{def:extension}, we see immediately that $Q(a,x;\varphi-s)=Q(a,x;\varphi)-sQ(a,x;1)$.

Hence, when $x,x' \in U_i$ and $0\leq s \leq 1$, then
\begin{align*}
	\abs{Q(a,x;\varphi-s)-Q(a,x';\varphi-s)} &\leq \abs{Q(a,x;\varphi)-Q(a,x';\varphi)}
	+ \abs{Q(a,x;1)-Q(a,x';1)} \\
	&\leq A + \abs{Q(a,x;1)-Q(a,x';1)}
\end{align*}
Because the constant function $1$ is bounded and satisfies the Bowen property, 
the first claim of the lemma implies that there are neighborhoods $U'_1,\ldots,U'_{n'}$ and a constant
$A'$ such that $\abs{Q(a,x;1)-Q(a,x';1)} \leq A'$ whenever $x,x'$ lie in a single $U'_i$.
Hence we obtain the uniform bound in the second claim by replacing $A$ with $A+A'$ and refining
the cover $U_1,\ldots,U_n$ so that it is finer than $U_1',\ldots,U_n'$.
\end{proof}

\subsection{Patterson-Sullivan measures and the Gibbs state}\label{subsec:ps} 
\begin{definition}\label{def:ps}
We say that a probability measure $\mu$ on $\partial_\infty X$ is a \textit{quasi-conformal
measure for $\varphi$ of exponent $\sigma \in \RR$}
if $\mu$ is quasi-conformal with respect to 
${Q(a,\xi;\varphi-\sigma)}$: that is, there exists some $k \geq 1$ such that,
for any $a \in \Gamma$ and $\mu$-a.e.\ $\xi \in \partial_\infty X$,
\[
	k^{-1}e^{Q(a,\xi;\varphi-\sigma)} \leq \dd{a_*\mu}{\mu}(\xi) \leq ke^{Q(a,\xi;\varphi-\sigma)}
.\] 
\end{definition}
\begin{proposition}
\label{prop:ps}
Suppose $\Gamma$ has finite critical exponent.
Then there exists a quasi-conformal measure $\mu$ of exponent $\delta_\varphi$
for $\varphi$ with support in the limit set $\Lambda$. 
\end{proposition}
The proof of the proposition is a straightforward generalization of the one in
Coornaert's paper \cite{Co93}. We give the details for completeness.
\begin{proof}
Since $\delta_\Gamma < \infty$, we have $\delta_{\varphi} < \infty$ by boundedness of $\varphi$.
First we assume that the Poincar\'e series for $\varphi$ is  of divergence type. For $s > \delta_\varphi$, let
\[
	\mu_s \coloneqq \frac{1}{\PPP(s;\varphi)}\sum_{a \in \Gamma}
	e^{\overline{(\varphi-s)}(p,ap)}\mathcal{D}_{ap}
,\] 
where $\mathcal{D}_{ap}$ denotes the Dirac mass at $ap$, considered as a measure on $\bar X$.
Since $\mu_s$ is a probability measure for each $s > \delta_\varphi$, and $\bar X$ is compact, then
there exists a weak* limit point $\mu$ of $\mu_s$ as $s \to \delta_{\varphi}^+$.

Since $\PPP(s,\varphi)$ approaches $\infty$ as $s \to \delta_{\varphi}^+$, the measure $\mu$ must
have its support contained in $\partial_\infty X$.
This is obvious from the fact that, if $f:\bar X\to\RR$ is a continuous function
supported in $X$, then its support is bounded, and thus
\[
	\sum_{a \in \Gamma} e^{\overline{(\varphi-s)}(p,ap)}f(ap)
\] 
is uniformly bounded for $s > \delta_\varphi$, so that after taking a limit,
$\mu(f)=0$.  Furthermore, since every point in $\partial_\infty X \setminus \Lambda$
has a neighborhood in $\bar X$ disjoint from $\Gamma p$,
no point in $\partial_\infty X \setminus \Lambda$ can lie in the support of $\mu$.
It follows that $\supp\mu \subseteq \Lambda$.

We fix $a \in \Gamma$. For each $s > \delta_\varphi$, we can directly compute that
\begin{equation}\label{eq:rn}
	\frac{da_*\mu_s}{d\mu_s}(x) = e^{Q(a,x;\varphi-s)}
\end{equation}
for $\mu_s$-a.e.\ $x \in \bar X$. Let $\nu$ be the measure defined by
\[
	d\nu(\xi) = e^{-Q(a,\xi,\varphi-\delta_\varphi)}da_*\mu(\xi)
.\] 
Using Corollary \ref{cor:nbhds}, take $A \geq 0$ and neighborhoods
$U_1,\ldots, U_n \subseteq \bar X$ covering $\partial_\infty X$
so that, for any $x,y \in U_i$ and $s \in [\delta_\varphi,\delta_\varphi+1)$, we have
\[
	\abs{Q(a,x;\varphi-s)-Q(a,y;\varphi-s)} \leq A.
\] 
We emphasize that the constant $A$ is independent of the choice of $a \in \Gamma$. Let $f : \bar X\to\RR$ be a continuous function with support in one of the $U_i$. We fix
some reference point $y \in U_i$. Then, for $s \in (\delta_\varphi,\delta_\varphi+1)$, we have
\begin{align*}
	\mu_s(f) = \int e^{-Q(a,x;\varphi-s)}f(x)da_*\mu_s(x) \leq e^{A} e^{-Q(a,y;\varphi-s)}a_*\mu_s(f).
\end{align*}
Passing to the limit,  we have $\mu(f) \leq e^{A} e^{-Q(a,y;\varphi-\delta_\varphi)}a_*\mu(f)$. Furthermore,
we have
\begin{align*}
	a_*\mu(f) = \int e^{Q(a,\xi;\varphi-\delta_\varphi)}f(\xi)d\nu(\xi) \leq e^{A}
	e^{Q(a,y;\varphi-\delta_\varphi)}\nu(f),
\end{align*}
and thus $\mu(f) \leq e^{2A}\nu(f)$.

The other inequality is proved analogously and we conclude that
\begin{equation}\label{eq:qc0}
	e^{-2A}\nu(f) \leq \mu(f) \leq e^{2A}\nu(f)
\end{equation}
for any $f$ with support in one of the $U_i$. We extend (\ref{eq:qc0})
to every continuous function on $\bar X$ by writing an arbitrary continuous
function $f : \bar X \to \RR$ as
$f = f_X + \sum_{i=1}^{n}f_i$, where $f_X$ has support inside of $X$ (which implies $\mu(f_X) = \nu(f_X) = 0$),
and $f_i$ has support inside of $U_i$. 

Since $\mu$ and $\nu$ are Radon measures satisfying
(\ref{eq:qc0}) for every continuous function $f$,
it follows that $\mu$ and $\nu$ are absolutely continuous
with respect to each other, with Radon-Nikodym derivative bounded between $e^{-2A}$ and $e^{2A}$.
Since $\nu$ and $a_*\mu$ are absolutely continuous by definition,
it follows that $\mu$ and $a_*\mu$ are absolutely continuous, and for $\mu$-a.e.\ $\xi \in \partial_\infty
X$ we have
\[
	\frac{da_*\mu}{d\mu}(\xi) = e^{Q(a,\xi,\varphi-\delta_\varphi)}\frac{d\nu}{d\mu}(\xi) = e^{\pm
	2A} e^{Q(a,\xi;\varphi-\delta_\varphi)}.
\]
This establishes that $\mu$ is a quasi-conformal measure of exponent $\delta_\varphi$, with $k = e^{2A}$.

Now suppose that $\varphi$ is of convergent type. By \cite[Lemma 4.2]{BPP},
there exists a nondecreasing function $h : \RR_{\geq 0} \to \RR_{\geq 0}$ defined so that
\[
	\overline{\PPP}(s;\varphi) \coloneqq \sum_{a\in\Gamma}e^{\overline{(\varphi-s)}(p,ap)}h(d(p,ap))
\] 
diverges at $\delta_\varphi$, but such that $h$ is a \emph{slowly growing} function:
for each $\epsilon > 0$ there exists $T \geq 0$ such that $h(t+r) \leq e^{\epsilon r}h(t)$
whenever $r \geq 0$ and $t \geq T$. The fact that $h$ is slowly growing ensures that
$\overline{\PPP}(s;\varphi)$ still converges for $s > \delta_\varphi$, so we define
\[
	\mu_s = \frac{1}{\overline{\PPP}(s;\varphi)}\sum_{a \in \Gamma}e^{\overline{(\varphi-s)}(p,ap)}
	h(d(p,ap))\mathcal{D}_{ap}
\] 
for $s > \delta_\varphi$ and let $\mu$ be a limit point of $\mu_s$ as $s \to \delta_\varphi^+$.
The previous argument shows that $\mu$ is supported on $\Lambda$. Note that
\[
	\frac{da_*\mu_s}{d\mu_s}(x) = \frac{h(d(p,a^{-1}x))}{h(d(p,x))}e^{Q(a,\xi;\varphi-s)}
\] 
for $\mu_s$-a.e.\ $x$. To control the additional term involving $h$, let $R = d(ap,p)$, and note that for any $x \in X$, we have
\[
	\abs{d(p,a^{-1}x)-d(p,x)} = \abs{d(ap,x)-d(p,x)}\leq R
.\] 
Let $\epsilon = R^{-1}$. Using the slowly growing property and that $h$ is non-decreasing,  there exists $T \geq 0$ such that if $d(p,x) \geq T+R$, then we have 
\[
e^{-1} \leq  h(d(p,a^{-1}x))(h(d(p,x)))^{-1} \leq e.
\]
 By possibly removing finitely many terms in the definitions of
$\overline{\PPP}(s,\varphi)$ and $\mu_s$, we can assume that $d(p,x) \geq T+R$ for $\mu_s$-a.e.\ $x$ without
affecting the limit measure $\mu$. For this modified definition of $\mu_s$, we have
\begin{equation}\label{eq:rn2}
	e^{Q(a,x,\varphi-s)-1} \leq \frac{da_*\mu_s}{d\mu_s}(x) \leq e^{Q(a,x,\varphi-s)+1}
\end{equation}
for $\mu_s$-a.e.\ $x$. We follow the proof in the divergence type case with \eqref{eq:rn2} replacing \eqref{eq:rn}, and we see that the limit measure $\mu$
satisfies the quasi-conformal condition for $\varphi$ with exponent $\delta_\varphi$ and
constant $k = e^{2A+1}$.
\end{proof}
\begin{lemma}\label{lem:limitset}
If $\mu$ is a quasi-conformal measure with support contained in $\Lambda$, then its support is equal to
$\Lambda$. In particular, the measure constructed in Proposition \ref{prop:ps} is fully supported on
$\Lambda$.
\end{lemma}
\begin{proof}
This follows immediately from the minimality of the action of $\Gamma$ on $\Lambda$ and the
$\Gamma$-invariance of $\supp\mu$.
\end{proof}
\begin{proposition}[Construction of a geodesic current]
\label{prop:Invariant}
Let $\mu$ and $\mu^{\iota}$ be quasi-conformal measures for 
$\varphi$ and $\varphi\circ\iota$ respectively, both of the same exponent $\sigma \in \RR$.
Then the measure $\lambda'$ defined on $\partial_\infty^2 X$ by
\[
	d\lambda'(\xi,\eta) \coloneqq e^{-2(\xi|\eta;\varphi-\sigma)}d\mu^{\iota}(\xi)
	d\mu(\eta)
\] 
is quasi-invariant under the diagonal action of $\Gamma$. Thus there exists a measurable function
$\psi : \partial_\infty^2 X \to \RR$, bounded away from $0$ and $\infty$, such that
\[
	d\lambda \coloneqq \psi\, d\lambda'
\] 
is invariant under the diagonal action of $\Gamma$. We call $\lambda$ a 
\emph{Gibbs current for $\varphi$ of exponent $\sigma$.}
\end{proposition}
\begin{proof}
A quasi-conformal measure for $\varphi$ of exponent $\sigma$ is a quasi-conformal
measure for $\varphi-\sigma$ of exponent $0$, and we have $(\varphi-\sigma)\circ\iota = \varphi\circ\iota -
\sigma$. Hence without loss of generality, we may assume that $\sigma = 0$ by replacing $\varphi$ with
$\varphi-\sigma$. For every $a \in \Gamma$, we can calculate that
\begin{align*}
	\abs{\log\frac{da_*\lambda'}{d\lambda'}(\xi,\eta)} &=
	\abs{
		2(\xi|\eta;\varphi) - 2(a^{-1}\xi|a^{-1}\eta;\varphi)
		+\log\frac{da_*\mu^\iota}{d\mu^\iota}(\xi)+\log\frac{da_*\mu}{d\mu}(\eta)
	}\\
	&\leq \abs{2(\xi|\eta;\varphi)-2(a^{-1}\xi|a^{-1}\eta;\varphi)
	+Q(a,\xi;\varphi\circ\iota)+Q(a,\eta;\varphi)} \vphantom{\frac{1}{2}} \\
	 &\quad + \log k + \log k^\iota, \vphantom{\frac{1}{2}}
\end{align*}
where $k \geq 1$ and $k^\iota \geq 1$ are the error constants from the quasi-conformality of 
$\mu$ and $\mu^\iota$, respectively.
By (\ref{eq:extcalc}), the first term above is bounded by $2A+2B$, where $A,B \geq 0$ are the constants
from Lemma \ref{lem:properties} for $\varphi$. Thus $\lambda'$ is quasi-invariant. 
Proposition \ref{prop.invariance} shows that the density function $\psi$  given by
\[
	\psi(\xi,\eta) \coloneqq \sup_{a\in\Gamma}\frac{da_*\lambda'}{d\lambda'}(\xi,\eta)
\]
provides the required measure $\lambda$ by setting  $d\lambda = \psi\, d\lambda'$.
\end{proof}
We now use the Gibbs current to define a Gibbs state on $GX$. 

\begin{definition}\label{def:state}
Let $\lambda$ be a Gibbs current for $\varphi$ of exponent $\sigma \in \RR$.  The associated \emph{Gibbs state $m$ of exponent $\sigma$}
on $GX \cong \partial_\infty^2 X \times \RR$ is the measure $m$ defined by
\[
	dm(\xi,\eta,t) = d\lambda(\xi,\eta)dt.
\] 
\end{definition}

The definition can be shown to be independent of the choice of Hopf parameterization used to identify $GX$
with $\partial_\infty^2 X \times \RR$. The measure $m$ is $\Gamma$-invariant because $\lambda$ is
$\Gamma$-invariant.  Our terminology will be justified in the next section, where we show that $m$
satisfies the Gibbs property (Definition \ref{def:Gibbspropertyupstairs2}).  The measure $m$ on $GX$
induces a measure on $GX_0$ by the procedure described in Remark \ref{pushdownameasure},  which we also denote $m$ and refer to as a Gibbs state for $\varphi$ of exponent
$\sigma$ on $GX_0$. 

When the critical exponent is finite, we conclude via Proposition \ref{prop:ps} and Proposition \ref{prop:Invariant} that a Gibbs current and associated Gibbs state of exponent $\delta_\varphi$ on $GX$ and $GX_0$ always exist.
\begin{remark}
We do not expect the density of a Gibbs state $m$ with respect to the product measure 
$\mu^\iota \otimes \mu \otimes dt$ to be continuous. This is in contrast with the
Riemannian case or the unweighted $\CAT(-1)$ case. 
\end{remark}

As a consequence of the quasi-product construction of $m$, we have the following result using a Hopf
argument such as \cite[Theorem 2.5]{Kai}.
\begin{lemma}\label{lem:finiteimplies}
	Suppose that $m$ is a Gibbs state for $\varphi$.
	Then if $m$ is conservative (in particular, if $m$ is finite) as a measure on $GX_0$,
	then it is ergodic.
\end{lemma}

\section{The Gibbs property} \label{s.gibbs}
\subsection{The Gibbs property for Gibbs currents}
We start by showing that Gibbs currents satisfy a natural analog of the Gibbs property on 
$\partial^2_\infty X$.
\begin{definition} \
	Suppose $x,y \in X$ and $R > 0$. We define $B(x,y;R) \subseteq \partial^2_\infty X$ to be
	the set of $(\xi,\eta) \in \partial_\infty^2 X$ such that, for some (equivalently, any)
	$c \in GX$ such that $c(-\infty) = \xi$ and $c(\infty) = \eta$,
	there exist $t \leq s$ such that $c(t) \in B(x,R)$ and $c(s) \in B(y,R)$.
\end{definition}

\begin{definition}[Gibbs property for currents]\label{def:gibbscurrent}
	Suppose that $\lambda$ is a $\Gamma$-invariant Radon measure on $\partial_\infty^2 X$. We say that $\lambda$ satisfies the \emph{Gibbs property for $\varphi$ with exponent $\sigma \in \RR$}
if, for each compact set $F \subseteq X$, there exists $R_1 > 0$ such that, for any $R \geq R_1$,
there exists $k'_{F,R} \geq 1$ such that, for any $a \in \Gamma$ and any $x \in F$, $y
\in aF$, we have
\[
	\frac{1}{k'_{F,R}} \leq \frac{\lambda(B(x,y;R))}{e^{\overline{(\varphi-\sigma)}(x,y)}}
	\leq k'_{F,R}.
\] 
\end{definition}

\begin{proposition}[Gibbs currents satisfy the Gibbs property]\label{prop:gibbs}
Suppose that $\lambda$ is a Gibbs current for $\varphi$ of exponent $\sigma \in \RR$.
Then $\lambda$ satisfies the Gibbs property for $\varphi$ with exponent $\sigma$.
\end{proposition}
We establish a series of lemmas to build up a proof of this proposition. 
We define the \emph{shadow of a set $U\subseteq X$ as seen from $x \in X$} to be
\[
	\mathcal{O}_x U \coloneqq \{\xi \in \partial_\infty X \colon [x,\xi) \cap U \neq \emptyset\}
.\] 

\begin{lemma}\label{lem:preMohsen}
Suppose that $\mu$ is a quasi-conformal measure.
Then there exists $k_0 > 0$ such that, for any compact set
$F \subseteq X$, there exists $R_0 > 0$ such that, for any $x \in \bar X$, $y \in F$, we have
\[
	\mu(\mathcal{O}_x B(y,R_0)) \geq k_0
.\] 
\end{lemma}
\begin{proof}
Given a compact set $F \subseteq X$, we can increase $R_0$ to ensure that for all $y \in F$, $B(y,R_0)$ contains a sufficiently large ball around $p$, and it thus suffices to prove the statement for $y=p$ .
Suppose the statement does not hold for $y=p$. Then there exist sequences $R_i \to \infty$ and
$(x_i)_{i\in\NN} \subseteq \bar X$ such that
\[
	\lim_{i \to \infty} \mu(\mathcal{O}_{x_i} B(p,R_i)) = 0.
\] 
We must have $x_i \notin B(p,R_i)$, or else $\mathcal{O}_{x_i}B(p,R_i) =
\partial_\infty X$. We can thus find a subsequence of $(x_i)_{i\in\NN}$ which converges to some
$\xi \in \partial_\infty X$. Following \cite[Lemma 3.10]{PPS}, any relatively compact subset $V$ of
$\partial_\infty X \setminus \{\xi\}$ must be a subset of $\mathcal{O}_{x_i}B(p,R_i)$ for all sufficiently large $i$.
Hence $\mu$ is supported on $\{\xi\}$.
Since the support of $\mu$ is $\Gamma$-invariant, this would imply that
$\xi$ is a fixed point, which is impossible because we assume that $\Gamma$ is non-elementary.
\end{proof}
\begin{lemma}[Mohsen's shadow lemma]\label{lem:Mohsen}
Suppose that $\mu$ is a quasi-conformal measure for $\varphi$ of exponent $\sigma \in \RR$. Then
for any compact set $F \subseteq X$ there exists $R_0 > 0$ such that, for any $R \geq R_0$,
there exists a constant $k_{F,R} \geq 1$ such that,
for any $a \in \Gamma$ and $x \in F$, $y \in aF$, we have
\[
	\frac{1}{k_{F,R}} \leq\frac{\mu(\mathcal{O}_x B(y,R))}{e^{\overline{(\varphi-\sigma)}(x,y)}}
	\leq k_{F,R}
.\] 
\end{lemma}
\begin{proof}
By replacing $\varphi$ with $\varphi-\sigma$, we may assume that $\sigma=0$.
Take $R_0, k_0 > 0$ as provided by Lemma \ref{lem:preMohsen}, and suppose that $R \geq R_0$.
Suppose that $a \in \Gamma$, $x \in F$, and $y \in a F$ are given. Then we have
\begin{align*}
	\mu(\mathcal{O}_xB(y,R)) &= (a^{-1}_*\mu)(\mathcal{O}_{a^{-1}x}B(a^{-1}y,R)) \\
	&=\int_{\mathcal{O}_{a^{-1}x}B(a^{-1}y,R)}\frac{da^{-1}_*\mu}{d\mu}(\xi)d\mu(\xi).
\end{align*}
Writing $A = k^{\pm}B$ for the inequality $k^{-1}B \leq A \leq kB$, we have
\begin{equation}\label{eq:mohsenproof}
\mu(\mathcal{O}_xB(y,R))
= k^{\pm}\int_{\mathcal{O}_{a^{-1}x}B(a^{-1}y,R)} e^{Q(a^{-1},\xi;\varphi)}d\mu(\xi).
\end{equation}

Let $\xi \in \mathcal{O}_{a^{-1}x}B(a^{-1}y,R)$. Take some sequence $z_n \to \xi$ with $z_n \in X$
such that
\[
	Q(a^{-1},\xi;\varphi) = \lim_{n\to\infty}Q(a^{-1},z_n;\varphi)
.\] 
Take any $R' > R$, and let $L = \max_{y\in F}d(y,p)$.
For large enough $n$, the geodesic segment from $a^{-1}x$ to $z_n$ passes
through $B(a^{-1}y,R')$, and thus it passes through $B(p,R'+L)$ since $a^{-1}y \in F$. Using the
properties of $\bphi$ from Lemma \ref{lem:bphi},
\begin{align*}
	\abs{Q(a^{-1},\xi;\varphi)-\bphi(x,y)} &=
	\lim_{n\to\infty}\abs{\bphi(a^{-1}p,z_n)-\bphi(p,z_n)-\bphi(a^{-1}x,a^{-1}y)}\\
	&\leq \limsup_{n\to\infty}\abs{\bphi(a^{-1}x,z_n)-\bphi(p,z_n) -\bphi(a^{-1}x,p)}
	+2K_L+2L\| \varphi \|_\infty \\
	&\leq K_{R'+L}+2K_L+2L\| \varphi \|_\infty \eqqcolon M.
\end{align*}
Hence we have shown that, whenever $\xi \in \mathcal{O}_{a^{-1}x}B(a^{-1}y,R)$, we have
\[
e^{-M}\leq \frac{e^{Q(a^{-1},\xi;\varphi)}}{e^{\bphi(x,y)}}
\leq e^{M}
.\] 
Integrating this inequality over $\xi \in \mathcal{O}_{a^{-1}x}B(a^{-1}y,R)$,
and using (\ref{eq:mohsenproof}) along with
the fact that $\mu$ is a probability measure, we obtain
\[
	k^{-1}e^{-M}\mu(\mathcal{O}_{a^{-1}x}B(a^{-1}y,R))
	\leq \frac{\mu(\mathcal{O}_xB(y,R))}{e^{\bphi(x,y)}} 
	\leq ke^{M}
.\] 
Finally, by Lemma \ref{lem:preMohsen}, the term $\mu(\mathcal{O}_{a^{-1}x}B(a^{-1}y,R))$ is bounded below
by $k_0$.
\end{proof}
The following is essentially the same statement as \cite[Lemma 3.17]{PPS}, although we use arguments from coarse geometry.
\begin{lemma}\label{lem:productsets}
For any $R' \geq 0$ there exist $R \geq R'$ and $T$ such that,
for any $x,y\in X$ such that $d(x,y) \geq T$, we have
\[
	\mathcal{O}_yB(x,R') \times \mathcal{O}_xB(y,R') \subseteq B(x,y;R)
.\] 
Additionally, for any $R \geq 0$, we have
\[
	B(x,y;R) \subseteq \mathcal{O}_yB(x,2R)\times\mathcal{O}_xB(y,2R)
.\] 
\end{lemma}
\begin{proof}
Let $R' \geq 0$ be given,
and suppose that $(\xi,\eta) \in \mathcal{O}_yB(x,R') \times \mathcal{O}_xB(y,R')$.
The key step in proving the first inclusion is to show that there are constants $\kappa, \epsilon$ such that, for any $L \geq 0$, 
there is $T \geq 0$ such that if $d(x,y) \geq T$ then 
there is a $(\kappa,\epsilon,L)$-local quasi-geodesic $\gamma$ joining $\xi$
to $\eta$ passing through $B(x,R')$ and then $B(y,R')$. See Figure \ref{fig:seesaw}. Let $x'$ be any point in $[y,\xi)\cap B(x,R')$ and $y'$ any point in
$[x,\eta) \cap B(y,R')$. We define $\gamma$ to be the path obtained by taking the geodesic ray from $\xi$ to $y$ and then jumping to the geodesic ray from $y'$ to $\eta$. Everywhere except for the jump point, $\gamma$ has the geodesic property, so in order to verify it is a
$(\kappa,\epsilon,L)$-local quasi-geodesic, we only have to show that, for any $z \in [y,\xi)$
such that $d(z,y) \leq L$, and any $w \in [y',\eta)$ such that $d(y',w) \leq L$, then
\[
	\kappa^{-1}d(z,w)-\epsilon \leq d(z,y)+d(y',w) \leq \kappa d(z,w)+\epsilon
.\] 
Given $L \geq 0$, we let $T = L+R'$ and suppose that $d(x,y) \geq T$. Note that 
\[
d(z,y) \leq L = T-R' \leq d(x,y)-R' \leq d(x',y)
.\]
Hence $z$ lies between $x'$ and $y$. Now, by
convexity there is some $z' \in [x,y]$ such that $d(z,z') \leq d(x,x') \leq R'$. Likewise, by convexity
there exists some point $z'' \in [x,y']$ such that $d(z',z'')\leq d(y,y') \leq R'$. Thus
\[
	d(z,y)+d(y',w) \leq d(z'',y')+d(y',w)+3R' = d(z'',w)+3R' \leq d(z,w)+5R'
,\] 
and similarly we have $d(z,y)+d(y',w)\geq d(z,w)-5R'$. This shows that, for any $L \geq 0$,
if we let $T = L+R'$ and suppose that $d(x,y) \geq T$, then $\gamma$ is a
$(1,5R',L)$-local quasi-geodesic from $\xi$ to $\eta$ passing through $B(x,R')$ and then $B(y,R')$.

Fixing $L$ large enough, then by
Lemma \ref{lem:localquasi}, our local quasi-geodesic $\gamma$ is actually a global
$(\kappa',\epsilon')$-quasi-geodesic. Furthermore, the quasi-geodesic $\gamma$
passes through $x' \in B(x,R')$ before it passes through $y' \in B(y,R')$.
Using Lemma \ref{lem:infinitestability}, 
there must exist some $R \geq R'$ depending only on $\kappa'$, $\epsilon'$, and $R'$ such that,
if $c \in GX$ is a geodesic line satisfying $c(-\infty) = \xi$ and $c(\infty) = \eta$,
then there exist $t \leq s$ such that $c(t) \in B(x,R)$ and $c(s) \in B(y,R)$.
Hence $(\xi,\eta) \in B(x,y;R)$. This completes the proof of the first inclusion.

\begin{figure}
\begin{tikzpicture}[scale=.4]
	\path (-10,0) coordinate (x) (-10,2) coordinate (x') (10,0) coordinate (y) (10,2) coordinate (y')
		(-14,4) coordinate (xi) (14,4) coordinate (eta);
	\draw (x) -- (y)
		node [pos=.83] (z') {}
		node [pos=.83,below,xshift=-.1cm] {$z'$};
	\draw[dashed] (x) circle [radius=3];
	\draw[dashed] (y) circle [radius=3];
	\draw[dashed] (x) -- ($(x) + (-135:3)$)
		node [pos=.6, xshift=-.16cm,yshift=.13cm] {\small $R'$};
	\draw[dashed] (y) -- ($(y) + (-45:3)$)
		node [pos=.6, xshift=.2cm,yshift=.13cm] {\small $R'$};
	\draw[red] (y) .. controls (5,.3) and (-5,1) ..  (x')
		node [pos=0.2] (z) {}
		node [pos=0.2,above] {$z$};
	\draw[red] (x') .. controls (-10.5,2.08) and (-13,2.9) ..  (xi);
	\draw[black] (x) .. controls (-5,.3) and (5,1) ..  (y')
		node [pos=0.8] (z'') {}
		node [pos=0.8,above] {$z''$};
	\draw[red] (y') .. controls (10.5,2.08) and (13,2.9) ..  (eta)
		node [pos=0.7] (w) {}
		node [pos=0.7,xshift=.2cm,yshift=-.15cm] {$w$};
	\filldraw (x) circle (0.1cm);
	\filldraw (y) circle (0.1cm);
	\filldraw (x') circle (0.1cm);
	\filldraw (y') circle (0.1cm);
	\filldraw[red] (z) circle (0.1cm);
	\filldraw[red] (w) circle (0.1cm);
	\filldraw (z') circle (0.1cm);
	\filldraw (z'') circle (0.1cm);
	\node[below] at (x) {$x$};
	\node[below] at (y) {$y$};
	\node[xshift=-.25cm,yshift=-.15cm] at (x') {$x'$};
	\node[xshift=.22cm,yshift=-.15cm] at (y') {$y'$};
	\node[left] at (xi) {$\xi$};
	\node[right] at (eta) {$\eta$};
\end{tikzpicture}
\caption{The set $B(x,y, R)$ and a local quasi-geodesic}
\label{fig:seesaw}
\end{figure}

For the second inclusion, suppose that $(\xi,\eta) \in B(x,y;R)$.
Let $x'' \in B(x,R)$ and $y'' \in B(y,R)$
be intersection points with the geodesic from $\xi$ to $\eta$.
By convexity, there exists a point $x' \in
[y,\xi)$ such that $d(x',x'')\leq R$, and therefore $d(x',x) \leq 2R$. Hence $[y,\xi)$ intersects
$B(x,2R)$, so $\xi \in \mathcal{O}_yB(x,2R)$. Similarly, $\eta \in \mathcal{O}_xB(y,2R)$. Hence
$(\xi,\eta) \in \mathcal{O}_yB(x,2R) \times \mathcal{O}_xB(y,2R)$.
\end{proof}
The following lemma gives a useful bound on $e^{-2(\xi|\eta;\varphi-\sigma)}$, which is needed to relate estimates involving $\mu$ and $\mu^\iota$ with estimates involving $\lambda$.
\begin{lemma}\label{lem:Gromovbound}
For any compact set $F\subseteq X$ and any $R \geq 0$,
there is a constant $M \geq 0$ such that, if $x \in F$ and $y\in X$, then
$|(\xi|\eta;\varphi)| \leq M$ for all $(\xi,\eta) \in B(x,y;R)$.
\end{lemma}
\begin{proof}
Let $L = \max_{x\in F} d(x,p)$, and suppose that $x \in F$ and $(\xi,\eta) \in B(x,y;R)$. Observe that the 
geodesic from $\xi$ to $\eta$ comes within $R+L$ of $p$. Take any $R' > R+L$, and
take sequences $x_n,y_n \in X$ with, $x_n \to \xi$, $y_n \to \eta$, and
$(x_n|y_n;\varphi) \to (\xi|\eta;\varphi)$.
Then for large $n$, the geodesic from $x_n$ to $y_n$
comes within $R'$ of $p$. Hence, using the roughly geodesic property, we have
\[
	\abs{(\xi|\eta;\varphi)} =
	\lim_{n\to\infty}\frac{1}{2}\abs{\bphi(x_n,p)+\bphi(p,y_n)-\bphi(x_n,y_n)}
	\leq \frac{1}{2}K_{R'}
. \qedhere\] 
\end{proof}
\begin{proof}[Proof of Proposition \ref{prop:gibbs}]
By replacing $\varphi$ with $\varphi - \sigma$, we may assume that $\sigma = 0$.
Choose a constant $R_0 > 0$ so that Lemma \ref{lem:Mohsen} holds for $\mu$ and
Lemma \ref{lem:preMohsen} holds for $\mu^\iota$.
By Lemma \ref{lem:productsets}, we can find $R_1 \geq R_0$ and $T \geq 0$ so that
\[
	\mathcal{O}_yB(x,R_0)\times\mathcal{O}_xB(y,R_0)\subseteq B(x,y;R_1)
\]
whenever $d(x,y) \geq T$.
Now suppose that $R \geq R_1$ is given. Using Lemma \ref{lem:productsets} again, we have
\[
	B(x,y;R)\subseteq\mathcal{O}_yB(x,2R)\times\mathcal{O}_xB(y,2R)
.\] 
Let $F \subseteq X$ be compact, and let $M \geq 0$ be the constant guaranteed by 
Lemma \ref{lem:Gromovbound} associated to $F$ and $R$. Recall from
the definition of $\lambda$ that there is some $k_1 \geq 1$ such that
\[
	k_1^{-1} \leq \frac{d\lambda}{d\lambda'}(\xi,\eta) \leq k_1,
\] 
where $d\lambda'(\xi,\eta) = e^{-2(\xi|\eta;\varphi)}d\mu^\iota(\xi)d\mu(\eta)$, and $\mu$, $\mu^\iota$
are quasi-conformal measures for $\varphi$, $\varphi\circ\iota$ of exponent $0$.
Thus for the upper bound, using Lemma \ref{lem:Mohsen} and the fact that $\mu^\iota$
is a probability measure, we have
\[
	\frac{\lambda(B(x,y;R))}{e^{\bphi(x,y)}} \leq
	k_1e^{2M}\mu^\iota(\mathcal{O}_y B(x,2R))\frac{\mu(\mathcal{O}_xB(y,2R))}{e^{\bphi(x,y)}}
	\leq k_1e^{2M}k_{F,2R}
.\] 

For the lower bound, we have two cases. If $d(x,y) \geq T$, then we know that
$\mathcal{O}_yB(x,R_0) \times \mathcal{O}_x B(y,R_0) \subseteq B(x,y;R)$, and thus we have
\[
	\frac{\lambda(B(x,y;R))}{e^{\bphi(x,y)}} \geq
	k_1^{-1}e^{-2M}\mu^\iota(\mathcal{O}_yB(x,R_0))\frac{\mu(\mathcal{O}_xB(y,R_0))}{e^{\bphi(x,y)}}
	\geq k_1^{-1}e^{-2M}k_0 k_{F,R_0}^{-1}
.\]
If $d(x,y) < T$, then choose some $y'$ on some geodesic ray that starts at $x$ and passes through $y$ so
that $d(x,y') = T$. By convexity, we have $B(x,y';R) \subseteq B(x,y;R)$. Let $F'$ 
be the closed neighborhood of $F$ of radius $T$. Then $y' \in F'$.
Note that $F'$ only depends on $F$ and $R_0$, since $T$ only depends on $R_0$.
By using the lower bound we obtained in the first case, we have
\[
	\frac{\lambda(B(x,y;R))}{e^{\bphi(x,y)}} \geq 
	\frac{\lambda(B(x,y';R))}{e^{\bphi(x,y')}} e^{-\|\varphi\|_{\infty}T}
	\geq k_1^{-1}e^{-2M}k_0k_{F',R_0}^{-1}e^{-\|\varphi\|_\infty T}.
\qedhere\] 
\end{proof}

\subsection{The dynamical Gibbs property for Gibbs states}

We now show how the dynamical Gibbs property for $m$ follows from the Gibbs
property for $\lambda$. We restate Definition \ref{def:Gibbspropertyupstairs2} using the notation of our setting.
\begin{definition} \label{def:Gibbsproperty}
Given a flow-invariant, $\Gamma$-invariant measure $m$ on $GX$, we say that $m$ satisfies the
\emph{Gibbs property for $\varphi$ with exponent $\sigma \in \RR$} if,
for any compact set $F \subseteq
GX$, there exists $R_2 > 0$ such that, for any $R \geq R_2$, there is a constant $k''_{F,R}\geq 1$ such
that, whenever $a \in \Gamma$ and $c \in GX$ are such that $c \in F$ and $g_tc \in a F$, then
\[
	\frac{1}{k''_{F,R}}
	\leq\frac{m(B_t(c,R))}{e^{\int_0^t(\varphi(g_sc)-\sigma)\,ds}}
	\leq k''_{F,R}
.\] 
\end{definition}
\begin{proposition}[Gibbs states satisfy the Gibbs property]\label{prop:dynamicalGibbs}
If $m$ is a Gibbs state on $GX$ for $\varphi$ of exponent $\sigma \in \RR$,
then $m$ satisfies the Gibbs property for $\varphi$ with exponent $\sigma$.
\end{proposition}
To prove this, first we compare Bowen balls $B_t(c,R)$ with sets of
the form $B(x,y;R') \times (-R',R')$ under the \emph{Hopf parametrization with
basepoint x}: this is the identification $c \sim (c(-\infty),c(+\infty),t)$,
where $t$ is chosen to minimize $d(c(t),x)$.
Since $m$ doesn't depend on the identification, to bound $m(B_t(c,R))$ it
will suffice to show these comparisons for this Hopf parametrization.
\begin{lemma}\label{lem:Bowenballs}
For any $R' \geq 0$, there exists $R \geq 0$ such that,
for any $x,y \in X$ and any extension $c \in \mathcal{C}([x,y])$, writing
$t = d(x,y)$, then
\[
	B(x,y;R') \times (-R',R') \subseteq B_t(c,R)
\] 
using the Hopf parametrization with basepoint $x$.
Furthermore, for any $R \geq 0$, there exists $R'' \geq 0$ such that, for any $c \in GX$
and any $t \geq 0$, then letting $x = c(0)$ and $y = c(t)$, we have
\[
	B_t(c,R) \subseteq B(x,y;R'')\times (-R'',R'')
.\] 
\end{lemma}
\begin{proof}
For the first statement, suppose that $R' \geq 0$ is given. By Lemma \ref{lem:footprint},
we can choose $R \geq 0$ so that $d(c_1(0),c_2(0)) \leq 4R'$ implies
$d_{GX}(c_1,c_2) \leq R$ for all $c_1,c_2 \in GX$. Now suppose that $x,y \in X$, $t = d(x,y)$, and $c
\in \mathcal{C}([x,y])$. Let
$c' = (\xi,\eta,t_0) \in B(x,y;R') \times (-R',R')$ be given.
Since $c'(t_0)$ is the closest point on $c'$ to
$x$, and since $c'$ passes through $B(x,R')$, we have $d(c'(t_0),x) \leq R'$,
and hence 
\[
	d(c'(0),x) \leq d(c'(t_0),x) + \abs{t_0} \leq 2R'
.\]
Furthermore, since $c'$ passes through $B(y,R')$, there is some $t_1 \in
\RR$ such that $d(c'(t_1),y) \leq R'$.  By the triangle inequality, we must have $\abs{t_1-t} \leq d(c'(0),x) +
d(c'(t_1),y) \leq 3R'$, and hence 
\[
	d(c'(t),y) \leq d(c'(t_1),y)+\abs{t-t_1} \leq 4R'
.\]
We have just shown that $d(c'(0),c(0))$ and $d(c'(t),c(t))$ are
both at most $4R'$. By convexity, $d(c'(s),c(s)) \leq 4R'$ for all $s \in [0,t]$.
By the definition of $R$ we have $d_{GX}(g_sc',g_sc) \leq R$ for all $s \in [0,t]$,
and so we have $c' \in B_t(c,R)$.

Now fix $R\geq 0$. Using that $\pi_{\mathrm{fp}}$ is a quasi-isometry, choose $R''$ large enough so that $d_{GX}(c_1,c_2) \leq R$ implies $d(c_1(0),c_2(0)) \leq \frac{1}{2}R''$
for all $c_1,c_2 \in GX$. Fix $c \in GX$ and $t \geq 0$, and set $x = c(0)$ and $y = c(t)$.
Let  $c' \in B_t(c,R)$; then $d_{GX}(c,c') \leq R$ and $d_{GX}(g_tc,g_tc') \leq R$. By the
definition of $R''$, this implies $d(x,c'(0)) \leq \frac{1}{2}R''$ and $d(y,c'(t)) \leq \frac{1}{2}R''$.
Therefore, if we write $c' = (\xi,\eta,t_0)$, then $(\xi,\eta) \in B(x,y;\frac{1}{2}R'') \subseteq
B(x,y;R'')$. To control $t_0$, note that by the definition of the Hopf parametrization,
$d(x,c'(t_0)) \leq d(x,c'(0)) \leq \frac{1}{2}R''$, and hence 
$\abs{t_0} \leq d(x,c'(t_0)) + d(x,c'(0)) \leq R''$. This shows that $c' \in B(x,y;R'')\times (-R'',R'')$.
\end{proof}

\begin{proof}[Proof of Proposition \ref{prop:dynamicalGibbs}]
By replacing $\varphi$ with $\varphi-\sigma$, we may assume that $\sigma = 0$.
Let $F' = \pi_{\mathrm{fp}}(F)$. Then $F'$ is a compact subset of $X$.
Let $x = c(0)$ and $y = c(t)$, so that $x \in F'$ and $y \in aF'$.
Write $dm = d\lambda \,dt$, where $\lambda$ is a Gibbs current for $\varphi$ of exponent $0$,
and let $R_1 > 0$ be the constant from
the Gibbs property for $\lambda$ (see Definition \ref{def:gibbscurrent}). Let $R_2 \geq R_1$
be guaranteed by Lemma \ref{lem:Bowenballs} applied with $R' = R_1$, so that
\[
	B(x,y;R_1) \times (-R_1,R_1) \subseteq B_t(c,R_2)
.\]
Suppose that $R \geq R_2$. Then by Lemma \ref{lem:Bowenballs} we obtain $R'' \geq 0$ so that 
\[
	B_t(c,R) \subseteq B(x,y;R'')\times (-R'',R'').
\]
Using the Gibbs property for $\lambda$, we have
\[
	(k'_{F',R_1})^{-1}2R_1 \leq \frac{m(B_t(c,R))}{e^{\bphi(x,y)}} \leq k'_{F',R''}2R''
.\] 
Furthermore, by Lemma \ref{lem:integrals}, there is a uniform constant $K\geq 0$
such that
\[
	\abs{\int_0^t \varphi(g_s c) \, ds - \bphi(x,y)} \leq K.
\] 
This implies that
\[
	(k'_{F',R_1})^{-1}2R_1e^{-K} \leq \frac{m(B_t(c,R))}{e^{\int_0^t\varphi(g_sc)\,ds}}
	\leq k'_{F',R''}2R''e^K
. \qedhere
\]
\end{proof}
\begin{proof} [Proof of Theorem \ref{maintheorem}]
Let $m_\varphi$ be a Gibbs state of exponent $\delta_\varphi$ provided
by Proposition \ref{prop:ps}, Proposition \ref{prop:Invariant}, and Definition \ref{def:state}.
The measure $m_{\varphi}$ is a quasi-product measure by construction. The measure
satisfies the Gibbs property with exponent $\delta_{\varphi}$
by Proposition \ref{prop:dynamicalGibbs}. The measure is Radon by the Gibbs property. By Lemma \ref{lem:finiteimplies}, it is ergodic when it is conservative.
It is a consequence of Lemma \ref{lem:limitset} and the construction of $m_{\varphi}$ that,
considered on $GX_0$, it is fully supported on the non-wandering set $\Omega X_0$.
This concludes our proof of Theorem \ref{maintheorem}.
\end{proof}


\section{The Ledrappier-Otal-Peign\'e Partition} \label{sec:op}
In this section, we prove Theorem \ref{thmx:op}. We recall some background on the entropy theory of
measurable partitions in the general setting of a homeomorphism $f : Z \to Z$ on a
complete separable metric space $(Z,d_Z)$, see \cite{parry,rokhlin, L13}. In our context, we will usually
apply this with $Z=GX_0$ and $f = g_\tau$, where $\tau$ is fixed. Let $\nu$ be an $f$-invariant Borel
probability measure on $Z$.

Let $\zeta$ be a partition of $Z$. We require that $\zeta$ is defined everywhere rather than as a
partition ``mod 0''. This is important for us because in \S\ref{s.thermodynamic} we will often consider a single partition in computations that involve two mutually singular measures. Our operations on partitions are defined everywhere unless we specify explicitly that they are ``mod $0$''. This differs from much of the standard literature where partitions are defined mod $0$.

 For $z \in Z$, we let $\zeta(z)$ denote the unique element of
$\zeta$ containing $z$. For $Y \subseteq Z$, we let $\zeta|_Y$ be the partition of $Y$ with
	$(\zeta|_Y)(z) = \zeta(z) \cap Y$ for every $y \in Y$. We write $f\zeta = \{f(A)\}_{A \in \zeta}$.
For $n \leq m \in \ZZ \cup \{-\infty, +\infty \}$, we use the notation
\begin{equation} \label{partitionconvention}
\zeta_n := f^{-n} \zeta, \hspace{20pt} \zeta^m_{n} \coloneqq \bigvee_{i=n}^{m} f^{-i}\zeta.
\end{equation}
In particular, the partition $\zeta^0_{-\infty}$ can be thought of as the partition into the
``present and infinite past as seen by $\zeta$".

Given two partitions $\zeta, \zeta'$ of $Z$, we write $\zeta \prec \zeta'$ if $\zeta'(z)
\subseteq \zeta(z)$ for every $z \in Z$.
We use the notation $\zeta \prec \zeta'$ ($\mathrm{mod}\,\nu$) if there is some $\nu$-measurable
set $A \subseteq GX_0$ with $\nu(A) = 1$ such that $\zeta'(c) \cap A \subseteq \zeta(c) \cap A$
for every $c \in A$.

	Given a partition $\zeta$ of $Z$, a set $B \subseteq Z$ is said to be \emph{$\zeta$-saturated} if
	$z \in B$ implies $\zeta(z) \subseteq B$ for every $z \in B$. We say that $\zeta$ is \emph{measurable} if there exists a countable collection
$(B_n)_{n\in\NN}$ of $\zeta$-saturated Borel sets $B_n \subseteq Z$  such that for all 
$A \neq A'$ in $\zeta$, then there exists some $n\in\NN$ such that either $A \subseteq B_n$
and $A' \subseteq B_n^c$, or vice versa. In particular, if $\zeta$ is measurable, then $\zeta(z)$ is a
Borel subset of $Z$ for every $z \in Z$.

If $\zeta$ is a measurable partition of $Z$, then there exists
a \emph{disintegration} of $\nu$ with respect to $\zeta$, denoted $(\nu_A)_{A\in\zeta}$, which
is a measurable family of probability measures $\nu_{\zeta(z)}$ defined on $\zeta(z)$
for $\nu$-a.e.\ $z \in Z$ and satisfying 
\[
	\int f \, d\nu = \int\int f|_{\zeta(z)}\,d\nu_{\zeta(z)}\,d\nu(z)
\] 
for any $\nu$-integrable $f : Z \to \RR$. Disintegrations are unique in the sense that, if
$(\nu'_A)_{A\in\zeta}$ is any other disintegration of $\nu$ with respect to $\zeta$,
then $\nu'_{\zeta(z)} = \nu_{\zeta(z)}$ for $\nu$-a.e.\ $z \in Z$.

Given two measurable partitions $\zeta, \zeta'$ of $Z$, the \emph{conditional entropy}
$H_\nu(\zeta'|\zeta)$ is defined by
\[
	H_\nu(\zeta'|\zeta) \coloneqq \int -\log \nu_{\zeta(z)}(\zeta'(z)\cap \zeta(z))\,d\nu(z)
.\] 
We define $H_\nu(\zeta) = H_\nu(\zeta | \tau)$, where $\tau$ is the trivial
partition  $\tau = \{Z\}$. 
The \emph{(measure-theoretic) entropy of $\nu$ with respect to $f$ and $\zeta$} is defined by
\[
	h_\nu(f,\zeta) \coloneqq H_\nu(\zeta_1 | \zeta^0_{-\infty}). 
\]
When $\zeta$ is a countable measurable partition with
$H_\nu(\zeta) < \infty$, this agrees with the standard definition of $h_\nu(f,\zeta)$ as
in \cite{VO}.

The \emph{(measure-theoretic) entropy of $\nu$ with respect to $f$} is then defined by
\[
h_\nu(f) = \sup \{ h_\nu(f,\zeta) : \zeta \text{ is a measurable partition of } Z\}.
\] 
This supremum is unchanged if restricted to countable measurable partitions $\zeta$ 
with $H_\nu(\zeta) < \infty$, or finite partitions,
and hence this definition agrees with the standard definition of entropy as in \cite{VO, Wa}.

We call a partition $\zeta$ of $Z$ \emph{decreasing} if $\zeta \prec \zeta_1$.
We emphasize that our definition of decreasing
asks for set-theoretic refinement rather than refinement up to a set of
zero measure.  If $\zeta$ is a decreasing partition,
then $\zeta^0_{-\infty} = \zeta$, and the entropy formula becomes
\[
	h_\nu(f,\zeta) = H_\nu(\zeta_1|\zeta) = \int-\log \nu_{\zeta(z)}(\zeta_1(z))\,d\nu(z)
.\]
If $\zeta$ is any measurable partition, then $\zeta^0_{-\infty}$ is decreasing,
and $h_\nu(f,\zeta) = h_\nu(f,\zeta^0_{-\infty})$.

We say that $\zeta$ is (one-sided) \emph{generating with respect to $\nu$} if there is some $A \subseteq
Z$ with $\nu(A) = 1$ for which $(\zeta_1^\infty)|_A$ is the partition into points. If $\zeta$ is generating with respect to $\nu$, without an additional assumption that $H_\nu(\zeta) < \infty$, it is not always true that $h_\nu(f,\zeta) = h_\nu(f)$. If  $\zeta$ is generating with respect to $\nu$ and $\mathcal{P}$ is a measurable partition of $Z$ with $H_\nu(\mathcal{P}) < \infty$, then we always have $\lim_{n\to\infty}H_\nu(\mathcal{P}|\zeta_1^n) = 0$.

From now on, we let $Z=GX_0$, and we let $\nu$ be an ergodic $(g_t)$-invariant measure on $GX_0$.
We fix $\tau > 0$ such that $g_\tau$ is ergodic with respect to $\nu$, and let $g \coloneqq g_{\tau}$. 
We say that a measurable partition $\zeta$ of $GX_0$ is \emph{$\nu$-subordinated to
$\mathcal{W}^{\mathrm{uu}}$} if $\zeta(c)$ is a relatively compact neighborhood of 
$c$ in $W^{\mathrm{uu}}(c)$ for $\nu$-a.e.\ $c \in GX_0$.
We restate Theorem \ref{thmx:op}, which is the main result of this section. 

\begin{proposition}\label{prop:op}
There exists a decreasing measurable partition $\zeta$ of $GX_0$ which is $\nu$-subordinated to
$\mathcal{W}^{\mathrm{uu}}$
and satisfies $\lim_{n \to \infty} \diam(\zeta_n(c)) = 0$ for $\nu$-a.e.\ ${c \in GX_0}$.
Furthermore, if $\Omega X$ has finite upper box dimension with respect to
$d_{GX}$, then we have $h_\nu(g) = h_\nu(g,\zeta)$.
\end{proposition}
\begin{remark}
The condition that $\lim_{n \to \infty} \diam(\zeta_n(c)) = 0$ for $\nu$-a.e.\ $c \in GX_0$ implies that
$\zeta$ is one-sided generating with respect to $\nu$. In \cite{thesis}, partitions satisfying this
stronger property are called \emph{metrically generating with respect to $\nu$}. We require this stronger property in \S\ref{s.thermodynamic} 
to show uniqueness of $m_\varphi$ as an
equilibrium state. 
\end{remark}
\begin{remark}\label{rem:balls}
	Proposition \ref{prop:op} was proved by Otal-Peign\'e \cite{OP04} in the pinched negative
	curvature manifold case.	The construction in \cite{OP04} begins with a dynamical `cellule.'
	Since the cellule is defined using strong stable and unstable leaves,
	Otal and Peign\'e  require
	H\"older continuity of the strong stable and unstable distributions
	to obtain estimates involving the
	distance of a point to the boundary of the cellule. However, by starting the construction with a ball rather than a cellule, as in \cite{L13, LS}, this part of the argument is replaced by the elementary
	estimate \eqref{eq:estimates}. This is the approach we take, and it sidesteps any need for H\"older continuity.
	\end{remark}

We start our proof of Proposition \ref{prop:op}.  The sets $\mathrm{Fix}_n(GX_0)$ defined in \S \ref{sec:isometriesetc}  are clearly flow-invariant. 
Since $\nu$ is ergodic, we can fix $N$ for which
$\nu(\mathrm{Fix}_N(GX_0)) = 1$. We fix some $c_0 \in \supp\nu \,\cap\, \mathrm{Fix}_N(GX_0)$
and a lift $\tilde{c}_0$ of $c_0$ to $GX$.  Suppose that $r \leq \epsilon(c_0)$,  so that $B(c_0,r)$
is identified with $\mathrm{Stab}_\Gamma(\tilde{c}_0)\backslash B(\tilde{c}_0,r)$ by Remark
\ref{rem:quotientball}.
Consider the partition $\mathcal{W}^{\mathrm{uu}}|_{B(\tilde{c}_0,r)}$ of
	$B(\tilde{c}_0,r)$ into its local strong unstable sets,
\[
	\mathcal{W}^{\mathrm{uu}}|_{B(\tilde{c}_0,r)} = 
	\{\su(\tilde{c}) \cap B(\tilde{c}_0,r) \colon \tilde{c} \in B(\tilde{c}_0,r)\}
.\]
Since $\mathrm{Stab}_\Gamma(\tilde{c}_0)$ preserves $\mathcal{W}^{\mathrm{uu}}|_{B(\tilde{c}_0,r)}$, 
we obtain a well-defined partition of $B(c_0,r)$ by quotienting 
$\mathcal{W}^{\mathrm{uu}}|_{B(\tilde{c}_0,r)}$ by $\mathrm{Stab}_\Gamma(\tilde{c}_0)$.
We denote this quotient partition by $\hat{\zeta}^r$, and we call the elements of $\hat{\zeta}^r$
the \emph{local strong unstable sets inside of} $B(c_0,r)$.
The \emph{local strong stable}, \emph{local weak unstable}, and
\emph{local weak stable sets inside of} $B(c_0,r)$ are defined similarly.

We extend $\hat{\zeta}^r$ to a partition of $GX_0$ by including the
complement of $B(c_0,r)$, and we let
\[
\zeta^r := (\hat{\zeta}^r)^0_{-\infty}.
\] 
It is not hard to check that
$\hat{\zeta}^r$ is measurable. Hence $\zeta^r$ is
decreasing and measurable. We show that, for
Lebesgue-a.e.\ small enough $r > 0$, the partition $\zeta^r$
satisfies all of the other properties in Proposition \ref{prop:op}.
\begin{lemma} \label{lem:strictdecrease} For any $0 < r\leq \epsilon(c_0)$, and $\nu$-almost every $c \in GX_0$, the partition $\zeta^r = (\hat{\zeta}^r)^0_{-\infty}$ satisfies
\[
\lim_{n \to \infty} \diam(\zeta^r_n(c)) = 0.\]
\end{lemma}
\begin{proof}
Because $c_0$ is in the support of $\nu$, and $\nu$ is ergodic,
then for $\nu$-a.e.\ $c \in GX_0$, there is a sequence $n_k \to \infty$ for which
$g^{n_k}c \in B(c_0,r)$. Then for each $k$, 
\[
	\zeta^r_{n_k}(c) \subseteq \hat{\zeta}^r_{n_k}(c) = g^{-n_k}(\hat{\zeta}^r(g^{n_k}c))
.\] 
Since $\hat{\zeta}^r(g^{n_k}c)$ is a local strong unstable set 
inside $B(c_0,r)$, we may take a lift $\tilde{c} \in GX$ of
$c$ such that $\hat{\zeta}^r(g^{n_k}c)$ lifts to $\wuu(g^{n_k}\tilde{c}) \cap B(\tilde{c}_0,r)$.
This lift of $\hat{\zeta}^r(g^{n_k}c)$ has diameter at most $2r$ and is contained in a strong unstable set
of $GX$. Thus
$\lim_{k\to\infty}\mathrm{diam}(g^{-n_k}(\hat{\zeta}^r(g^{n_k}c))) = 0$, because the contraction
of the unstable subsets of $GX$ is uniform. Hence $\lim_{k\to\infty}\diam(\zeta^r_{n_k}(c)) = 0$.
The conclusion follows since $\zeta^r$ is decreasing.
\end{proof}

\begin{lemma} \label{lem:subordinate}
	There exists $\rho > 0$ such that, for Lebesgue-a.e.\ $r \in (0,\rho)$, the partition $\zeta^r = (\hat{\zeta}^r)^0_{-\infty}$
is $\nu$-subordinated to $\mathcal{W}^\mathrm{uu}$.
\end{lemma}
\begin{proof}
For $\nu$-a.e.\ $c \in GX_0$,
there is some $n \geq 0$ such that $g^{-n}c \in B(c_0,r)$. Since
$\zeta^r(c) \subseteq \hat{\zeta}^r_{-n}(c) = g^n(\hat{\zeta}^r(g^{-n}c))$,
and $\hat{\zeta}^r(g^{-n}c)$ is a relatively compact subset of
$\su(g^{-n}c)$, it follows that $\zeta^r(c)$ is a relatively compact subset of $W^\mathrm{uu}(c)$.

We show how to choose $r$ to ensure that $\zeta^r(c)$ is a neighborhood of $c$ inside of $\su(c)$ for $\nu$-a.e.\ $c \in GX_0$. Let $\rho \leq \epsilon(c_0)$ be chosen small enough that, for each $c \in B(c_0,\rho)$,
if $c'\in \hat{\zeta}^\rho(c)$, then $B^+(c',\rho) \subseteq \hat{\zeta}^{\epsilon(c_0)}(c)$.
With this choice of $\rho$, if $r \in (0,\rho)$ and $c,c' \in B(c_0,r)$
are in the same strong unstable set, then
$d^+(c,c') < \rho$ implies that $\hat{\zeta}^r(c) = \hat{\zeta}^r(c')$. Indeed, if $d^+(c,c') < \rho$,
then we have $c \in B^+(c',\rho) \subseteq
\hat{\zeta}^{\epsilon(c_0)}(c')$, and hence 
$\hat{\zeta}^{\epsilon(c_0)}(c) = \hat{\zeta}^{\epsilon(c_0)}(c')$.
However, since $c,c' \in B(c_0,r)$, then
\[
	\hat{\zeta}^r(c) = \hat{\zeta}^{\epsilon(c_0)}(c) \cap B(c_0,r) 
	= \hat{\zeta}^{\epsilon(c_0)}(c') \cap B(c_0,r) =
	\hat{\zeta}^r(c')
.\]

For each $r \in (0,\rho)$ define a function
\[
	\beta_r(c) \coloneqq \min\left\{\inf_{n\geq 0}\{\tfrac{1}{2}e^{n\tau}d^+(g^{-n}c,\partial
	B(c_0,r))\},\rho\right\}
.\] 
Let us show that $B^+(c,\beta_r(c)) \subseteq \zeta^r(c)$ for any $c \in GX_0$ and any $r \in (0,\rho)$. 
Suppose that $c' \in B^+(c,\beta_r(c))$, and let $n \geq 0$. We know that
\[
	d^+(g^{-n}c,g^{-n}c') < e^{-n\tau}\beta_r(c)
	\leq \frac{1}{2}d^+(g^{-n}c,\partial B(c_0,r)),
\]
and hence either $g^{-n}c$ and $g^{-n}c'$ are both in $B(c_0,r)$ or they are both in $B(c_0,r)^c$. If the
latter is true, then $\hat{\zeta}^r(g^{-n}c) = \hat{\zeta}^r(g^{-n}c')$ by definition. If the former is true,
then we have $\hat{\zeta}^r(g^{-n}c) = \hat{\zeta}^r(g^{-n}c')$ since 
$d^+(g^{-n}c,g^{-n}c') < e^{-n\tau}\rho \leq \rho$.
Therefore $\hat{\zeta}^r(g^{-n}c) = \hat{\zeta}^r(g^{-n}c')$ for each $n \geq 0$, and so
$c' \in \zeta^r(c)$.

Thus it suffices to show that, for Lebesgue-a.e.\ $r \in (0,\rho)$, the function
$\beta_r$ is positive $\nu$-a.e.
We use the following fact from measure theory \cite{OP04}:
for any Borel probability measure $\mu$ on $\RR$ and any $a \in (0,1)$,
then Lebesgue-a.e.\ $r \in \RR$ satisfies 
\[
	\sum_{n=0}^{\infty}\mu[r-a^n,r+a^n]<\infty.
\]
We apply this fact to $\mu \coloneqq h_*\nu$,
where $h(c) \coloneqq d_{GX_0}(c_0,c)$. By the triangle inequality, and then \eqref{eq.downstairshamenstadt}, we have
\begin{equation}\label{eq:estimates}
	\abs{h(c)-r} \leq d_{GX_0}(c,\partial B(c_0,r)) \leq Cd^+(c,\partial B(c_0,r)).
\end{equation}
Hence, using also the invariance of $\nu$ under $g$, we have
\begin{align*}
	&\sum_{n=0}^{\infty}\nu(\{c\in GX_0 : d^+(g^{-n}c,\partial B(c_0,r))\leq e^{-n\tau}\})\\
	&=\sum_{n=0}^{\infty}\nu(\{c\in GX_0 : d^+(c,\partial B(c_0,r))\leq e^{-n\tau}\})\\
	&\leq \sum_{n=0}^{\infty}\nu(\{c\in GX_0\colon d_{GX_0}(c,\partial B(c_0,r))\leq
	Ce^{-n\tau}\})\\
	&\leq \sum_{n=0}^{\infty}\nu(\{c\in GX_0 \colon \abs{h(c)-r} \leq Ce^{-n\tau}\}) \\
	&= \sum_{n=0}^{\infty} \mu[r-C(e^{-\tau})^n,r+C(e^{-\tau})^n],
\end{align*}
which is finite for Lebesgue-a.e.\ $r > 0$. By the Borel-Cantelli lemma, for $\nu$-a.e.\ $c \in GX_0$
we have
\[
	d^+(g^{-n}c,\partial B(c_0,r)) > e^{-n\tau}
\]
for all but finitely many $n$. The only way that $\beta_r(c)$ can be zero is if
$g^{-n}c \in\partial B(c_0,r)$ for some $n$.
Since the sets $\partial B(c_0,r)$ are pairwise disjoint for $r \in
(0,\rho)$, there can only be countably many of them with positive measure.
By avoiding this countable set, we can assume that $\nu(\cup_{n\geq 0}g^n\partial B(c_0,r)) = 0$
, and thus we have $\beta_r > 0$ almost
everywhere with respect to $\nu$.
\end{proof}
\begin{lemma}\label{lem:mane}
	Suppose that $0 < r \leq \frac{1}{3}\epsilon(c_0)$.
	If $\Omega X$ has finite upper box dimension with respect to $d_{GX}$,
	then there exists a countable partition 
	$\hat{\mathcal{P}}^r$ of $GX_0$ with $H_\nu(\hat{\mathcal{P}}^r) < \infty$ 
	such that $\mathcal{P}^r \coloneqq (\hat{\mathcal{P}}^r)^0_{-\infty}$ satisfies 
	$\mathcal{P}^r(c) \subseteq W^\mathrm{u}(c)$ for $\nu$-a.e.\ $c \in GX_0$.
	Moreover, $\mathcal{P}^r$ satisfies the stronger property that, 
	for $\nu$-a.e.\ $c \in GX_0$ and any $n \geq 0$ such that $g^{-n}c \in
	B(c_0,r)$, then $g^{-n}(\mathcal{P}^r(c))$ is contained in the local weak unstable set around 
	$g^{-n}c$ within $B(c_0,r)$.
\end{lemma}
\begin{proof} 
Let $V$ be the collection of elements of
$\mathrm{Fix}_N(GX_0) \cap B(c_0,r)$ that are forward and backward recurrent to $B(c_0,r)$ with respect
to $g$. We have $\nu(V) = \nu(B(c_0,r)) > 0$, and we write $\nu|_V$ for the restriction of $\nu$ to $V$. Since $V$ lifts to a subset of $\Omega X$, and $\pi_{GX}$ is Lipschitz, then $V$ must have finite
upper box dimension with respect to $d_{GX_0}$. Hence there is some $C>0$
and some $\delta>0$ such that, for any $s > 0$, there exists a partition of $V$ into at most 
$C(1/s)^\delta$ sets of diameter less than $s$. Therefore Ma\~n\'e's argument \cite[Lemma 2]{Mane}
provides, for any function $\rho : V \to (0,\infty)$ such that $\int_V-\log \rho\,d\nu < \infty$,
a countable partition $\hat{\mathcal{P}}$ of $V$ such that $H_{\nu|_V}(\hat{\mathcal{P}}) < \infty$ and
$\hat{\mathcal{P}}(c) \subseteq B(c,\rho(c))$ for $\nu|_V$-a.e.\ $c \in V$.

For $c \in V$, let $n(c) \in \NN$ be the smallest natural number with $g^{-n(c)}c \in B(c_0,r)$, and let 
\[
	\rho(c) \coloneqq \frac{1}{2}\epsilon_0 e^{-4n(c)\tau},
\]
where $\epsilon_0 > 0$ is the constant from Corollary \ref{cor:epsbound1} associated to $N$ and $r$.
The function $-\log \rho$ is integrable on $V$ by Ka{\v c}'s lemma. Let $\hat{\mathcal{P}}^r$
be the partition obtained by applying Ma\~n\'e's argument described above to this choice of $\rho$.
By including $V^c$, we think of
$\hat{\mathcal{P}}^r$ as a partition of $GX_0$, and we have $H_\nu(\hat{\mathcal{P}}^r) < \infty$.

By ergodicity, $\nu$-a.e.\ $c\in GX_0$ is in $\mathrm{Fix}_N(GX_0)$ and
enters $B(c_0,r)$ infinitely often in both forward and backward
time. Fix any such choice of $c \in GX_0$. Let $(n_k)_{k \in \NN}$
be the sequence of nonnegative integers
such that $g^{-n_k}c \in B(c_0,r)$, written in increasing order. We have $g^{-n_k}c \in V$
for every $k \in \NN$ by choice of $c$ and $V$. Thus, for each $k \in \NN$, we have
\[
	\hat{\mathcal{P}}^r(g^{-n_k}c) \subseteq B(g^{-n_k}c,
	\rho(g^{-n_k}c)) = B(g^{-n_k}c,\tfrac{1}{2}\epsilon_0e^{-4(n_{k+1}-n_k)\tau}).
\]
Let $c'\in \mathcal{P}^r(c)$ and $k \in \NN$. Since
$g^{-n_k}c' \in \hat{\mathcal{P}}^r(g^{-n_k}c)$ for any $t \in [n_k\tau,n_{k+1}\tau]$, we have
\begin{align*}
	d_{GX_0}(g_{-t}c,g_{-t}c') 
	&\leq d_{GX_0}(g_{-n_k\tau}c,g_{-n_k\tau}c')e^{2(n_{k+1}-n_k)\tau}\\
	&< \frac{1}{2}\epsilon_0 e^{-2(n_{k+1}-n_k)\tau} \\
	&\leq \frac{1}{2}e^{-2(n_{k+1}-n_k)\tau}\epsilon(g_{-n_k\tau}c)
	\leq \frac{1}{2}\epsilon(g_{-t}c),
\end{align*}
where we have used \eqref{eq:bddgrowth}, Lemma \ref{lem:epsbound2} and Corollary \ref{cor:epsbound1}, which can be applied because $g_{-n_k\tau}c \in\mathrm{Fix}_N(GX_0) \cap B(c_0,r)$. Since the inequality is true for $t\in[n_k\tau, n_{k+1}\tau]$ 
for each $k \in \NN$, then for all $t \geq n_1 \tau$, we have
\[
	d_{GX_0}(g_{-t}c,g_{-t}c') < \frac{1}{2}\epsilon(g_{-t}c). 
\] 
Let $\tilde{c}, \tilde{c}' \in GX$ be lifts of $c, c'$, respectively,
chosen so that 
\begin{equation}\label{eq.lifts}
d_{GX}(g_{-n_1\tau}\tilde{c},g_{-n_1\tau}\tilde{c}') 
= d_{GX_0}(g_{-n_1\tau}c,g_{-n_1\tau}c').
\end{equation}
We claim that, for each $t \geq n_1\tau$, we still have
\[
	d_{GX}(g_{-t}\tilde{c},g_{-t}\tilde{c}') = d_{GX_0}(g_{-t}c,g_{-t}c')
.\] 
Clearly the set of $t$ satisfying this equality is closed and contains $n_1\tau$.
On the other hand, if we are given any $t \geq
n_1\tau$ such that the above is true, then 
by continuity of $\epsilon(g_{-t'}c)$ in $t'$, and
since $d_{GX}(g_{-t}\tilde{c},g_{-t}\tilde{c}') = d_{GX_0}(g_{-t}c,g_{-t}c') <
\frac{1}{2}\epsilon(g_{-t}c)$, then there is some $\delta > 0$ such that
$d_{GX}(g_{-t'}\tilde{c},g_{-t'}\tilde{c}') < \epsilon(g_{-t'}c)$
for all ${t' \in [t,t+\delta)}$. By the definition of $\epsilon$, this implies
$d_{GX}(g_{-t'}\tilde{c},g_{-t'}\tilde{c}') = d_{GX_0}(g_{-t'}c,g_{-t'}c')$ for each $t' \in
[t,t+\delta)$, which shows the claim.

Since $g^{-n_k}c' \in \hat{\mathcal{P}}^r(g^{-n_k}c) \subseteq B(c_0,r)$, then we have
$d_{GX_0}(g^{-n_k}c,g^{-n_k}c') < 2r$
for each $k \in \NN$. We showed that $d_{GX}(g^{-n_k}\tilde{c}, g^{-n_k}\tilde{c}') =
d_{GX_0}(g^{-n_k}c,g^{-n_k}c')$ for every $k \in \NN$, and thus we have
$\liminf_{t\to+\infty}d_{GX}(g_{-t}\tilde{c},g_{-t}\tilde{c}') \leq 2r < \infty$.
This implies that $\tilde{c}' \in W^{\mathrm{u}}(\tilde{c})$,
which proves that $c' \in W^{\mathrm{u}}(c)$.

Let $k \in \NN$. Since $g^{-n_k}c \in B(c_0,r)$, the lifts $\tilde{c}$, $\tilde{c}'$ can be chosen so that in addition to \eqref{eq.lifts}, we have $g^{-n_k}\tilde{c}
\in B(\tilde{c}_0,r)$. As shown above, we have 
$d_{GX}(g^{-n_k}\tilde{c},g^{-n_k}\tilde{c}') < 2r$. Hence, we have
\[
	d_{GX}(g^{-n_k}\tilde{c}',\tilde{c}_0) \leq
	d_{GX}(g^{-n_k}\tilde{c}',g^{-n_k}\tilde{c})+d_{GX}(g^{-n_k}\tilde{c},\tilde{c}_0) < 3r \leq
	\epsilon(c_0)
.\] 
Since $g^{-n_k}c' \in B(c_0,r)$, there is some $a \in \Gamma$ such that $ag^{-n_k}\tilde{c}' \in
B(\tilde{c}_0,r)$. Since $g^{-n_k}\tilde{c}'$ and $ag^{-n_k}\tilde{c}'$ are both in
$B(\tilde{c}_0,\epsilon(c_0))$, we have $a \in \mathrm{Stab}_\Gamma(\tilde{c}_0)$. It follows that
\[
	d_{GX}(g^{-n_k}\tilde{c}',\tilde{c}_0) = d_{GX}(ag^{-n_k}\tilde{c}',\tilde{c}_0) < r
.\] 
Therefore, in addition to knowing $g^{-n_k}\tilde{c} \in B(\tilde{c}_0,r)$ by the choice of lift,
we know that $g^{-n_k}\tilde{c}' \in B(\tilde{c}_0,r)$. Since
we chose the lifts $\tilde{c}, \tilde{c}'$ to
satisfy  \eqref{eq.lifts}, then by the argument above, we also have $\tilde{c}' \in
\wu(\tilde{c})$. Hence, we have
\[
	g^{-n_k}\tilde{c}, \: g^{-n_k}\tilde{c}' \in W^\mathrm{u}(g^{-n_k}\tilde{c})\cap B(\tilde{c}_0,r)
,\] 
which proves that $g^{-n_k}c'$ is in the local weak unstable set around $g^{-n_k}c$ inside $B(c_0,r)$.
Since $c' \in \mathcal{P}^r(c)$ was arbitrary, we have shown that $g^{-n_k}(\mathcal{P}^r(c))$
is contained in the local weak unstable set around $g^{-n_k}c$ inside $B(c_0,r)$.
\end{proof}
For every $r'>0$, we define a set
\[
	\tilde{U}_{r'} \coloneqq \bigcup_{\abs{s} < r'}g_s\left(\cup_{\tilde{c}_- \in B^-(\tilde{c}_0,r')}
	B^+(\tilde{c}_-,r')\right)
,\] 
and let $U_{r'} \coloneqq \pi_{GX}(\tilde{U}_{r'})$.
Fix a scale $r' > 0$ small enough such that $U_{r'} \subseteq B(c_0,\epsilon(c_0))$. 
Then, since $a\in \mathrm{Stab}_{\Gamma}(\tilde{c}_0)$ if and only if
$a\tilde{U}_{r'} = \tilde{U}_{r'}$, we have
\[
	U_{r'} = \mathrm{Stab}_{\Gamma}(\tilde{c}_0)\backslash\tilde{U}_{r'} 
.\]

A geodesic line $\tilde{c} \in \tilde{U}_{r'}$ can be assigned `local coordinates'
$(\tilde{c}_-,\tilde{c}_+,s)$ defined by
\[
	\tilde{c}_- \in B^-(\tilde{c}_0,r'), \quad \tilde{c}_+ \in B^+(\tilde{c}_-,r'),
	\quad g_s\tilde{c}_+ = \tilde{c} \quad (\abs{s} < r')
.\] 
These coordinates are unique, since we can write
\[
	\{\tilde{c}_-\} = W^{\mathrm{ss}}(\tilde{c}_0) \cap W^\mathrm{u}(\tilde{c}), \quad \{\tilde{c}_+\}
	= W^\mathrm{uu}(\tilde{c}_-)\cap W^s(\tilde{c}),
	\quad s = \beta_{\tilde{c}(+\infty)}(\tilde{c}_+(0),\tilde{c}(0))
.\] 
Note that if $a \in \mathrm{Stab}_\Gamma(\tilde{c}_0)$,
then the coordinates for $a\tilde{c}$ are $(a\tilde{c}_-,a\tilde{c}_+,s)$.
Therefore, if we pass to the quotient, we find that any $c \in U_{r'}$
has unique local coordinates $(c_-,c_+,s)$ satisfying
\[
	c_- \in B^-(c_0,r'), \quad c_+ \in B^+(c_-,r'), \quad g_sc_+ = c \quad (\abs{s} < r')
.\] 

\begin{lemma}\label{lem:containment}
Suppose that $\Omega X$ has finite upper box
dimension with respect to $d_{GX}$. Then for any small enough $r > 0$,
the following is true. Let $\hat{\mathcal{P}}^r$ be a partition
guaranteed by Lemma \ref{lem:mane}, and let
$\hat{\mathcal{P}}'$ be the partition of $GX_0$ into $U_{r'}$ and $U_{r'}^c$. Then, defining
$\hat{\mathcal{Q}}^r = \hat{\mathcal{P}}^r\vee \hat{\mathcal{P}}'$ and 
$\mathcal{Q}^r = (\hat{\mathcal{Q}}^r)_{-\infty}^0$, we have 
\[\zeta^r \prec \mathcal{Q}^r \,\,(\mathrm{mod}\,\nu).\]
\end{lemma}
\begin{proof}
It follows from the last claim in Lemma \ref{lem:mane} that
\[\{B(c_0,r),B(c_0,r)^c\} \prec \hat{\mathcal{P}}^r \,(\mathrm{mod}\,\nu),\]
and hence
\[\{B(c_0,r),B(c_0,r)^c\} \prec \hat{\mathcal{Q}}^r \,(\mathrm{mod}\,\nu).\]
Without loss of generality, we may therefore assume that $\{B(c_0,r),B(c_0,r)^c\} \prec 
\hat{\mathcal{Q}}^r$.

Let $n \geq 0$ be given.
If $g^{-n}c \in B(c_0,r)^c$, then
\[
	\mathcal{Q}^r(c) \subseteq \hat{\mathcal{Q}}^r_{-n}(c) = g^n(\hat{\mathcal{Q}}^r(g^{-n}c))
	\subseteq g^n B(c_0,r)^c = \hat{\zeta}^r_{-n}(c)
.\]
So suppose instead that $g^{-n}c \in B(c_0,r)$. We will show that, for $\nu$-almost every choice of $c$,
we still have the containment $\mathcal{Q}^r(c) \subseteq \hat{\zeta}^r_{-n}(c)$.

Since $\hat{\mathcal{Q}}^r \prec \hat{\mathcal{P}}^r$,
then for $\nu$-a.e.\ $c \in GX_0$, the set $g^{-n}(\mathcal{Q}^r(c))$ is
contained inside the local weak unstable set around $g^{-n}c$ in $B(c_0,r)$.
Suppose that $c' \in \mathcal{Q}^r(c)$, and let $t$ be the `local time difference'
between $g^{-n}c'$ and $g^{-n}c$, which satisfies
\begin{equation}\label{eq.t}
g_tg^{-n}c' \in \zeta^r(g^{-n}c).
\end{equation}

Define two sequences of open sets,
\[
	V^+_m \coloneqq \{(c_-,c_+,s) \in U_{r'} \colon s > r'-\tfrac{1}{m}\},
	\quad V^-_m \coloneqq \{(c_-,c_+,s) \in U_{r'} \colon s < -r'+\tfrac{1}{m}\}
\] 
for $m\in \NN$, where we are using local coordinates inside of $U_{r'}$ to make this definition.
Since the support of $\nu$ is flow-invariant, and hence $g_sc_0$ is in the support of $\nu$ for any $s \in
\RR$, then $\nu(V^+_m),\nu(V^-_m) > 0$ for any $m \in \NN$.
By ergodicity, for $\nu$-a.e.\ $c \in GX_0$, the set
$\{g^{-n'}c : n' \geq n\}$ intersects $V^+_m$ and $V^-_m$ for every $m \in \NN$.

Let $m\in \NN$.
Let $n' \geq n$ with $g^{-n'}c \in V^+_m$. Since 
$\hat{\mathcal{Q}}^r$ is finer than $\hat{\mathcal{P}}'$,
and $V^+_m \subseteq U_{r'}$, we also know that $g^{-n'}c' \in \hat{\mathcal{P}}'(g^{-n'}c) = U_{r'}$. Thus we can write $g^{-n'}c = (c_-,c_+,s)$ and $g^{-n'}c' = (c_-',c_+',s')$ using the local coordinates on
$U_{r'}$. Since $n' \geq n$, and the Hamenst\"adt distance contracts uniformly in negative time,
then by assuming that $r$ is small enough, we can guarantee that
$g_tg^{-n'}c'$ and $g^{-n'}c$ lie in the same local strong unstable set within $U_{r'}$, given that
$g_tg^{-n}c'$ and $g^{-n}c$ are in the same local strong unstable set within $B(c_0,r)$.
By uniqueness of the coordinates on $U_{r'}$, we conclude that $t = s-s'$.  Since $g^{-n'}c \in V^+_m$ and
$g^{-n'}c' \in U_{r'}$, we have $s > r' - \frac{1}{m}$ and $s' < r'$, and hence  $t > -\frac{1}{m}$.
Since $m$ was arbitrary, we have $t \geq 0$. Applying the same argument to the sets $V^-_m$, we conclude
also that $t \leq 0$, and hence $t = 0$.
By \eqref{eq.t}, we have $g^{-n}c' \in \zeta(g^{-n}c)$, and hence
$\mathcal{Q}^r(c) \subseteq g^n(\hat{\zeta}^r(g^{-n}c)) = \hat{\zeta}^r_{-n}(c)$
since $c'$ was arbitrary.
\end{proof}
\begin{lemma}\label{lem:entropylowerbound}
	Let $\mathcal{R}$ and $\zeta$ be measurable partitions of $GX_0$. Suppose that
	$\zeta$ is decreasing and generating with respect to $\nu$, $H_\nu(\mathcal{R}) < \infty$, and 
	$\zeta\prec \mathcal{R}_{-\infty}^0\,\,(\mathrm{mod}\,\nu)$.
	Then $h_\nu(g,\mathcal{R}) \leq h_\nu(g,\zeta).$
\end{lemma}
\begin{proof}
The proof follows Ledrappier \cite[Lemme 6.5]{L13}. For every $n \geq 1$, we have
\[
h_\nu(g,\mathcal{R}) = H_\nu(\mathcal{R}_1|\mathcal{R}^0_{-\infty}) =
\frac{1}{n}\sum_{k=0}^{n-1}H_\nu(\mathcal{R}_{k+1}|\mathcal{R}^{k}_{-\infty})
= \frac{1}{n}H_\nu(\mathcal{R}^{n}_{1}|\mathcal{R}^0_{-\infty}).
\] 
Since $\zeta \prec \mathcal{R}^0_{-\infty} \,\,(\mathrm{mod}\,\nu)$, then
\begin{align*}
	h_\nu(g,\mathcal{R}) &\leq
	\lim_{n\to\infty}\frac{1}{n}H_\nu(\mathcal{R}_{1}^{n}|\zeta) \\
	&\leq \lim_{n\to\infty}\frac{1}{n}H_\nu(\mathcal{R}_{1}^n \vee \zeta_{n} | \zeta) \\
	&\leq \lim_{n\to\infty}\frac{1}{n}H_\nu(\zeta_{n}|\zeta) 
	+ \lim_{n\to\infty}\frac{1}{n}H_\nu(\mathcal{R}_{1}^n|\zeta_n)\\
	&= h_\nu(g,\zeta) + \lim_{n\to\infty}\frac{1}{n}H_\nu(\mathcal{R}_{1}^n|\zeta_n).
\end{align*}

Note that
\[
	\frac{1}{n}H_\nu(\mathcal{R}_{1}^n|\zeta_n)
	\leq \frac{1}{n}\sum_{i=0}^{n-1} H_\nu(\mathcal{R}_{n-i}|\zeta_n)
	= \frac{1}{n}\sum_{i=0}^{n-1}H_\nu(\mathcal{R}|\zeta_{i})
.\]
Since $\mathcal{R}$ has finite entropy and $\zeta$ is decreasing and generating with respect to $\nu$,
then \[\lim_{i\to\infty}H_\nu(\mathcal{R}|\zeta_i) = 0.\]
Hence $\lim_{n\to\infty}\frac{1}{n}H(\mathcal{R}^n_1|\zeta_n) = 0$, which completes the proof.
\end{proof}

\begin{proof}[Proof of Proposition \ref{prop:op}]
For small enough $r > 0$, let $\hat{\mathcal{Q}}^r$ be a partition guaranteed by Lemmas
\ref{lem:mane} and \ref{lem:containment}.
Let $\mathcal{F}$ be any finite measurable partition of $GX_0$, and let
$\mathcal{R} = \mathcal{F}\vee \hat{\mathcal{Q}}^r$. Clearly we have $H_\nu(\mathcal{R}) < \infty$ since
$H_\nu(\hat{\mathcal{Q}}^r) < \infty$ and $\mathcal{F}$ is finite.
We have $\zeta^r \prec \mathcal{R}^0_{-\infty}\,\,(\mathrm{mod}\,\nu)$ since 
$\zeta^r \prec \mathcal{Q}^r = (\hat{\mathcal{Q}}^r)^0_{-\infty}\,\,(\mathrm{mod}\,\nu)$.
The partition $\zeta^r$ is decreasing by definition, and it satisfies
\[
	\lim_{n\to\infty} \mathrm{diam}(\zeta^r_n(c)) = 0
\] 
for $\nu$-a.e.\ $c \in GX_0$ by Lemma
\ref{lem:strictdecrease}. In particular, $\zeta^r$ is generating with respect to $\nu$.
Therefore we may apply Lemma \ref{lem:entropylowerbound} to $\mathcal{R}$ and $\zeta^r$,
and so we have
\[
h_\nu(g,\mathcal{F}) \leq h_\nu(g,\mathcal{R})  \leq  h_\nu(g,\zeta^r) 
,\] 
where the first inequality uses the fact that $\mathcal{F}$ and $\mathcal{R}$ are finite entropy
partitions (otherwise $h_\nu(g,\,\cdot\,)$ is not necessarily monotone).
Since $\mathcal{F}$ was arbitrary,
we have $h_\nu(g,\zeta^r) = h_\nu(g)$. By Lemma \ref{lem:subordinate},
we can also choose $r$ so that $\zeta^r$ is $\nu$-subordinated to $\mathcal{W}^\mathrm{uu}$.
Thus $\zeta = \zeta^r$ is a partition with all the required properties.
\end{proof}

\section{Gibbs measure as an equilibrium state} \label{s.thermodynamic}
Let $\varphi : GX_0 \to \RR$ be a bounded continuous potential function with the Bowen
property. In this section, we show that if a Gibbs measure $m$ of exponent $\sigma \in \RR$ for $\varphi$ is finite, then $ P(\varphi) = \sigma$, and after normalizing $m$ to be a probability measure,
it must be the unique equilibrium state. Recall that $P(\varphi)$ is defined variationally via the expression \eqref{variationalpressure}.
\subsection{Strategy}
We review the strategy of \cite[\S6]{PPS}, highlighting the issues in generalizing to the $\CAT(-1)$
setting. In \cite[\S 6]{PPS}, the proof is broken into four steps. 

\subsubsection{Step 1} The first step is to construct the Ledrappier-Ma\~n\'e-Otal-Peign\'e partition,
which we carried out in the previous section.

\subsubsection{Step 2} In the Riemannian setting, the second step is to gave an exact and explicit description of a disintegration $(\mo_{\zeta(c)})_{c\in GX_0}$ of the Gibbs state $m$ with
respect to the partition $\zeta$ in terms of strong unstable measures $(\mu_{\su(c)})_{c \in
GX_0}$, which in turn have an exact description in terms of the Patterson-Sullivan measures.
 We follow an analogous argument but in our setting there is no canonical choice for the strong unstable measures
$\mu_{\su(c)}$ (see \S \ref{s.unstable measures}). We obtain only a comparison up to a uniform
constant, rather than an exact disintegration formula.

\subsubsection{Step 3 in the manifold case}We recall the main part of the proof in the manifold case. See
\cite[Lemma 6.6]{PPS}. It starts with analysis of the expression
\[
	 - \log \mo_{\zeta(c)}((g_{-\tau}\zeta)(c)).
\]
For an ergodic probability measure $\nu$, by choosing $\zeta$ according to $\nu$ as in the previous section,
one can ensure that the expression is well-defined $\nu$-almost everywhere. A calculation then shows that,  in the pinched negative curvature manifold setting, we have
\begin{equation} \label{eqstep31}
	  \int  - \log \mo_{\zeta(c)}((g_{-\tau}\zeta)(c))\,d \nu (c) =
	  \tau \sigma- \tau \int \varphi \, d \nu.
\end{equation}
This shows that $h (m) + \int \varphi\,  d m = \sigma$ by setting $\nu = m$ and canceling $\tau$. Showing that $\sigma$ is an upper bound,
and that $m$ is the only measure realizing this upper bound relies on analysis of the function
\[
	\psi(c) \coloneqq
	\frac{\mo_{\zeta(c)}((g_{-\tau}\zeta)(c))}
	{\nu_{\zeta(c)}((g_{-\tau}\zeta)(c))},
\] 
which is well-defined and finite $\nu$-almost everywhere. One computes
\begin{equation} \label{eqstep32}
	\int -\log\psi \,d\nu = \tau( \sigma - \int\varphi\,d\nu - h(\nu)),
\end{equation}
and observe that the term in parentheses is the difference in the free energies
of $m$ and $\nu$. On the other hand, it is shown from the definition of $\psi$ 
that $\int \psi\,d\nu \leq 1$. Thus by Jensen's inequality,
\[
	\int-\log\psi\,d\nu \geq -\log\int\psi\,d\nu \geq 0
,\]
and it follows that the free energy of $\nu$ is at most that of $m$.
Furthermore, if $h(\nu) + \int \varphi \,d \nu = \sigma$,
then we are in the equality case in Jensen's inequality, and therefore $\psi(c)=1$
for $\nu$-a.e.\ $c \in GX_0$. The rest of the proof uses this fact and a Hopf argument to show that $\nu=m$.

\subsubsection{Limitations of Step 3 in our setting and a new strategy for uniqueness} In our setting, the
formulae \eqref{eqstep31} and \eqref{eqstep32} pick up an additional error term.
However, if we use $n\tau$ in place of $\tau$, this error term is independent of $n$, and much of the same strategy from \cite{PPS} carries through. 
For example, the arguments from \cite{PPS} adapted to our setting tells us that
\[
	\abs{n\tau h(m) + n\tau\int\varphi\,dm - n\tau\sigma} \leq K.
\] 
This is enough to conclude that the free energy of $m$ is equal to $\sigma$ by dividing by $n\tau$ and
letting $n \to \infty$. Jensen's inequality tells us
that $\sigma$ is an upper bound on the free energies, proving that $m$ is an equilibrium
state.

Uniqueness is considerably more tricky,
since it is not possible in our setting to arrive at the
equality case of Jensen's inequality given that $h(\nu)+\int\varphi\,d\nu = \sigma$.
Rather, this only implies that $\int-\log\psi_n\,d\nu \leq K$ for all $n \geq 1$, where $\psi_n$ is defined similarly to $\psi$ using $n\tau$ instead
of $\tau$. 

A more robust argument is required to conclude uniqueness.
Lemma \ref{lem:divergence} is the key idea.
We interpret $\int -\log\psi_n\,d\nu$ as a
Kullbeck-Leibler divergence and show directly that if $\nu$ is mutually singular to $m$,
then $\lim_{n\to\infty}\int-\log\psi_n\,d\nu = \infty$.
Therefore if $h(\nu) + \int\varphi\,d\nu = \sigma$, then since $\int-\log\psi_n\,d\nu$ is bounded, 
$\nu$ cannot be mutually singular to $m$, which implies $\nu = m$ by ergodicity. We remark that our
uniqueness proof is in the spirit of Bowen's uniqueness proof \cite[Lemma 8]{Bowen} for ergodic Gibbs
measures on compact expansive systems.
\subsubsection{On Step 4: The Variational Principle} \label{principev1}
The argument outlined above yields that if there exists a Gibbs state $m$ with exponent $\sigma$ (with no
assumption that $m$ is finite), then $P(\varphi) \leq \sigma$, and if $m$ is finite, we have the equality $P(\varphi) = \sigma$. In particular, using our result that a Gibbs state $m$ with exponent $\delta_\varphi$ always exists, we obtain that $P(\varphi) \leq \delta_\varphi$, with equality if $m$ is a finite measure.  

In the pinched negative curvature manifold setting,  `Step 4' in \cite[\S6]{PPS} shows that  $P(\varphi) = \delta_\varphi$, and that if the Gibbs state of exponent $\delta_\varphi$ is infinite there is no equilibrium state. To obtain that $P(\varphi) = \delta_\varphi$ in our setting, all that remains is to find an argument that $P(\varphi) \geq \delta_\varphi$ which covers the case that the Gibbs state $m$ with exponent $\delta_\varphi$ is infinite. 
We expect that this can be done. However, the situation where there is no equilibrium state is less
interesting for the current project, so we do not pursue these arguments in this paper. We hope that this
piece of the picture will be completed by interested researchers in the future.

\subsection{Conditional measures on strong unstable sets} \label{s.unstable measures}
We introduce a notation convention for the bounded errors that appear in our calculations.
We use $K \geq 0$ for uniform additive errors and $k \geq 1$ for uniform multiplicative errors,
with the exact value changing from line to line, and we write
\[
	A = B \pm K \iff \abs{A - B} \leq K,
\] 
\[
A = k^{\pm}B \iff k^{-1}B \leq A \leq kB
.\]

Recall that our potential $\varphi$ is bounded, continuous, and satisfies the Bowen property. We fix a Gibbs state $m$ on $GX$ associated to a pair of quasi-conformal measures
$\mu$, $\mu^\iota$ for $\varphi$, $\varphi\circ\iota$ of exponent $\sigma$. It is convenient to extend our definition of the Gibbs quasi-cocycle.
\begin{definition}
Let $x,y \in X$, $\xi \in \partial_\infty X$. For $z \in X$, let
\[
	Q_z(x,y;\varphi) \coloneqq \bphi(x,z)-\bphi(y,z)
,\] 
and for $\xi \in \partial_\infty X$, let
\[
	Q_{\xi}(x,y;\varphi) := \sup\left\{\limsup_{n\to\infty}Q_{z_n}(x,y;\varphi)
	: z_n \in X, n\in \NN, z_n \to \xi\right\}
.\] 
\end{definition}
This quasi-cocycle is related to the one defined in \S \ref{s.weightedPS} by
\[
Q(a,\xi;\varphi) = Q_{\xi}(ap,p;\varphi).
\] 
By the same argument based on trees that we used in \S \ref{subsec:qc and gp},
but with $x, y$ replacing $ap, p$, we have the following.
\begin{lemma}\label{lem:extension again}
Suppose that $x,y\in X$ and $z_n, z_n'$ are two sequences in $X$ approaching $\xi \in \partial_\infty X$.
Then
\[
	\limsup_{n\to\infty}\abs{Q_{z_n}(x,y;\varphi) - Q_{z'_n}(x,y;\varphi)} \leq K
.\] 
\end{lemma}
\begin{lemma}
For any $x,y,z \in X$, $\xi \in \partial_\infty X$, we have
\begin{equation}\label{eq:qc}
	Q_{\xi}(x,z;\varphi) =  Q_{\xi}(x,y;\varphi) + Q_{\xi}(y,z;\varphi) \pm K.
\end{equation}
If $y$ lies on the geodesic ray from $x$ to $\xi$, then
\begin{equation}\label{eq:inline}
	Q_{\xi}(x,y;\varphi) = \bphi(x,y) \pm K.
\end{equation}
Given $\eta \in \partial_\infty X$, and any $x$ on the geodesic line from $\xi$ to $\eta$, we have
\begin{equation}\label{eq:breakup}
	2(\xi|\eta;\varphi) = Q_\xi(p,x;\varphi\circ\iota) + Q_\eta(p,x;\varphi) \pm K.
\end{equation}
\end{lemma}
\begin{proof}
For $x,y,z,w \in X$, an easy calculation shows that
\[
	Q_{z}(x,w;\varphi)=Q_z(x,y;\varphi)+Q_z(y,w;\varphi)
,\] 
which proves \eqref{eq:qc} by letting $z \to \xi$ and using Lemma \ref{lem:extension again}.

If $y \in [x,z]$, then by the roughly geodesic property, we have
\[
	Q_z(x,y;\varphi) = \bphi(x,z)-\bphi(y,z) = \bphi(x,y)+\bphi(y,z) \pm K -\bphi(y,z) 
	= \bphi(x,y)\pm K
,\] 
which gives \eqref{eq:inline} by taking $z \to \xi$ along the ray from $x$ to $\xi$.

Finally, if $x \in [y,z]$, then using the roughly geodesic property, we have
\begin{align*}
	2(y|z;\varphi) &= \bphi(y,p)+\bphi(p,z)-\bphi(y,z) \\
	&= \bphi(y,p)+\bphi(p,z)-\bphi(y,x)-\bphi(x,z)\pm K\\
	&= (\bphi(y,p)-\bphi(y,x))+(\bphi(p,z)-\bphi(x,z)) \pm K\\
	&= Q_y(p,x;\varphi\circ\iota) + Q_z(p,x;\varphi) \pm K.
\end{align*}
We let $y \to \xi$ and $z \to \eta$ along the geodesic joining $\xi$ and $\eta$ to obtain \eqref{eq:breakup}.
\end{proof}
We use this Gibbs quasi-cocycle to define strong unstable measures
$\mu_{\su(c)}$ and stable measures $\mu_{W^{\mathrm{s}}(c)}$ for $c \in GX$, and we outline their properties.
\begin{definition}
Let $c \in GX$. Define a measure $\mu_{\su(c)}$ on $\su(c)$ by 
\[
	d\mu_{\su(c)}(c') = e^{Q_{c'(+\infty)}(c'(0),p;\,\varphi-\sigma)}d\mu(c'(+\infty))
,\] 
and a measure $\mu_{W^{\mathrm{s}}(c)}$ on $W^{\mathrm{s}}(c)$ by
\[
	d\mu_{W^{\mathrm{s}}(c)}(c') =
	e^{Q_{c'(-\infty)}(c'(0),p;\,\varphi\circ\iota-\sigma)}d\mu^\iota (c'(-\infty))dt
.\]
\end{definition}
The above definitions are independent of $c$. Note that our definition of $\mu_{W^{\mathrm{s}}(c)}$ uses an identification of $W^s(c)$ with $(\partial_\infty X \setminus \{c(+\infty)\}) \times \RR$. While there 
are many such identifications, they all differ by a reparameterization of the form $(\xi,t) \to
(\xi,h(\xi)+t)$ for some function $h : \partial_\infty X \setminus \{c(+\infty)\} \to \RR$,
so $\mu_{W^{\mathrm{s}}(c)}$ is independent of this choice.  The same applies for $\mu_{\su(c)}$.
\begin{lemma}
Let $c \in GX$. For all $a \in \Gamma$ and $\mu_{\su(ac)}$-a.e.\ $c' \in \su(ac)$, we have
\begin{equation}\label{eq:quasi-inv}
	\frac{da_*\mu_{\su(c)}}{d\mu_{\su(a c)}}(c') = k^{\pm},
\end{equation}
and for all $t \in \RR$ and  $\mu_{\su(g_tc)}$-a.e.\ $c' \in \su(g_tc)$, we have
\begin{equation}\label{eq:flow}
	\frac{d(g_t)_*\mu_{\su(c)}}{d\mu_{\su(g_t c)}}(c') = 
	k^{\pm} e^{\int_{-t}^0 (\varphi(g_{s} c')-\sigma) \,ds}.
\end{equation}
\end{lemma}
\begin{proof}
First we prove \eqref{eq:quasi-inv}. Using \eqref{eq:qc}, for $c' \in \wuu(ac)$, we have
\begin{align*}
	da_*\mu_{\su(c)}(c') &=
	e^{Q_{a^{-1}c'(+\infty)}(a^{-1}c'(0),p;\,\varphi-\sigma)}\,da_*\mu(c'(+\infty))\\
	&= k^\pm e^{Q_{c'(+\infty)}(c'(0),ap;\,\varphi-\sigma)+Q_{c'(+\infty)}(ap,p;\,\varphi-\sigma)}
	\,d\mu(c'(+\infty))\\
	&= k^\pm e^{Q_{c'(+\infty)}(c'(0),p;\,\varphi-\sigma)}\,d\mu(c'(+\infty)) \\
	&= k^\pm d\mu_{\su(ac)}(c').
\end{align*} 
To prove \eqref{eq:flow}, let $c' \in \wuu(g_tc)$, and we use \eqref{eq:qc} to calculate that
\begin{align*}
	d(g_t)_*\mu_{\su(c)}(c')&=e^{Q_{(g_{-t}c')(+\infty)}((g_{-t}c')(0),p;\,\varphi-\sigma)}
	\,d\mu((g_{-t}c')(+\infty))\\
	&= k^\pm e^{Q_{c'(+\infty)}(c'(-t),c'(0);\,\varphi-\sigma)
	+Q_{c'(+\infty)}(c'(0),p;\,\varphi-\sigma)}\,d\mu(c'(+\infty))\\
	&= k^\pm e^{Q_{c'(+\infty)}(c'(-t),c'(0);\,\varphi-\sigma)}
	\,d\mu_{\su(g_tc)}(c').
\end{align*}
If $t \geq 0$, then by \eqref{eq:inline}, we have that
\[Q_{c'(+\infty)}(c'(-t),c'(0);\,\varphi-\sigma) = \overline{(\varphi-\sigma)}(c'(-t),c'(0)) \pm K.\]
If $t \leq 0$, then using \eqref{eq:qc} and \eqref{eq:inline}, we have that
\begin{align*}
	Q_{c'(+\infty)}(c'(-t),c'(0);\,\varphi-\sigma) &= - Q_{c'(+\infty)}(c'(0),c'(-t);\varphi-\sigma)
	\pm K\\
	&= -\overline{(\varphi-\sigma)}(c'(0),c'(-t))\pm K.
\end{align*}
In either case, applying Lemma \ref{lem:integrals}, we find that
\[
	Q_{c'(+\infty)}(c'(-t),c'(0);\,\varphi-\sigma) = \int_{-t}^0(\varphi(g_sc')-\sigma)\,ds\pm K
.\qedhere\] 
\end{proof}
The following result is interpreted as an `approximate disintegration' of the
Gibbs state $m$ with conditional measures  $e^{Q_{c(-\infty)}(c'(0),c(0);\,\varphi\circ\iota)}d\mu_{W^{\mathrm{uu}}(c)}(c')$  on the strong unstable
leaves.

\begin{proposition}\label{prop:disintegration}
For any $f : GX \to \RR$ which is integrable with respect to $m$ and any $c_0 \in GX$ such
that $c_0(+\infty)$ is not an atom of $\mu^\iota$, we have
\[
	\int f(c') \,dm(c') = k^{\pm} \int\int
	f(c')e^{Q_{c(-\infty)}(c'(0),c(0);\,\varphi\circ\iota)}
	\,d\mu_{W^{\mathrm{uu}}(c)}(c')d\mu_{W^{\mathrm{s}}(c_0)}(c)
.\] 
\end{proposition}
\begin{proof}
We identify $GX$ with $\partial_\infty^2 X \times \RR$ using the parametrization
$c' \mapsto (\xi,\eta,t)$ where $\xi = c'(-\infty)$, $\eta = c'(+ \infty)$, and
$t = \beta_{c'(-\infty)}(c'(0),p)$. We let $c(\xi,\eta,t)$ denote the geodesic line corresponding
to $(\xi,\eta,t)$, and let $x(\xi,\eta,t) = \pi_{\mathrm{fp}}(c(\xi,\eta,t))$. Sets of the form $\{\xi\} \times (\partial_\infty X \setminus \{\xi\} ) \times \{t\}$
correspond to the strong unstable sets of $GX$, and sets
of the form $(\partial_\infty X \setminus \{\eta\}) \times \{\eta\} \times \RR$
correspond to the weak stable sets. 

Let $\eta_0 = c_0(+\infty)$. Since $\eta_0$ is not an atom for $\mu^\iota$, it suffices to consider geodesic lines  with backwards endpoint not  equal to $\eta_0$.  We compute that
\begin{gather*}
	e^{Q_{c(-\infty)}(c'(0),c(0);\,\varphi\circ\iota)}
	d\mu_{W^{\mathrm{uu}}(c)}(c')\,d\mu_{W^{\mathrm{s}}(c_0)}(c)\\
	= e^{
		Q_{\xi}(x(\xi,\eta,t),x(\xi,\eta_0,t);\,\varphi\circ\iota)
		+Q_{\eta}(x(\xi,\eta,t),p;\,\varphi-\sigma)
		+Q_{\xi}(x(\xi,\eta_0,t),p;\,\varphi\circ\iota-\sigma)
	}d\mu(\eta)d\mu^\iota(\xi)dt\\
	= e^{
		Q_{\xi}(x(\xi,\eta,t),x(\xi,\eta_0,t);\,\varphi\circ\iota-\sigma)
		+Q_{\eta}(x(\xi,\eta,t),p;\,\varphi-\sigma)
		+Q_{\xi}(x(\xi,\eta_0,t),p;\,\varphi\circ\iota-\sigma)
	}d\mu(\eta)d\mu^\iota(\xi)dt\\
	= k^{\pm}e^{
		Q_{\xi}(x(\xi,\eta,t),p;\,\varphi\circ\iota-\sigma)
		+Q_{\eta}(x(\xi,\eta,t),p;\,\varphi-\sigma)
	}d\mu(\eta)d\mu^\iota(\xi)dt\\
	= k^{\pm}e^{
		-Q_{\xi}(p,x(\xi,\eta,t);\,\varphi\circ\iota-\sigma)
		-Q_{\eta}(p,x(\xi,\eta,t);\,\varphi-\sigma)
	}d\mu(\eta)d\mu^\iota(\xi)dt,
\end{gather*}
while by definition of $m$, we have
\[
	dm(c') = k^\pm e^{-2(\xi|\eta;\,\varphi-\sigma)} \,d\mu(\eta)\,d\mu^\iota(\xi)\,dt
.\] 
Since $x(\xi,\eta,t)$ lies on the geodesic from $\xi$ to $\eta$, the result follows by
\eqref{eq:breakup}.
\end{proof}
We define measures $\mu_{W^{\mathrm{uu}}(c)}$ on $W^{\mathrm{uu}}(c)$ for $c \in GX_0$.
In our setting, different choices of lifts for $W^{\mathrm{uu}}(c)$ 
result in different measures. We must ensure that these choices are made so that the
strong unstable measures on $GX_0$ are a measurable family. By \eqref{eq:quasi-inv}, different choices are comparable by a uniform constant. Our analysis is independent of how the lifts are chosen, as long as we ensure measurability.

Take a countable cover $\mathcal U$ of $GX_0$ of the form $\{B(c_i,\epsilon(c_i))\}_{i\in I}$, where
$\eps$ is the function defined at \eqref{epsilonc} with $Y = GX$.
For each $i \in I$, let $\hat{\zeta}_i$ denote the partition of $B(c_i,\epsilon(c_i))$ into its local
strong unstable sets, see \S\ref{sec:op}.  Choose a lift $\tilde{c}_i \in GX$ of each $c_i$. For $\tilde{c} \in
B(\tilde{c}_i,\epsilon(c_i))$, we define
\[
	\Gamma_{i,\tilde{c}} = \{a \in \mathrm{Stab}_{\Gamma}(\tilde{c}_i) \colon \wuu(a\tilde{c})
	= \wuu(\tilde{c})\},\] 
which is the subgroup of $\mathrm{Stab}_{\Gamma}(\tilde{c}_i)$
that fixes (as a set) the element of
$\mathcal{W}^{\mathrm{uu}}|_{B(\tilde{c}_i,\epsilon(c_i))}$ that contains $\tilde{c}$.
For any $c \in B(c_i,\epsilon(c_i))$ and any lift $\tilde{c} \in B(\tilde{c}_i,\epsilon(c_i))$ of $c$,
we can thus naturally identify the sets $\hat{\zeta}_i(c)$ and 
$\Gamma_{i,\tilde{c}} \, \backslash \, (\wuu(\tilde{c})\cap B(\tilde{c}_i,\epsilon(c_i)))$.

Given $\tilde{c} \in B(\tilde{c}_i,\epsilon(c_i))$,
define a measure $\mu_{i,\tilde{c}}$ on $\wuu(\tilde{c})\cap B(\tilde{c}_i,\epsilon(c_i))$ by
\[
	\mu_{i,\tilde{c}} \coloneqq \frac{1}{\# \mathrm{Stab}_\Gamma(\tilde{c}_i)}
	\sum_{a\in\mathrm{Stab}_\Gamma(\tilde{c}_i)}
	(a_*^{-1}\mu_{W^\mathrm{uu}(a\tilde{c})})
	|_{\wuu(\tilde{c})\cap B(\tilde{c}_i,\epsilon(c_i))}.
\] 
By definition, the measure $\mu_{i,\tilde{c}}$ is $\Gamma_{i,\tilde{c}}$-invariant. Furthermore, by
\eqref{eq:quasi-inv} we have
\begin{equation}\label{eq.averageunstable}
\mu_{i,\tilde{c}} = k^{\pm}\mu_{\wuu(\tilde{c})}|_{\wuu(\tilde{c})\cap B(\tilde{c}_i,\epsilon(c_i))}.
\end{equation}
Let $c \in B(c_i,\epsilon(c_i))$, and choose a lift $\tilde{c} \in B(\tilde{c}_i,\epsilon(c_i))$ of
$c$. We use Remark \ref{pushdownameasure} to define a measure $\mu_{i,c}$ on $\hat{\zeta}_i(c)$ induced from $\mu_{i,\tilde{c}}$ by the group $\Gamma_{i,\tilde{c}}$.
This definition of $\mu_{i,c}$ is independent of the choice of lift $\tilde{c}$.

For each $i \in I$ and $c \in GX_0$, let $\mu^i_{W^\mathrm{uu}(c)}$ be the measure on
$W^{\mathrm{uu}}(c) \cap B(c_i,\epsilon(c_i))$ whose restriction to $\hat{\zeta}_i(c')$
is the measure $\mu_{i,c'}$ for any $c' \in W^{\mathrm{uu}}(c) \cap B(c_i,\epsilon(c_i))$.
Take a partition of unity $\{\varphi_i : GX_0 \to \RR\}_{i\in I}$ subordinate to the open cover
$\mathcal U$, and define
\[
	d\mu_{W^{\mathrm{uu}}(c)} = \sum_{i\in I}\varphi_id\mu^i_{W^\mathrm{uu}(c)}
.\]
Clearly $\mu_{W^{\mathrm{uu}}(c)} = \mu_{W^\mathrm{uu}(c')}$ if $c' \in W^\mathrm{uu}(c)$. Since $\mu_{W^\mathrm{uu}(c)}$ is locally an average of strong
unstable measures pushed down from $GX$, all of which differ by a uniform constant by
\eqref{eq.averageunstable}, we obtain the following version of \eqref{eq:flow} on $GX_0$.
\begin{lemma}\label{lem:flow2}
For any $c \in GX_0$, any $t \in \RR$, and $\mu_{W^{\mathrm{uu}}(g_tc)}$-a.e.\ $c' \in
W^{\mathrm{uu}}(g_tc)$, then
\begin{equation}\label{eq:flow2}
	\frac{d(g_t)_*\mu_{\su(c)}}{d\mu_{\su(g_t c)}}(c') = 
	k^{\pm} e^{\int_{-t}^0 (\varphi(g_{s} c')-\sigma) \,ds}.
\end{equation}
\end{lemma}
We define another useful quantity for pairs of geodesic lines in the same strong unstable set, which we need in the statements which follow.
\begin{definition}
For any $c',c \in GX_0$ in the same strong unstable set, we define
\[
	C(c',c;\varphi) = \limsup_{t\to+\infty}\int_{-t}^0 (\varphi(g_{s}c') -\varphi(g_{s}c))\,ds
.\]
\end{definition}

\begin{lemma}\label{lem:CQ}
For any $c,c' \in GX_0$ in the same strong unstable set and any two lifts $\tilde{c}, \tilde{c}' \in GX$
which are chosen to be in the same strong unstable set of $GX$, we have
\[
	C(c',c;\varphi) = Q_{\tilde{c}(-\infty)}(\tilde{c}'(0),\tilde{c}(0);\varphi\circ\iota) \pm K
.\] 
\end{lemma}
\begin{proof}
For sufficiently large $t \geq 0$, using Lemma \ref{lem:integrals}, the roughly geodesic property,
and the fact that $d(\tilde{c}(-t),\tilde{c}'(-t))$ approaches $0$ as $t \to \infty$, we have
\begin{align*}
	\int_{0}^{t}(\varphi(g_{-s}c')-\varphi(g_{-s}c))\,ds
	&= \int_{0}^{t}(\varphi(g_{-s}\tilde{c}')-\varphi(g_{-s}\tilde{c}))\,ds\\
	&= \int_{-t}^{0}(\varphi(g_{s}\tilde{c}')-\varphi(g_{s}\tilde{c}))\,ds\\
	&= \bphi(\tilde{c}'(-t),\tilde{c}'(0)) - \bphi(\tilde{c}(-t),\tilde{c}(0)) \pm K \\
	&= \bphi(\tilde{c}(-t),\tilde{c}'(0)) - \bphi(\tilde{c}(-t),\tilde{c}(0)) \pm K \\
	&= \overline{(\varphi\circ\iota)}(\tilde{c}'(0),\tilde{c}(-t))
	- \overline{(\varphi\circ\iota)}(\tilde{c}(0),\tilde{c}(-t)) \pm K \\
	&= Q_{\tilde{c}(-t)}(\tilde{c}'(0),\tilde{c}(0);\varphi\circ\iota) \pm K.
\end{align*}
Taking the upper limit as $t \to +\infty$ and using Lemma \ref{lem:extension again}
gives the desired result.
\end{proof}
Lemma \ref{lem:CQ} implies that
$C(\,\cdot\,,\,\cdot\,\,;\varphi)$ is finite wherever it is defined, and given any 
$c,c',c'' \in GX_0$ in the same strong unstable set, it satisfies a quasi-cocycle property
\[
	C(c'',c;\varphi) = C(c'',c';\varphi) + C(c',c;\varphi) \pm K.
\] 
For $c \in GX_0$, we say that a set $A$ is an \emph{admissible uu-neighborhood of $c$} if $A$ is a relatively compact Borel subset of a strong unstable set $\su(c)$ inside $GX_0$ such that the relative interior of $A$ in $\su(c)$ intersects the support of $\mu_{\su(c)}$. We define a measure $m^0_A$ on $A$ by
\[
	dm^0_A(c') = \frac{\bbid_A(c')}
	{\int_A e^{C(c'',c';\,\varphi)}d\mu_{\su(c)}(c'')}\,d\mu_{\su(c)}(c')
.\] 
The definition of $m^0_A$ does not depend on the choice of $c$ in the strong unstable set
containing $A$. The following lemma is an immediate consequence of the quasi-cocycle property of
$C(\,\cdot\,,\,\cdot\,\,;\varphi)$.
\begin{lemma}\label{lem:normalized}
Suppose $c \in GX_0$, and let $A$ be an admissible uu-neighborhood of $c$.  Then
\[
dm^0_A(c') = k^\pm \frac{\bbid_A(c') e^{C(c',c;\,\varphi)}\,d\mu_{W^\mathrm{uu}(c)}(c')}
{\int_A e^{C(c'',c;\,\varphi)}\,d\mu_{\su(c)}(c'')}
.\] 
In particular, $\lVert m^0_A\rVert = k^\pm$.
\end{lemma}
\begin{proposition}\label{prop:disintegration2}
Suppose that $m$ is finite on $GX_0$, and that it has been normalized so
that ${m(GX_0) = 1}$. Let $\zeta$ be a partition of $GX_0$ which is $m$-subordinated to
$\mathcal{W}^\mathrm{uu}$. Then for any measurable function $f : GX_0 \to \RR$
which is bounded with compact support, we have
\[
	\int f\, dm = k^{\pm}\int \int f|_{\zeta(c)}\, d\mo_{\zeta(c)} \,dm(c)
.\] 
\end{proposition}
Note that the measure $\mo_{\zeta(c)}$ is
well-defined for $m$-a.e.\ $c \in GX_0$ since $\zeta$ is $m$-subordinated to $\mathcal{W}^\mathrm{uu}$. Proposition \ref{prop:disintegration2} can be interpreted as an analog of Proposition
\ref{prop:disintegration} stated on $GX_0$. This can be seen by Lemma \ref{lem:CQ} and the fact that 
$m^0_A$ is, up to a constant, a normalized restriction of the measure  $e^{C(c',c;\,\varphi)}\,d\mu_{\su(c)}(c')$ by Lemma \ref{lem:normalized}. We omit the proof of Proposition \ref{prop:disintegration2} since it is identical to the detailed proof of 
\cite[Lemma 6.5]{PPS} after including the uniformly bounded errors that arise in our setting.
\subsection{Existence and uniqueness} \label{sec:existence}
The goal of this section is to prove Theorem \ref{thmx:existunique}. We do this by showing the
following result.
\begin{theorem}\label{thm:existence}
Suppose $\varphi$ is bounded, continuous, and satisfies the Bowen property.
Let $m$ be a Gibbs state for $\varphi$ of exponent $\sigma \in \RR$, considered as a measure on $GX_0$. If $m$ is finite, then $\sigma = P(\varphi)$, and normalizing $m$ to be a probability measure,
$m$ is the unique equilibrium state for $\varphi$.
\end{theorem}

We fix a Gibbs state $m$ for $\varphi$ of exponent $\sigma \in \RR$, assume that it is finite on $GX_0$,
and normalize it to be a probability measure. The measure $m$ is conservative and it is thus ergodic by Lemma \ref{lem:finiteimplies}.

Given an invariant probability measure $\nu$ on $GX_0$,
we denote a disintegration for $\nu$ with respect to a
measurable partition $\zeta$ by $(\nu_A)_{A \in \zeta}$.
We emphasize that it is not sufficient to work with an abstractly provided disintegration
$(m_A)_{A \in \zeta}$ for $m$ because it would only be defined $m$-almost everywhere,
and we need to compare $m$ with mutually singular measures $\nu$.
This is why it is crucial that an explicit reference measure $m^0_A$ 
is defined for \emph{every} admissible uu-neighborhood $A$. This gives us the following statement, which we apply implicitly throughout this section.
\begin{lemma}
	Let $\nu$ be an invariant probability measure on $GX_0$.
	Suppose that $\zeta$ is a measurable partition of $GX_0$ which is
	$\nu$-subordinated to $\mathcal{W}^\mathrm{uu}$. Then, for $\nu$-a.e.\ $c \in GX_0$,
	the reference measure $m^0_{\zeta(c)}$ is well-defined.
\end{lemma}
\begin{proof}
Since $\nu$ is finite and invariant, its support must be contained in the nonwandering set
$\Omega X_0$ for the flow. Since we assumed $m$ to be finite, the same is true for $m$,
and by the quasi-product structure of $m$ on $GX$ and the fact that the nonwandering set lifts to
the set $\Omega X$ of geodesic lines whose endpoints both lie in $\Lambda$, then the support of $\mu$
is contained in $\Lambda$.  By Lemma \ref{lem:limitset}, its support is equal
to $\Lambda$. Hence any $c \in \Omega X_0$ lies in the support of $\mu_{\wuu(c)}$, which proves that
$\nu$-a.e.\ $c \in GX_0$ lies in the support of $\mu_{W^\mathrm{uu}(c)}$.
Suppose $\zeta$ is a measurable partition of $GX_0$ which is $\nu$-subordinated to
$\mathcal{W}^{\mathrm{uu}}$. Then for $\nu$-a.e.\ $c \in GX_0$, the set $\zeta(c)$ is a
relatively compact neighborhood of $c$ in $\su(c)$, and $c$ lies in the support of $\mu_{\wuu(c)}$.
This implies that $\zeta(c)$ is an admissible uu-neighborhood,
and therefore $\mo_{\zeta(c)}$ is well-defined.
\end{proof}

\begin{lemma}\label{lem:approx}
Let $\nu$ be an ergodic $g_t$-invariant probability measure on $GX_0$,
let $\tau > 0$ be such that $g = g_\tau$ is
ergodic for $\nu$, and suppose that $\zeta$ is a partition
guaranteed by Proposition \ref{prop:op} for $\nu$ and $\tau$. For $c \in GX_0$ we let
\[
	G(c) = -\log \int_{\zeta(c)}e^{C(c',c;\varphi)}\,d\mu_{\su(c)}(c')
,\] 
and note that $G(c)$ is finite for $\nu$-a.e.\ $c \in GX_0$ since $\zeta(c)$ is relatively compact in
$\wuu(c)$ for $\nu$-a.e.\ $c \in GX_0$. Then for $\nu$-a.e.\ $c \in GX_0$ and any $n \geq 0$, we have
 \[
	 -\log \mo_{\zeta(c)}((g^{-n}\zeta)(c))
	 = n\tau\sigma - \int_{0}^{n\tau}\varphi(g_t c)\,dt +
	 G(g^n c)- G(c) \pm K
,\] 
where $K$ is uniform in $n$. Furthermore, for $\nu$-a.e.\ $c \in GX_0$ we have
\[
	\int -\log \mo_{\zeta(c)}((g^{-n}\zeta)(c))\,d\nu = n\tau \sigma - n\tau \int\varphi \,d\nu \pm K
.\] 
\end{lemma}
\begin{proof}
Using Lemma \ref{lem:normalized}, for $\nu$-a.e.\ $c \in GX_0$ we calculate that
\begin{IEEEeqnarray*}{rCl}
	-\log m^0_{\zeta(c)}((g^{-n}\zeta)(c)) &=& -\log
	\int_{g_{-n\tau}(\zeta(g_{n\tau}c))}e^{C(c',c;\,\varphi-\sigma)}\,d\mu_{W^\mathrm{uu}(c)}(c') \\
	&& + \log\int_{\zeta(c)}e^{C(c',c;\,\varphi-\sigma)}\,d\mu_{W^\mathrm{uu}(c)}(c')\pm K\\
	&=& -\log \int_{\zeta(g_{n\tau}c)}e^{C(g_{-n\tau}c',c;\,\varphi-\sigma)}
	\,d(g_{n\tau})_*\mu_{W^\mathrm{uu}(c)}(c')\\
	&&-\,G(c) \pm K. 
\end{IEEEeqnarray*}
Using Lemma \ref{lem:flow2}, we have
\[
d(g_{n\tau})_*\mu_{W^\mathrm{uu}(c)}(c')
= k^{\pm}e^{\int_{-n\tau}^{0}(\varphi(g_tc')-\sigma)\,dt} \,d\mu_{W^{\mathrm{uu}}(g_{n\tau}c)}(c')
.\] 
We also have
\[
	C(g_{-n\tau}c',c;\varphi-\sigma)
	= C(c',g_{n\tau}c;\varphi-\sigma) - \int_{-n\tau}^{0}(\varphi(g_tc')-\sigma) \,dt
	+\int_{0}^{n\tau}(\varphi(g_tc)-\sigma)\,dt
.\] 
The first statement follows from combining these calculations.

For the second statement, since $m^0_{\zeta(c)}((g^{-n}\zeta)(c)) \leq k$ whenever
$\mo_{\zeta(c)}$ is defined, it follows that for $\nu$-a.e.\ $c \in GX_0$ we have
\[
	G(g^nc)-G(c) \geq -\log k - n\tau\sigma + \int_{0}^{n\tau}\varphi(g_tc)\,dt - K,
\]
and hence the negative part of $h(c) \coloneqq G(g^nc)-G(c)$ is integrable with respect to $\nu$.
By \cite[Lemme 8]{OP04}, it follows that $h$ is integrable with $\int h\,d\nu = 0$.
\end{proof}

\begin{lemma}\label{lem:almostequal}
Suppose $\zeta$ is a measurable partition of $GX_0$ which is 
$m$-subordinated to $\mathcal{W}^{\mathrm{uu}}$, and 
let $(m_A)_{A \in \zeta}$ denote a disintegration of $m$ with respect to $\zeta$. Then $m_{\zeta(c)} = k^{\pm}\mo_{\zeta(c)}$ for $m$-a.e.\ $c \in GX_0$.
\end{lemma}
\begin{proof}
Let $m'$ be the Radon measure on $GX_0$ defined by the functional equation
\[
	m'(f) = \int \int f|_{\zeta(c)}\,d\mo_{\zeta(c)}\,dm(c),
\] 
where $f$ belongs to the space of compactly supported continuous functions on $GX_0$. 
Proposition \ref{prop:disintegration2} says that $m = k^{\pm}m'$. Thus, $m$ and $m'$ are in the same measure class, and the function
\[
	\Psi(c) \coloneqq \frac{dm}{dm'} (c)
\] 
satisfies $\Psi(c) = k^{\pm}$ for $m$-a.e.\ $c \in GX_0$. We define $\hat{m}_A$ by $d\hat{m}_A \coloneqq \Psi|_A \,dm^0_A$.
We know that the families $(m_A)_{A\in\zeta}$ and $(\hat{m}_A)_{A\in\zeta}$
are both disintegrations of $m$ with respect to $\zeta$. Thus, by the uniqueness of disintegrations, for
$m$-a.e.\ $c \in GX_0$, it follows that $m_{\zeta(c)} = \hat{m}_{\zeta(c)} = k^\pm
	\mo_{\zeta(c)}$.
\end{proof}

\begin{lemma}\label{lem:entropy}
Let $\tau > 0$ be such that $m$ is ergodic for $g=g_\tau$. For any $n \geq 0$, we have
\[
	h_m(g^n) = n\tau\sigma - n\tau\int \varphi \,d m \pm K
,\]
where $K$ is uniform in $n$. \end{lemma}
\begin{proof}
Let $\zeta$ be a partition associated to $m$, $\tau$ by Proposition \ref{prop:op}.
By Lemma \ref{lem:almostequal} and the fact that $\zeta$ realizes the entropy for $m$
with respect to $g$, and hence also for $g^n$, we have
\begin{align*}
	h_m(g^n) = h_m(g^n, \zeta)
	&= \int -\log m_{\zeta(c)}((g^{-n}\zeta)(c))\,d m(c)\\
	&= \int -\log \mo_{\zeta(c)}((g^{-n}\zeta)(c))\,d m(c) \pm K.
\end{align*}
By Lemma \ref{lem:approx}, we see that
\[
	 \int -\log \mo_{\zeta(c)}((g^{-n}\zeta)(c))\,d m(c)
	 = n\tau\sigma - n\tau \int\varphi\,d m \pm K
.\] 
It follows that $h_m(g^n) = n\tau\sigma - n\tau\int \varphi \,d m \pm K
.$
\end{proof}
\begin{corollary}
The measure $m$ satisfies $h_m((g_t)_{t\in\RR})+\int\varphi\,dm = \sigma$.
\end{corollary}
\begin{proof}
For $g=g_\tau$, where $\tau > 0$ is chosen so that $g_\tau$ is ergodic for $m$,
by Lemma \ref{lem:entropy}, we have  $h_m(g^n) = n\tau\sigma - n\tau\int \varphi \,d m \pm K$ for any $n\geq0$. We divide by $n\tau$, use Abramov's formula, and take a limit as $n \to \infty$.
\end{proof}
We establish some key technical lemmas towards our proof that $m$ is the unique equilibrium state.
\begin{lemma}\label{lem:divergence}
Let $\mu$, $\mu'$ be two Borel probability measures on a metric space $(Z,d)$,
and suppose that $(\PPP_n)_{n\in\NN}$ is a sequence of countable Borel partitions of a Borel subset
$Y \subseteq Z$ with $\mu(Y) = 1$ such that, for any $y \in Y$,
\begin{enumerate}
	\item $\PPP_m(y) \subseteq \PPP_n(y)$ if $m \geq n$;
	\item the diameter of $\PPP_n(y)$ approaches $0$ as
		$n \to \infty$.
\end{enumerate}
For each $n$, define a (possibly infinite) Kullbeck-Leibler divergence
\[
	D_{\PPP_n}(\mu \Vert \mu') \coloneqq \sum_{A \in \PPP_n}-\mu(A)\log \frac{\mu'(A)}{\mu(A)}
,\] 
where we interpret $\log 0 = -\infty$ and $0 \log \frac{0}{0} = 0 \log \infty = 0$.
Then $D_{\PPP_n}(\mu \Vert \mu')$ is nonnegative and increasing in $n$. Furthermore,
if $\mu \centernot{\ll} \mu'$, then 
\[
	\lim_{n\to\infty }D_{\PPP_n}(\mu \Vert \mu') = \infty.
\]
\end{lemma}
\begin{proof}
To see that $D_{\PPP_n}(\mu \Vert \mu') \geq 0$, we consider $(p_i)_{i\in\NN}$ and $(q_i)_{i\in\NN}$ with
$p_i, q_i \in [0,1]$ for all $i \in \NN$  such that $\sum_i p_i =1$ and $\sum_i q_i \leq 1$.
We compute that
\[
	\sum_i-p_i\log\frac{q_i}{p_i} =
	\sum_{p_i \neq 0}-p_i\log\frac{q_i}{p_i} 
	\geq \sum_{p_i\neq 0} p_i\left(1-\frac{q_i}{p_i}\right) = \sum_i p_i - \sum_{p_i\neq 0} q_i
	\geq 1-1 = 0.
\] 
To show that $D_{\PPP_n}(\mu \Vert \mu')$  increases with $n$, take two sequences
$(p_{i})_{i\in\NN}$ and $(q_{i})_{i\in\NN}$  as above and suppose that we have
$(p_{ij})_{i,j\in\NN}$ and $(q_{ij})_{i,j\in\NN}$  with $p_{ij}, q_{ij} \geq 0$ for all $i,j \in \NN$ which satisfy
\[
\sum_{j}p_{ij} = p_i, \quad \sum_{j}q_{ij} = q_i
.\] 
Then we have
\begin{align*}
	\sum_{ij}-p_{ij}\log\frac{q_{ij}}{p_{ij}} &= \sum_{ij}-p_{ij}\log
	\frac{(\tfrac{q_{ij}}{q_i})}{(\tfrac{p_{ij}}{p_i})} + \sum_{ij}-p_{ij} \log\frac{q_i}{p_i} \\
	&= \sum_{i}p_i
	\sum_j-\left(\frac{p_{ij}}{p_{i}}\right)\log\frac{(\tfrac{q_{ij}}{q_i})}{(\tfrac{p_{ij}}{p_i})}
	+\sum_i -p_i\log \frac{q_i}{p_i}.
\end{align*}
The presence of zeros does not affect the above calculation, since $p_{ij} \neq 0$ implies $p_i
\neq 0$, and if some $q_{ij}=0$ while $p_{ij}\neq 0$, then both sides of the above equality are infinity.
The first term above is a nonnegative combination of Kullbeck-Leibler divergences,
which we have already shown to be nonnegative, so we have
\[
\sum_{ij}-p_{ij}\log\frac{q_{ij}}{p_{ij}} \geq \sum_i -p_i \log \frac{q_i}{p_i}
.\]
This fact, combined with the fact that $\PPP_m(y) \subseteq \PPP_n(y)$ if $m \geq n$, proves that
$D_{\PPP_n}(\mu \Vert \mu')$ is increasing with $n$.

To prove the last statement, we start with an elementary bound. Suppose that
\[
\sum_i-p_i \log \frac{q_i}{p_i} \leq C
.\] 
Then for any $K > 0$ we can write
\begin{align*}
	\sum_i -p_i\log\frac{q_i}{p_i} &=
	\sum_{p_i \leq Kq_i}-p_i\log\frac{q_i}{p_i} + \sum_{p_i > Kq_i}-p_i\log\frac{q_i}{p_i}\\
	&\geq \sum_{p_i\leq Kq_i}p_i\left(1-\frac{q_i}{p_i}\right)
	+ \sum_{p_i > Kq_i}p_i\log K \\
	&\geq -1 + \log K \sum_{p_i > Kq_i}p_i,
\end{align*}
and hence
\[
\sum_{p_i > Kq_i}p_i \leq \frac{C+1}{\log K}
.\] 

We show that if $D_{\PPP_n}(\mu \Vert \mu')$ is bounded above as $n \to \infty$,
then $\mu \ll \mu'$. Suppose that $C$ is an upper bound for $D_{\PPP_n}(\mu \Vert \mu')$,
let $\epsilon > 0$ be given, and choose $K$ so that $\frac{C+1}{\log K} < \epsilon$. Then, letting
\[
G_{K,n} \coloneqq \{y \in Y \colon \mu(\PPP_n(y)) \leq K\mu'(\PPP_n(y))\}
,\] 
we have $\mu(G_{K,n}) \geq 1-\epsilon$ for each $n \in \NN$. Consider the set 
\[
	G_K = \limsup_{n\to\infty} G_{K,n} = \{y \in Y : \mu(\PPP_n(y)) \leq K\mu'(\PPP_n(y))
		\text{ for infinitely many $n$}\}
,\]
which also satisfies $\mu(G_K) \geq 1-\epsilon$.
We show that $\mu \leq K \mu'$ on $G_K$. Once this is established, we can show that
$\mu \ll \mu'$ as follows: if $B \subseteq Z$ is any Borel set such that $\mu'(B) = 0$, 
then $\mu(B \cap G_K) \leq K \mu'(B \cap G_K) = 0$, and 
thus $\mu(B) \leq \mu(B\cap G_K) + \mu(G^c_K) \leq \epsilon$.
Since $\epsilon > 0$ was arbitrary, then we must have $\mu(B) = 0$.

Let $F \subseteq G_K$ be any Borel subset, and suppose that $\delta > 0$. 
Choose an open set $U \supseteq F$ such that $\mu'(U) \leq \mu'(F) + \delta$.
For each $y \in F$ we define
\[
	\mathcal{Q}(y) \coloneqq \PPP_{n(y)}(y)
,\] 
where $n(y)$ is smallest such that both $\PPP_{n(y)}(y) \subseteq U$ and $\mu(\PPP_{n(y)}(y)) \leq
K\mu'(\PPP_{n(y)}(y))$. By hypothesis (2) of the lemma, and since $y \in G_K$,
then $n(y)$ is finite for every $y \in F$. Hence
$\mathcal{Q}(y)$ is well defined for every $y \in F$.

We check that $\mathcal{Q}$ is a partition. Let $y' \in \mathcal{Q}(y)$.
Since $\PPP_{n(y)}(y') = \PPP_{n(y)}(y)$, we have $n(y') \leq n(y)$. On the other hand, for any 
$m < n(y)$, we have $y' \in \PPP_{n(y)}(y) \subseteq \PPP_m(y)$ since $\PPP_m$ is coarser 
than $\PPP_{n(y)}$, and hence $\PPP_m(y) = \PPP_m(y')$. Therefore if $n(y') < n(y)$,
then $\PPP_{n(y')}(y) = \PPP_{n(y')}(y')$, and hence $n(y) \leq n(y')$, which is a contradiction.
Thus $n(y) = n(y')$, which shows that $\mathcal{Q}(y) = \mathcal{Q}(y')$.

Since $\mathcal{Q}$ is a partition of a subset of $U$ which contains
$F$, and since $\mathcal{Q}$ is clearly countable and consists of Borel subsets, we have
\[
	\mu(F) \leq \sum_{A \in \mathcal{Q}} \mu(A) \leq \sum_{A \in \mathcal{Q}}K \mu'(A) \leq K \mu'(U)
	\leq K(\mu'(F)+\delta)
.\] 
Since $\delta > 0$ was arbitrary, then we have proved $\mu(F) \leq K\mu'(F)$.
\end{proof}
Suppose that $\nu$ is an ergodic $g_t$-invariant probability measure on $GX_0$.
Let $\tau > 0$ be chosen so that $g = g_\tau$ is ergodic for $\nu$,
and let $\zeta$ be a partition guaranteed by Proposition
\ref{prop:op} for $\nu$ and $\tau$. Recall that $(\nu_A)_{A \in \zeta}$ denotes a disintegration of $\nu$ with respect to $\zeta$. For any $n \geq 1$, we define the function
\[
	\psi_n(c) \coloneqq
	\begin{cases}
	\frac{\mo_{\zeta(c)}((g^{-n}\zeta)(c))}
	{\nu_{\zeta(c)}((g^{-n}\zeta)(c))} & \nu_{\zeta(c)}((g^{-n}\zeta)(c)) \neq 0, \\
	+\infty & \text{otherwise.}
	\end{cases}
\] 
Note that $\psi_n(c)$ is well-defined for $\nu$-a.e.\ $c \in GX_0$.
We define 
\[
	F_n(c) \coloneqq \int -\log \psi_n(c') \,d\nu_{\zeta(c)}(c')
,\]
and we observe that $F_n$ and $\psi_n$ are related by 
\[
\int F_n\,d\nu = \int -\log\psi_n \,d\nu.
\]
\begin{lemma} \label{Fncincreases}
Suppose $\nu \neq m$. Then $F_n(c)$ increases to infinity for $\nu$-a.e.\ $c \in GX_0$.
\end{lemma}
\begin{proof}
Since $\nu \neq m$, we can choose a continuous function with compact support $f$ such that $\int f\,d\nu
\neq \int f\, d m$. Let
\[
	A_{f,m} = \left\{c \in GX_0 \colon
	\lim_{t\to+\infty}\frac{1}{t}\int_0^t f(g_sc)\,ds = \int f\,dm\right\}
,\] 
and similarly for $A_{f,\nu}$. Note that $m(A_{f,m}) = 1$ by Birkhoff's ergodic theorem.
Using the Hopf argument, $A_{f,m}$ must be saturated with respect to the weak stable foliation,
and using the quasi-product description provided by Proposition \ref{prop:disintegration}, we have that 
$\mu_{\su(c)}(W^{\mathrm{uu}}(c)\setminus A_{f,m}) = 0$
for \emph{every} $c \in GX_0$. Since $m^0_{\zeta(c)}$ is absolutely continuous with respect
to $\mu_{\su(c)}|_{\zeta(c)}$ by definition, then $m^0_{\zeta(c)}(\zeta(c)\setminus A_{f,m}) = 0$
any time that $m^0_{\zeta(c)}$ is well-defined. On the other hand, since $\nu(A_{f,\nu}) = 1$,
then $\nu_{\zeta(c)}(\zeta(c)\setminus A_{f,\nu}) = 0$ for $\nu$-a.e.\ $c \in GX_0$.
Since $A_{f,\nu}$ and $A_{f,m}$ are disjoint sets, then
we conclude that $m^0_{\zeta(c)}$ and $\nu_{\zeta(c)}$
are mutually singular for $\nu$-a.e.\ $c \in GX_0$.

Since $\zeta$ is $\nu$-subordinated to $\mathcal{W}^{\mathrm{uu}}$, then for any $n \in \NN$,
the partition $g^{-n}\zeta$ is also $\nu$-subordinated to $\mathcal{W}^{\mathrm{uu}}$. This implies that,
for $\nu$-a.e.\ $c \in GX_0$, the set $\zeta(c)$ is a relatively compact neighborhood of $c$ inside
$\wuu(c)$, and for any $n \in \NN$ and
$\nu_{\zeta(c)}$-a.e.\ $c' \in \zeta(c)$, the set $(g^{-n}\zeta)(c')$ contains an open set inside of
$\zeta(c)$. It follows that the partition $(g^{-n}\zeta)|_{\zeta(c)}$ 
must be a countable partition of a subset of $\zeta(c)$
with full $\nu_{\zeta(c)}$-measure; otherwise
there would be uncountably many disjoint open subsets of $\zeta(c)$. This is impossible because
$\zeta(c)$ inherits the subspace topology from $GX_0$, which is a separable
metric space.

Proposition \ref{prop:op} implies that,
for $\nu$-a.e.\ $c \in GX_0$, we have $\diam((g^{-n}\zeta)(c'))\to 0$
for $\nu_{\zeta(c)}$-a.e.\ $c' \in \zeta(c)$.
Let $Y_c \subseteq \zeta(c)$ be a full
$\nu_{\zeta(c)}$-measure set such that $(g^{-n}\zeta)|_{Y_c}$ is countable for
every $n\in \NN$ and $\mathrm{diam}((g^{-n}\zeta)(c'))\to 0$ for every $c' \in Y_c$.
Then, for $\nu$-a.e.\ $c \in GX_0$, we can apply Lemma \ref{lem:divergence}
to the measures $\mu = \nu_{\zeta(c)}$ and 
$\mu' = \frac{\mo_{\zeta(c)}}{\lVert\mo_{\zeta(c)}\rVert}$, using the partitions
$\mathcal{P}_n = (g^{-n}\zeta)|_{Y_c}$. For $\nu$-a.e.\ $c \in GX_0$, 
we find that, since $\nu_{\zeta(c)}$ and $\mo_{\zeta(c)}$ are
mutually singular, then
\[
	D_{\zeta_n}(\mu \Vert \mu') =
	\int -\log \frac{m^0_{\zeta(c)}((g^{-n}\zeta)(c'))/\lVert m^0_{\zeta(c)}\rVert}
	{\nu_{\zeta(c)}((g^{-n}\zeta)(c'))}\,d\nu_{\zeta(c)}(c')
\] 
increases to $\infty$. Since $\lVert \mo_{\zeta(c)}\rVert$ is constant in $n$,
then $F_n(c)$ also increases to $\infty$.
\end{proof}
We are now ready to show that $m$ is the unique equilibrium state on $GX_0$ for $\varphi$.
\begin{lemma}\label{lem:upperbound}
Let $\nu$ be an ergodic probability measure. If $\nu \neq m$, then
\[
	h_\nu((g_t)_{t\in\RR}) + \int \varphi \,d \nu < \sigma
.\] 
\end{lemma}
\begin{proof}
Using Lemma \ref{lem:approx}, we have
\begin{align*}
	\int -\log \psi_n\,d \nu &= \int -\log \mo_{\zeta(c)}((g^{-n}\zeta)(c))\,d \nu(c)
	- \int -\log \nu_{\zeta(c)}((g^{-n}\zeta)(c))\,d \nu(c) \\
	&= \left(n\tau\sigma - n\tau \int \varphi\,d \nu \pm K\right) - h_\nu(g^n).
\end{align*}
In particular, using Abramov's formula, we have 
\begin{equation} \label{eq:contradictionkeystep}
\int -\log \psi_n\,d \nu \leq n\tau \left (\sigma - \left (h_\nu((g_t)_{t\in\RR}) + \int\varphi\,d\nu \right ) \right ) + K.
\end{equation}
Suppose that $\nu \neq m$.  Recall that $\int -\log\psi_n \,d\nu = \int F_n(c)\,d\nu(c)$.
By Lemma \ref{Fncincreases}, for $\nu$-a.e.\ $c \in GX_0$, $F_n(c)$ increases to infinity. By the monotone convergence theorem, it follows that 
\[
	\lim_{n\to\infty}\int -\log \psi_n\,d\nu= \lim_{n\to\infty} \int F_n(c)\, d\nu(c) = \infty.
\] 
Thus, \eqref{eq:contradictionkeystep} is only possible if we have $h_\nu((g_t)_{t\in\RR}) + \int\varphi\,d\nu < \sigma$,
proving the lemma and completing the proof of Theorem \ref{thmx:existunique}.
\end{proof}
\subsection*{Acknowledgements} We thank the anonymous referees for insightful and detailed comments which
benefited the present work. We also thank Nicola Cavallucci for helpful comments on the dimension theory
of the limit set. 
\bibliographystyle{siam} 
\bibliography{biblioPSconstruction-final2025}

\end{document}